\documentclass[a4paper,leqno]{amsart}

\usepackage{amsmath}
\usepackage{amssymb}
\usepackage{amsfonts}
\usepackage{amsthm}
\usepackage{amsbsy}

\usepackage{enumerate}
\usepackage{hyperref}

\usepackage[dvips]{graphicx}
\usepackage{xcolor}
\usepackage{todonotes}
\usepackage{cancel}

\theoremstyle{plain}
\begingroup
\newtheorem{teo}{Theorem}[section]
\newtheorem{prop}[teo]{Proposition}
\newtheorem{cor}[teo]{Corollary}
\newtheorem{lem}[teo]{Lemma}
\endgroup

\theoremstyle{definition}
\begingroup
\newtheorem{defin}[teo]{Definition}
\newtheorem{oss}[teo]{Remark}
\newtheorem{exam}[teo]{Example}

\endgroup

\DeclareMathOperator*{\argmin}{arg\,min}

\renewcommand{\eqref}[1]{\textnormal{(\ref{#1})}}

\numberwithin{equation}{section}

\newcommand{\rme}{\mathrm{e}}
\newcommand{\rmd}{\mathrm{d}}

\title[Multiscale for Image Registration and Inverse Problems]{A Multiscale Theory for Image Registration and Nonlinear Inverse Problems}

\author[K.~Modin]{Klas Modin}
\author[A.~Nachman]{Adrian Nachman}
\author[L.~Rondi]{Luca Rondi}

\address[K.~Modin]{Mathematical Sciences, Chalmers and University of Gothenburg, Sweden}
\email{klas.modin@chalmers.se}
\address[A.~Nachman]{Department of Mathematics, University of Toronto, Canada}
\email{nachman@math.toronto.edu}
\address[L.~Rondi]{Dipartimento di Matematica, Universit\`a di Milano, Italy}
\email{luca.rondi@unimi.it}

\begin{document}

\setcounter{section}{0}
\setcounter{secnumdepth}{2}

\begin{abstract}In an influential paper, Tadmor, Nezzar and Vese (Multiscale Model. Simul. (2004)) introduced a hierarchical decomposition of an image as a sum of constituents of different scales. Here we construct analogous hierarchical expansions for diffeomorphisms, in the context of image registration, with the sum replaced by composition of maps. We treat this as a special case of a general framework for multiscale decompositions, applicable to a wide range of imaging and nonlinear inverse problems. 
As a paradigmatic example of the latter, we consider the Calder\'on inverse conductivity problem. We prove that we can simultaneously perform a numerical reconstruction and a multiscale decomposition of the unknown conductivity, driven by the inverse problem itself. 
We provide novel convergence proofs which work in the general abstract settings, yet are 
sharp enough to prove that the hierarchical decomposition of Tadmor,
Nezzar and Vese converges for arbitrary functions in $L^2$, a problem left open
in their paper.
We also give counterexamples that show the optimality of our general results.

\medskip


\medskip

\end{abstract}
\maketitle
\tableofcontents

\section{Introduction}\label{introsec}

In a beautiful and influential paper, Tadmor, Nezzar and Vese \cite{T-N-V} introduced a multiscale hierarchical representation of an image, and proved corresponding convergence and energy decomposition results.
 Their starting point is the Rudin-Osher-Fatemi model for image restoration: given a (possibly noisy) image $f\in L^2(\mathbb{R}^2)$
 and a positive parameter $\lambda_0$, one seeks the solution $u_0$ of 
\begin{equation}\label{firstminpbm0}
\min \left\{\lambda_0\|f-u\|^2_{L^2(\mathbb{R}^2)}+\|u\|_{BV (\mathbb{R}^2)}:\ u\in L^2(\mathbb{R}^2)\right\}.
\end{equation}
Here $BV (\mathbb{R}^2)$ is the homogeneous $BV(\mathbb{R}^2)$ space and for any $u\in BV (\mathbb{R}^2)$, the norm $\|u\|_{BV (\mathbb{R}^2)}$ denotes the total variation of its distributional gradient $Du$; see the beginning of Subsection~\ref{L2BVsec} for details.
The variational problem \eqref{firstminpbm0} is uniquely solvable, and yields a decomposition of $f$ as $f=u_0+v_0$, where $v_0$ is the residual. The solution $u_0$ is expected to keep the most relevant features of the image while the residual $v_0$ contains the noisy part. The fidelity parameter  $\lambda_0$ determines the amount of features preserved and the noise at that scale. Indeed, for higher $\lambda_0$ the solution $u_0$ is closer to $f$ and less noise is removed. The idea in \cite{T-N-V} is to start with a relatively low value of $\lambda_0$ and then iterate the procedure by replacing $\lambda_0$ with a larger parameter $\lambda_1$ and $f$ with $v_0$. Then $f=u_0+u_1+v_1. $ Continuing in this manner, given an increasing sequence of positive parameters $\lambda_n$, $n=0,1,2,\ldots$, for any $j\in \mathbb{N}$ one obtains
$$f=u_0+u_1+\ldots+u_j+v_j.$$
If  $v_j$ converges to $0$ as $j$ goes to infinity, then this method provides a multiscale representation $f=\sum_{j=0}^{\infty}u_j$ of the image $f$.
The result proved in \cite[Theorem~2.2]{T-N-V} is the following:
\begin{teo}[\textbf{Tadmor-Nezzar-Vese}]
\label{hierarchythmTNV}
Let $f\in BV(\mathbb{R}^2)$. Let $\lambda_n=\lambda_02^n$ for some $\lambda_0>0$ and any $n\in\mathbb{N}$.
Then $f$ admits the following $(BV,L^2)$ hierarchical decomposition\textnormal{:}
\begin{equation}\label{hierarchdecomp}
f=\sum_{j=0}^{+\infty}u_j
\end{equation}
where the convergence is in the strong sense in $L^2(\mathbb{R}^2)$. Furthermore, the following energy identity holds\textnormal{:}
\begin{equation}\label{energyeq}
\|f\|_{L^2(\mathbb{R}^2)}^2=\sum_{j=0}^{+\infty}\left[\frac{1}{\lambda_j}\|u_j\|_{BV(\mathbb{R}^2)}+\|u_j\|_{L^2(\mathbb{R}^2)}^2\right].
\end{equation}
\end{teo}

This result was extended in \cite[Corollary~2.3]{T-N-V} to $f$ belonging to a class of intermediate spaces between $L^2$ and $BV$. The question of whether it holds for any $f\in L^2$ was left open.\footnote{We wish to thank the anonymous referee for bringing to
our attention reference \cite{Jaw-Mil} where such a result was announced in the
context of a more general interpolation theory approach. A complete proof
does not appear to have been published.}
 We will show in Theorem~\ref{ourhierarchythm} that the above theorem extends to arbitrary $f\in L^2$, as a special case of a more general result, Theorem~\ref{mainthm0}.

One of the main aims of our paper is to construct analogous hierarchical expansions for \emph{diffeomorphisms} in the context of image registration, with the sum above being replaced by \emph{composition} of maps. In image registration, one seeks an optimal diffeomorphism between two given images. This is an important problem in medical imaging, when one needs to align two images obtained at different times or with different instrumentation by transforming one to the other. Mathematically, given two images $I_0$ and $I_1$ as $L^2$-functions on a domain $\Omega$, one wishes to find a diffeomorphism $g$ of $\Omega$ which solves the minimization problem: 
\begin{equation}\label{minpbmquater0}
\min\{\|I_0\circ g^{-1} -I_1\|_{L^2(\Omega)}:\ g\in G_{\mathcal{H}}\}.
\end{equation}Here $G_{\mathcal{H}}$ is a Banach manifold of diffeomorphisms (depending on a choice of Hilbert space $\mathcal{H}$) which will be defined in  Subsection ~\ref{LDDMMintrosub}. This problem is sometimes referred to as ``greedy matching". The standard approach to its solution is via a gradient flow. This approach often leads to serious difficulties, both theoretically and practically. (See ~\cite{You} where these issues are explained in detail). The Large Deformation Diffeomorphic Metric Mapping (LDDMM) theory of image registration provides a beautiful geometric regularization of  \eqref{minpbmquater0} by introducing a Riemannian metric on  $G_{\mathcal{H}}$ and penalizing the geodesic distance of the diffeomorphism $g$ to the identity map. This will be described in more in detail in Subsection ~\ref{LDDMMintrosub}. Starting from this regularized problem, we develop a geometric multiscale framework and use it to prove that if a solution to \eqref{minpbmquater0} exists then it has a hierarchical expansion (analogous to  \eqref{hierarchdecomp}) as the composition of infinitely many deformations of increasingly finer scales.   The multiscale approach to image registration developed in this paper can thus be seen as a bridge between greedy matching and LDDMM: the multiscale construction consists of a series of ``LDDMM steps'', and we show that it yields a convergent decomposition of an optimal solution, provided a solution to \eqref{minpbmquater0} exists. In this context, an optimal solution 
is one with minimal distance to the identity. 
In Subsection ~\ref{LDDMMintrosub} below we briefly describe the LDDMM approach to image registration and illustrate our main results in this direction, in particular Theorem~\ref{diffeoconv2} which we believe to be the best result of the paper. The detailed proofs and further results are given in Section ~\ref{registrationsec}. For other multiscale approaches to image registration, completely different from the one presented here, we refer the reader to \cite{Risser-et-al0,Risser-et-al}, \cite{Sommer1,Sommer2}, \cite{Paquin-Levy-Schreibmann-Xing,Paquin-Levy-Xing1,Paquin-Levy-Xing2}, \cite {AXRNG}
and \cite{Gris-et-al}.

 A second aim of our paper is to develop an analogous multiscale framework suitable for nonlinear inverse problems. To illustrate the main ideas in this direction, we focus on one particular inverse problem which has been extensively investigated, namely the Calder\'on inverse conductivity problem. Initially motivated by geophysical prospection, and more recently by medical imaging, this concerns the determination of the conductivity $\sigma$ of a body $\Omega$ from voltage and current measurements at the surface $\partial \Omega$. In particular, here we allow for possibly nonsmooth conductivities for which uniqueness results may not be available. The given data is encoded in a nonlinear operator $\mathcal{N}(\sigma)$, the Neumann-to-Dirichlet (or current-to-voltage) map on $\partial \Omega$. 

Specifically, for a given input current $i$ on the boundary (assumed to have zero mean), $\mathcal{N}(\sigma)(i)$ is defined as: $$\mathcal{N}(\sigma)(i)=v|_{\partial\Omega}$$
with the potential $v$ the solution to
 the electrostatic boundary value-problem:\begin{equation}\label{Neumannproblem0}
\left\{\begin{array}{ll}
\mathrm{div}(\sigma\nabla v)=0&\text{in }\Omega\\
\sigma\nabla v\cdot \nu=i&\text{on }\partial\Omega\\
\int_{\partial \Omega}v=0.&
\end{array}\right. 
\end{equation}
Typical of many inverse problems, the problem of determining $\sigma$ from knowledge of $\mathcal{N}(\sigma)$ is severely ill-posed. To overcome this difficulty 
a regularization method is often used. Our multiscale procedure starts with a relatively well-posed problem, corresponding to a low value of the regularization parameter $\lambda$, that allows us to recover stably the main features of the unknown conductivity. Subsequent steps involve higher values of $\lambda$ to recover finer details; we then have to deal with ill-posed problems but have a very good initial guess at every stage. (In this respect, our method shares some of the advantages of homotopy continuation algorithms for inverse problems which rely, for instance, on multifrequency data; see e.g. \cite{Chen}). The resulting iterative procedure allows us to numerically solve the inverse problem and simultaneously obtain a multiscale representation of its solution. Significantly, this multiscale representation is driven by the inverse problem itself rather than some post-processing of the solution. We present our main multiscale results for the Calder\'on problem in Subsection ~\ref{Calderonsub}, in particular Theorem ~\ref{mainthmcond20}, along with a brief review of the relevant background. The proofs and further results are given in Section~\ref{inversesec}.

The above apparently completely different problems (LDDMM and Calder\'on) from two distinct fields led us to a general multiscale theory relevant to a wide range of applications that involve the minimization of the sum of a fidelity term and a regularization term. Our general abstract results for nonlinear inverse problems are introduced in Subsections~\ref{abstract1subsec} and \ref{abstractsecondsub} with the details and proofs given in Section~\ref{abstractsec}. The extension of the general framework, replacing addition by other group actions so as to be able to handle composition of maps, is developed in Section~\ref{topgroupsubs}.

Finally, in Appendix~\ref{counterexsub} we give several counterexamples showing the optimality of our abstract results.

\section{Background and main results}\label{mainresults}

We begin with a simple general formulation motivated by nonlinear inverse problems. This will serve to introduce some of the main ideas, and it will already provide a result sufficiently sharp to include the extension of Theorem~\ref{hierarchythmTNV} to $f\in L^2$. Subsequently, to obtain the convergence properties we seek, we will need to introduce a tighter multiscale algorithm.

\subsection{A multiscale framework for nonlinear inverse problems (first version)}\label{abstract1subsec}
Let $X$ be a real Banach space with norm $\|\cdot\|=\|\cdot\|_X$. Let $E$ be a
closed nonempty subset of $X$.

Let $Y$ be a metric space with distance $d=d_Y$. Let $\mathcal{N}:E\to Y$ be a possibly nonlinear map and let
$\hat{\mathcal{N}}\in Y$.
 (We think of $\hat{\mathcal{N}}$ as the given data). We assume that the function $E\ni \sigma\mapsto d(\hat{\mathcal{N}},\mathcal{N}(\sigma))$  is continuous with respect to the (strong) convergence in $X$.

\newcounter{bvassum}

We also assume that there exists a function $|\cdot|:X\to [0,+\infty]$ such that:
\begin{enumerate}[1)]
\item\label{BVass1} $|0|=0$ and $|-x|=|x|$ for any $x\in X$; 
\item\label{BVass2} $|x+y|\leq |x|+|y|$ for any $x$, $y\in X$;
\item\label{BVass3} $\{x\in E:\ |x|<+\infty\}$ is dense in $E$.
\setcounter{bvassum}{\theenumi}
\end{enumerate}

We note that condition \ref{BVass3}) is satisfied if $\{x\in X:\ |x|<+\infty\}$ is dense in $X$ and $E$ is a suitable closed subset of $X$, for example $E$ is the closure of an open set.

We fix two positive constants
 $\alpha$ and $\beta$ and we assume that the following regularized minimization problem admits a solution for $\hat{\sigma}=0$ and for any $\hat{\sigma}\in E$ and any $\lambda>0$
\begin{equation}\label{existenceminimizereq}
\min\left\{\left(\lambda[d(\hat{\mathcal{N}},\mathcal{N}(\hat{\sigma}+\sigma))^{\alpha}]+|\sigma|^{\beta}\right):\ \hat{\sigma}+\sigma\in E\right\}.
\end{equation}

Let us note here that existence of a solution to \eqref{existenceminimizereq}
may be guaranteed, under the above assumptions, if $|\cdot|$ also satisfies the following:
\begin{enumerate}[1)]
\setcounter{enumi}{\thebvassum}
\item\label{BVass4} $\{x\in X:\ |x|\leq b\}$ is relatively sequentially compact in $X$ for any $b\in \mathbb{R}$;
\item\label{BVass5} $|\cdot|$ is sequentially lower semicontinuous on $X$, with respect to the (strong) convergence in $X$.
\end{enumerate}

Inspired by the procedure in \cite{T-N-V}, we consider the following construction.
Let us fix positive parameters $\lambda_n$, $n\in\mathbb{N}$, and 
let $\sigma_0$ be a solution to
\begin{equation}\label{regularizedpbm}
\min\left\{\left(\lambda_0[d(\hat{\mathcal{N}},\mathcal{N}(\sigma))^{\alpha}]+|\sigma|^{\beta}\right):\ \sigma\in E\right\}.
\end{equation}
The multiscale algorithm then constructs $\sigma_n$, $n\geq 1$, inductively as a solution to
 \begin{equation}\label{inductiveconstr}\min\left\{\left(\lambda_n[d(\hat{\mathcal{N}},\mathcal{N}(\tilde{\sigma}_{n-1}+\sigma))^{\alpha}]+|\sigma|^{\beta}\right):\ \tilde{\sigma}_{n-1}+\sigma\in E\right\}\end{equation}
where we denote
 by $\tilde{\sigma}_n$ the partial sum:
 \begin{equation}\label{partialsum1}
 \tilde{\sigma}_n=\sum_{j=0}^n\sigma_j\quad\text{for any }n\in\mathbb{N}.
 \end{equation}
 
Our assumptions guarantee that the sequence $\{\sigma_n\}_{n\in\mathbb{N}}$ exists, however in general it need not be uniquely determined.

Note that by taking $\sigma=0$, we have
$$\lambda_n[d(\hat{\mathcal{N}},\mathcal{N}(\tilde{\sigma}_n))^{\alpha}]+|\sigma_n|^{\beta}\leq \lambda_n[d(\hat{\mathcal{N}},\mathcal{N}(\tilde{\sigma}_{n-1}))^{\alpha}]+|0|^{\beta},$$
hence
\begin{equation}
d(\hat{\mathcal{N}},\mathcal{N}(\tilde{\sigma}_n))\leq d(\hat{\mathcal{N}},\mathcal{N}(\tilde{\sigma}_{n-1}))\quad\text{for any }n\geq 1.
\end{equation}

Let us denote
$$\varepsilon_0=\lim_nd(\hat{\mathcal{N}},\mathcal{N}(\tilde{\sigma}_n))$$
and
$$\delta_0=\inf\{d(\hat{\mathcal{N}},\mathcal{N}(\sigma)):\ \sigma\in E\}.$$Our first general result is the following. See Subsection~\ref{abstractresultproofsec} for a proof.

\begin{teo}\label{mainthm0}
Under the assumptions listed above, if\begin{equation}\label{firstcondcoeff}
\limsup_n \frac{2^{\beta n}}{\lambda_n}<+\infty,\end{equation}
then for the multiscale sequence $\{\tilde{\sigma}_n\}_{n\in\mathbb{N}}$ given by \eqref{regularizedpbm}, \eqref{inductiveconstr} and \eqref{partialsum1}
we have $\varepsilon_0=\delta_0$, that is: 
$$\lim_nd(\hat{\mathcal{N}},\mathcal{N}(\tilde{\sigma}_n))=\inf\{d(\hat{\mathcal{N}},\mathcal{N}(\sigma)):\ \sigma\in E\}.$$
\end{teo}

In particular, this result improves on Theorem~\ref{hierarchythmTNV}, namely we have the following.
\begin{teo}\label{ourhierarchythm}The conclusions of Theorem~\textnormal{\ref{hierarchythmTNV}} are valid for any $f \in L^2(\mathbb{R}^2)$.
\end{teo}
Specifically, to obtain Theorem~\ref{ourhierarchythm} we apply Theorem~\ref{mainthm0} to the following setting: let $X=E=L(\mathbb{R}^2)$,
$Y=L(\mathbb{R}^2)$, with the distance induced by the $L^2$ norm, $\mathcal{N}=\mathrm{Id}$ and $\hat{\mathcal{N}}=f\in L^2(\mathbb{R}^2)$. Also let $|\cdot|=\|\cdot\|_{BV(\mathbb{R}^2)}$, $\alpha=2$ and $\beta=1$. We refer to Subsection~\ref{L2BVsec} for further details. Indeed, in Theorems~\ref{hierarchythm} and \ref{hierarchythmbdd} we shall state and prove a generalization of Theorem~\ref{ourhierarchythm} to any dimension as well as to bounded and Lipschitz open subsets of $\mathbb{R}^N$, $N\geq 1$. We point out, however, that the proof of the energy equality \eqref{energyeq} is only valid for $N=2$.

To show the versatility of our abstract framework we list below some simple examples which satisfy the required assumptions. A much more elaborate and interesting application, to the inverse problem of Calder\'on, will be described in Subsection  ~\ref{Calderonsub}.\begin{exam}
Let $\Omega$ be an open and bounded subset of $\mathbb{R}^N$, $N\geq 1$, and assume that $\Omega$ has a Lipschitz boundary. Then our assumptions \ref{BVass1})---\ref{BVass5}) are verified in the following cases.
\begin{itemize}
\item $X=L^1(\Omega)$, with its usual norm, $E$ the closure of a nonempty open subset of $X$, and
$|u|=\|u\|_{BV(\Omega)}$ for any $u\in L^1(\Omega)$;
\item $X=L^2(\Omega)$, with its usual norm, $E$ the closure of a nonempty open subset of $X$, and $|u|=\|u\|_{W^{1,2}(\Omega)}=\|u\|_{L^2(\Omega)}+\|\nabla u\|_{L^2(\Omega)}$ for any $u\in L^2(\Omega)$;
\item $X=C^0(\overline{\Omega})$, with its usual sup norm, $E$ the closure of a nonempty open subset of $X$, and, for some $\alpha$, $0<\alpha\leq 1$,
$|u|=\|u\|_{C^{0,\alpha}(\Omega)}=\|u\|_{L^{\infty}(\Omega)}+|u|_{C^{0,\alpha}(\Omega)}$
for any $u\in C^0(\overline{\Omega})$. Here, as usual,
$$|u|_{C^{0,\alpha}(\Omega)}=\sup\left\{\frac{|u(x)-u(y)|}{\|x-y\|^{\alpha}}: x,\, y\in \Omega,\ x\neq y\right\}.$$
\end{itemize}
We note that $E$ need not to be the closure of an open subset, provided it satisfies the density assumption \ref{BVass3}).
\end{exam}

 We recall that for any bounded open set $\Omega\subset\mathbb{R}^N$, $N\geq 1$, a function $u\in L^1(\Omega)$ belongs to $BV(\Omega)$ if $Du$, its gradient in the distributional sense, is a bounded vector valued Radon measure on $\Omega$. We equip $BV(\Omega)$ with the usual norm
$$\|u\|_{BV(\Omega)}=\|u\|_{L^1(\Omega)}+|u|_{BV(\Omega)}$$
where the seminorm $|\cdot|_{BV(\Omega)}$ is defined as the total variation of $Du$ on $\Omega$ that is
$$|u|_{BV(\Omega)}=|Du|(\Omega).$$

\begin{oss}\label{iterativeTikhonov}
We point out that the algorithm \eqref{inductiveconstr}, as well as the one of Tadmor, Nezzar and Vese, is formally the same as the so-called Iterative Tikhonov Regularisation, see for instance the book \cite[Section~7.1]{Sch-et-al}. 
A crucial difference is in the choice of the sequence of parameters $\lambda_n$ as in \eqref{firstcondcoeff} (for Iterative Tikhonov we would have $\lambda_n\to 0$). This corresponds to the completely different objective of the multiscale approach and of the convergence results proved here.
\end{oss}

While Theorem ~\ref{mainthm0} shows that the construction in \eqref{regularizedpbm}, \eqref{inductiveconstr} and \eqref{partialsum1} yields a minimizing sequence $\{\mathcal{N}(\tilde{\sigma}_n)\}_{n\in\mathbb{N}}$, much of the work in the paper will be to go beyond this and also prove convergence results for $\{\tilde{\sigma}_n\}_{n\in\mathbb{N}}$ or one of its subsequences. This is of course automatic in the case of Theorem  ~\ref{ourhierarchythm} where $\mathcal{N}=\mathrm{Id}$, and also easy to show if  $\mathcal{N}(\tilde{\sigma})$ satisfies a mild coercivity condition (see Proposition~\ref{coerciveprop}), but for general nonlinear ill-posed problems it will become clear that we need a tighter multiscale construction. Such a construction is presented in the next subsection. 
\subsection{A tighter multiscale construction for nonlinear problems}\label{abstractsecondsub}

We keep the assumptions of the previous subsection, in particular we suppose that $|\cdot|$ satisfies assumptions \ref{BVass1})---\ref{BVass5}).
 We fix positive constants 
 $\alpha$, $\beta$, $\gamma$ and let $\lambda_n>0$ and $a_n\geq 0$ for any $n\in \mathbb{N}$. We assume that 
$a_n\leq a_{n-1}$ for any $n\geq 1$.

Let now $\sigma_0$ be a solution to
\begin{equation}\label{regularizedpbmbis}
\min\left\{\left(\lambda_0[d(\hat{\mathcal{N}},\mathcal{N}(\sigma))^{\alpha}+a_0|\sigma|^{\gamma}]+|\sigma|^{\beta}\right):\ \sigma\in E\right\}.
\end{equation}
Our assumptions guarantee that at least one minimizer $\sigma_0$ does exist. We then construct $\sigma_n$, $n\geq 1$, inductively by solving \begin{equation}\label{tightconstr}\min\left\{\left(\lambda_n[d(\hat{\mathcal{N}},\mathcal{N}(\tilde{\sigma}_{n-1}+\sigma))^{\alpha}+a_n|\tilde{\sigma}_{n-1}+\sigma|^{\gamma}]+|\sigma|^{\beta}\right): \tilde{\sigma}_{n-1}+\sigma\in E\right\},
\end{equation}
where for any $n\geq 1$ we denote as before \begin{equation}\label{tildetight}\tilde{\sigma}_{n-1}=\sum_{j=0}^{n-1}\sigma_j.\end{equation}Again our assumptions guarantee that the sequence $\{\sigma_n\}_{n\in\mathbb{N}}$ exists, however we cannot guarantee that it is uniquely determined. 

We point out that when $a_n$ is $0$ for all $n\in\mathbb{N}$, we are exactly in the case described in the previous subsection. On the other hand, for nonzero $a_n$ we not only penalize the value of $|\cdot|$ of the increment $\sigma_n$ but also that of the partial sum $\tilde{\sigma}_n$. By taking $\sigma=0$, one immediately finds that
\begin{equation}
d(\hat{\mathcal{N}},\mathcal{N}(\tilde{\sigma}_n))^{\alpha}\leq
d(\hat{\mathcal{N}},\mathcal{N}(\tilde{\sigma}_n))^{\alpha}+a_n|\tilde{\sigma}_n|^{\gamma}
\leq d(\hat{\mathcal{N}},\mathcal{N}(\tilde{\sigma}_{n-1}))^{\alpha}+a_{n-1}|\tilde{\sigma}_{n-1}|^{\gamma}
\end{equation}
for any $n\geq 1$.
Let
$$\varepsilon_0=\lim_n\left(d(\hat{\mathcal{N}},\mathcal{N}(\tilde{\sigma}_n))^{\alpha}+a_n|\tilde{\sigma}_n|^{\gamma}\right)^{1/\alpha}.
$$
Clearly we have that $\varepsilon_0\geq \delta_0=\inf\{d(\hat{\mathcal{N}},\mathcal{N}(\sigma)):\ \sigma\in E\}$.

 We first show that the conclusion of Theorem~\ref{mainthm0} still holds in this more general case.

 \begin{teo}\label{mainthmbis0}
We assume that 
\begin{equation}\label{secondcondcoeff}
a_n\leq a_{n-1}\text{ for any }n\geq 1,\quad
\lim_n a_n=0\quad\text{and}\quad\limsup_n \frac{2^{\beta n}}{\lambda_n}<+\infty.
\end{equation}
Then for the sequence $\{\tilde{\sigma}_n\}_{n\in\mathbb{N}}$ defined by \eqref{tildetight} from the sequence $\{\sigma_n\}_{n\in\mathbb{N}}$ obtained from \eqref{regularizedpbmbis} and \eqref{tightconstr} we have $\varepsilon_0= \delta_0$ and we also have
$$\lim_nd(\hat{\mathcal{N}},\mathcal{N}(\tilde{\sigma}_n))=\delta_0.$$
\end{teo}

For a proof we again refer to Subsection~\ref{abstractresultproofsec}. We now turn to the main point of the new construction, which is to find conditions for the convergence of $\{\tilde{\sigma}_n\}_{n\in\mathbb{N}}$.  We begin by observing that, if this sequence (or one of its subsequences) converges to some $\tilde{\sigma}_{\infty}$, then
$\tilde{\sigma}_{\infty}$ is a solution to 
\begin{equation}\label{minpbm0}
\min\{d(\hat{\mathcal{N}},\mathcal{N}(\sigma)):\ \sigma\in E\}.
\end{equation}Therefore, an immediate necessary condition for the convergence of $\{\tilde{\sigma}_n\}_{n\in\mathbb{N}}$, or of one of its subsequences,
is that a solution to \eqref{minpbm0} does exist. A sufficient condition is guaranteed by the following stronger assumption. Suppose that there exists $\hat{\sigma}\in E$ such that
\begin{equation}\label{anothermin00}
d(\hat{\mathcal{N}},\mathcal{N}(\hat{\sigma}))
=\delta_0=\min\{d(\hat{\mathcal{N}},\mathcal{N}(\sigma)):\ \sigma\in E\}\quad\text{and}\quad |\hat{\sigma}|<+\infty.
\end{equation} 
Without loss of generality we may then assume that $\hat{\sigma}$ solves the following minimization problem
\begin{equation}\label{minimal00}
\min\{ |\sigma|:\ \sigma\in E\text{ and }d(\hat{\mathcal{N}},\mathcal{N}(\sigma))=\delta_0\}<+\infty.
\end{equation}
This condition may seem rather restrictive. However, in Appendix~\ref{counterexsub}, we show through several examples that our general abstract results are optimal in several respects. In particular, the two cases given in Example~\ref{L2minexample} suggest that a condition such as  \eqref{anothermin00} might not be removed or even relaxed if we wish to have convergence of $\{\tilde{\sigma}_n\}_{n\in\mathbb{N}}$.  Let us call $\hat{E}$ the set of solutions of \eqref{minimal00}. We note that $\hat{E}$ is sequentially compact in $X$.

 In the result below we also need a stronger assumption on the parameters, namely
\begin{equation}\label{crucialcondition0}
a_n\leq a_{n-1}\text{ for any }n\geq 1,\quad
\lim_n a_n=0\quad\text{and}\quad
\limsup_n \frac{2^{\beta n}}{\lambda_na_n}=0.
\end{equation}
Note that in particular we are assuming $a_n>0$ for any $n\geq 0$ and that \eqref{crucialcondition0} implies that $\limsup_n 2^{\beta n}/\lambda_n=0$.

We have the following convergence result, which will be proved in Subsection~\ref{abstractresultproofsec}.

\begin{teo}\label{minimizercor0}
 Assume that
\eqref{crucialcondition0} holds 
and that there exists a solution $\hat{\sigma}$ of \eqref{minimal00}. Consider the sequence $\{\tilde{\sigma}_n\}_{n\in\mathbb{N}}$ defined by \eqref{tildetight} from the sequence $\{\sigma_n\}_{n\in\mathbb{N}}$ obtained from \eqref{regularizedpbmbis} and \eqref{tightconstr}.

Then $\tilde{\sigma}_n$ converges, up to a subsequence, to $\tilde{\sigma}_{\infty}$ where $\tilde{\sigma}_{\infty}$ is a \textnormal{(}possibly different from $\hat{\sigma}$\textnormal{)} solution to \eqref{minimal00}, that is, $d(\hat{\mathcal{N}},\mathcal{N}(\tilde{\sigma}_{\infty}))=\delta_0$ and $|\tilde{\sigma}_{\infty}|=|\hat{\sigma}|$.
Furthermore, we have that
$$\lim_n|\tilde{\sigma}_n|= |\hat{\sigma}|$$
and
\begin{equation}\label{distance}
\lim_n \mathrm{dist}(\tilde{\sigma}_n,\hat{E})=0
,\end{equation}
where for any $\sigma\in X$, $\mathrm{dist}(\sigma,\hat{E})=\inf\{\|\sigma-\hat{\sigma}\|:\ \hat{\sigma}\in \hat{E}\}$.
\end{teo}

It is still possible, however, that two different subsequences converge to two different solutions of \eqref{minimal00}, as the counterexample in Subsection~\ref{wholesequenceex} shows.

On the other hand, if \eqref{minimal00} has a unique solution $\hat{\sigma}$, then the above construction yields the multiscale decomposition $$\hat{\sigma}=\lim_n\tilde{\sigma}_n=\sum_{j=0}^{\infty}\sigma_j$$
in the sense of strong convergence in $X$.

The proof of Theorem ~\ref{minimizercor0} will be given in Subsection~\ref{abstractresultproofsec}  along with further details on the abstract multiscale framework. In the next subsection we will describe the application of this general framework to multiscale results for the inverse conductivity problem.
In the subsequent subsection we will introduce our multiscale approach to the registration problem.

\subsection{The inverse problem of Calder\'on}\label{Calderonsub}
The inverse problem proposed by Cal\-de\-r\'on in 1980 concerns the determination of the conductivity of an object from electrostatic measurements of current and voltage type at the boundary. 

In the case of scalar (i.e. isotropic) conductivities, uniqueness was proved, in dimension $3$ and higher, first in \cite{Koh e Vog84:1,Koh e Vog85} for the determination of the conductivity at the boundary and for the
 analytic case, then in \cite{Syl e Uhl87} for $C^2$ smooth conductivities. The two dimensional case was first solved in \cite{Nac} for conductivities in $W^{1,p}(\Omega)$ with $p>1$.

Recently these uniqueness results have been considerably sharpened. For $N=2$ uniqueness was proved for bounded conductivities, without any regularity assumptions, in \cite{Ast e Pai} and even for certain classes of conductivities which need not be bounded from above or below, in \cite{ALP2} and \cite{NRT}. For $N\geq 3$, uniqueness has been shown for $C^1$ conductivities, as well as for Lipschitz conductivities close to the identity in \cite{Hab-Tat}; this smallness condition was removed in \cite{Car-Rog}. For dimensions  $N=3,4$ uniqueness has been proved in \cite{Hab} for conductivities in $W^{1,N}(\Omega)$. 
In the case of anisotropic conductivities, since boundary measurements are invariant under suitable changes of coordinates that keep the boundary fixed, uniqueness does not hold.
However, at least in dimension $2$ for symmetric conductivity tensors, this is the only obstruction as shown first in \cite{Syl} and  \cite{Nac} in the smooth case and then in \cite{Ast e Pai e Las} in the general case.

We need some notation in order to describe the classes of conductivities we will be working with.
 Let us fix
positive constants $a$ and $b$, with $0<a\leq  b$. For $N\geq 2$, we call $\mathbb{M}^{N\times N}(\mathbb{R})$ the space of real valued $N\times N$ matrices.
We shall use the following ellipticity condition for a given $\sigma\in \mathbb{M}^{N\times N}(\mathbb{R})$
\begin{equation}\label{ell2}
\left\{\begin{array}{ll}
\sigma\xi\cdot\xi\geq a\|\xi\|^2&\text{ for any }\xi\in\mathbb{R}^N\\
\sigma^{-1}\xi\cdot\xi\geq b^{-1}\|\xi\|^2&\text{ for any }\xi\in\mathbb{R}^N.\end{array}\right.
\end{equation}
If $\sigma$ is symmetric then \eqref{ell2} is equivalent to the condition
\begin{equation}\label{ell2sym}
a\|\xi\|^2\leq \sigma\xi\cdot\xi\leq b\|\xi\|^2\quad \text{ for any }\xi\in\mathbb{R}^N.
\end{equation}
Finally, if $\sigma=s I_N$, where $I_N$ is the $N\times N$ identity matrix and $s$ is a real number, the condition further reduces to
$$a\leq s\leq b.$$

We define the following classes of conductivity tensors in $\Omega$,
$\Omega\subset\mathbb{R}^N$ being a bounded open set. We call
$\mathcal{M}(a,b)$ the set of $\sigma\in L^{\infty}(\Omega,\mathbb{M}^{N\times N}(\mathbb{R}))$ such that, for almost any $x\in\Omega$, $\sigma(x)$ satisfies \eqref{ell2}. We call $\mathcal{M}_{sym}(a,b)$, respectively $\mathcal{M}_{scal}(a,b)$, the set of $\sigma\in \mathcal{M}(a,b)$ such that, for almost any $x\in\Omega$, $\sigma(x)$ is symmetric, respectively  $\sigma(x)=s(x) I_N$ with $s(x)$ a real number.
By a conductivity tensor $\sigma$ in $\Omega$, respectively symmetric conductivity tensor or scalar conductivity, we mean $\sigma\in \mathcal{M}(a,b)$, respectively
$\mathcal{M}_{sym}(a,b)$ or $\mathcal{M}_{scal}(a,b)$, for some constants $0<a\leq b$.

Let $\Omega\subset \mathbb{R}^N$, $N\geq 2$, be a bounded domain with Lipschitz boundary. 
We will use the notation $W^{1/2,2}(\partial \Omega)$ for the Sobolev space of traces of
$W^{1,2}(\Omega)$ functions on $\partial \Omega$ and 
$W^{-1/2,2}(\partial \Omega)$ for its dual. We recall that $W^{1/2,2}(\partial\Omega)\subset L^2(\partial\Omega)$ with continuous immersion. 

We denote by $L^2_{\ast}(\partial\Omega)$ the subspace of functions $f\in L^2(\partial\Omega)$ such that $\int_{\partial\Omega}f=0$.
 Correspondingly, we write 
$W^{-1/2,2}_{\ast}(\partial\Omega)$ for the subspace of $g\in W^{-1/2,2}(\partial\Omega)$ such that
$$\langle g,1\rangle_{(W^{-1/2,2}(\partial \Omega),W^{1/2,2}(\partial \Omega))}=0.$$
Note that $L^2_{\ast}(\partial\Omega)\subset W^{-1/2,2}_{\ast}(\partial\Omega)$, with continuous immersion. 
We also denote with $W^{1/2,2}_{\ast}(\partial \Omega)$ the subspace of $\psi\in W^{1/2,2}(\partial\Omega)$ such that
$\int_{\partial\Omega}\psi=0$. Clearly we have $W^{1/2,2}_{\ast}(\partial \Omega)\subset L^2_{\ast}(\partial\Omega)$ with continuous immersion.
 
For any two Banach spaces $B_1$, $B_2$, $\mathcal{L}(B_1,B_2)$
will denote the Banach space of bounded linear operators from $B_1$ to $B_2$ with the usual operator norm.

 For a conductivity tensor $\sigma\in \mathcal{M}(a,b)$, the corresponding Neumann-to-Dirichlet map
 $\mathcal{N}(\sigma)$ is defined for each $g\in W^{-1/2,2}_{\ast}(\partial\Omega)$, as 
$$\mathcal{N}(\sigma)(g)=v|_{\partial\Omega}$$
with $v$ the solution to 
\begin{equation}\label{Neumannproblem1}
\left\{\begin{array}{ll}
-\mathrm{div}(\sigma\nabla v)=0&\text{in }\Omega\\
\sigma\nabla v\cdot \nu=g&\text{on }\partial\Omega\\
\int_{\partial \Omega}v=0.&
\end{array}\right. 
\end{equation}
Then $\mathcal{N}(\sigma)$ is bounded linear operator $$\mathcal{N}(\sigma):W^{-1/2,2}_{\ast}(\partial\Omega)\to W^{1/2,2}_{\ast}(\partial\Omega)$$ with norm bounded by a constant depending only on $N$, $\Omega$ and $a$.

 The inverse conductivity problem thus consists in inverting the operator 
$$\mathcal{N}:\mathcal{M}(a,b)
\to \mathcal{L}(W^{-1/2,2}_{\ast}(\partial \Omega),W^{1/2,2}_{\ast}(\partial \Omega)),$$
i.e. in determining an unknown conductivity $\sigma$ from knowledge of  the corresponding Neumann-to-Dirichlet map $\mathcal{N}(\sigma)$.

 To apply our general multiscale approach to this problem we proceed as follows.
 Let $X=L^1(\Omega,\mathbb{M}^{N\times N}(\mathbb{R}))$, with its natural norm, namely
$$\|\sigma\|_{L^1(\Omega)}=\|(\|\sigma\|)\|_{L^1(\Omega)},$$
where for any $N\times N$ matrix $\sigma$, $\|\sigma\|$ denotes its norm as a linear operator of $\mathbb{R}^N$ into itself.
We may take as subset $E$ any of the following classes $\mathcal{M}(a,b)$,
$\mathcal{M}_{sym}(a,b)$ or $\mathcal{M}_{scal}(a,b)$. We need continuity of the nonlinear operator $\mathcal{N}: E\to Y$.
This is guaranteed, for example, if we choose 
$Y=\mathcal{L}(L^2_{\ast}(\partial\Omega),L^2_{\ast}(\partial\Omega))$, with the distance $d$ induced by its norm. Such a choice for $Y$ is a particularly convenient one, see for instance the discussion in \cite{Ron15}.
 For the continuity of $\mathcal{N}$ with respect to the strong convergence in $X$ and the distance $d$ on $Y$, 
see Proposition~\ref{Hconvcont} and Remark~\ref{L2L2rem}.

 Here $\hat{\mathcal{N}}\in Y$ will be the measured 
Neumann-to-Dirichlet map. The nonnegative number $\delta_0=\inf\{\|\hat{\mathcal{N}}-\mathcal{N}(\sigma)\|_Y:\ \sigma\in E\}$ corresponds to the noise level of the measurements.

There are several possible choices for $|\cdot|$. A particularly interesting one, already widely used in applications, is the total variation regularization. Namely,
we define, for any $\sigma\in X$,
$TV(\sigma)$ as the matrix such that
$TV(\sigma)_{ij}=TV(\sigma_{ij})=|D\sigma_{ij}|(\Omega)$ and set
$|\sigma|_{BV(\Omega)}=\| TV(\sigma) \| $ for any $\sigma\in X$. We also define for any $\sigma\in X$
$$\|\sigma\|_{BV(\Omega)}=\|\sigma\|_{L^1(\Omega)}+|\sigma|_{BV(\Omega)}.$$
Then we may choose as $|\cdot|$ either $|\cdot|_{BV(\Omega)}$ or $\|\cdot\|_{BV(\Omega)}$.

We note that the use of total variation regularizations for the inverse conductivity problem has been shown to be effective from a numerical point of view in several papers, see for instance \cite{Dob e San94,Ron e San,Chan e Tai,Chu e Chan e Tai}. Analytical evidence, through a convergence analysis, of the efficacy of these regularization methods was proved in \cite{Ron08}; see also \cite{Ron16} for further developments in this direction.

 In the setting described above, all the assumptions \ref{BVass1})---\ref{BVass5}) are verified. Therefore, Theorem~\ref{mainthm0}, with the same notation, reads as follows.

\begin{teo}\label{mainthmcond0}
If \eqref{firstcondcoeff} is satisfied, then 
for the multiscale sequence $\{\tilde{\sigma}_n\}_{n\in\mathbb{N}}$ given by \eqref{regularizedpbm}, \eqref{inductiveconstr} and \eqref{partialsum1}
we have that
$$\lim_n\|\hat{\mathcal{N}}-\mathcal{N}(\tilde{\sigma}_n)\|_Y=\inf\{\|\hat{\mathcal{N}}-\mathcal{N}(\sigma)\|_Y:\ \sigma\in E\}.$$
\end{teo}

 For the inverse conductivity problem, we can actually obtain from this a very weak convergence result on $\tilde{\sigma}_n$ as well, if we make use of the notion of $H$-convergence, introduced in the context of homogenization.

\begin{cor}\label{corollary1}
Under the assumptions of Theorem~\textnormal{\ref{mainthmcond0}}, if $E=\mathcal{M}(a,b)$ or 
$E=\mathcal{M}_{sym}(a,b)$, then, up to a subsequence, $\tilde{\sigma}_n$ $H$-converges to $\tilde{\sigma}_{\infty}\in E$, where $\tilde{\sigma}_{\infty}$ is a solution of 
\begin{equation}\label{minpbmcond0}
\min\{\|\hat{\mathcal{N}}-\mathcal{N}(\sigma)\|_Y:\ \sigma\in E\}.
\end{equation}
\end{cor}

Theorem~\ref{mainthmcond0} and Corollary~\ref{corollary1} are special cases of Theorem~\ref{mainthmcond}, that will be stated and proved in Section~\ref{inversesec}.

We note that a solution to \eqref{minpbmcond0} corresponds to a solution to our inverse conductivity problem. Therefore, we have found a numerical algorithm to obtain the solution $\tilde{\sigma}_{\infty}$ to the inverse problem and, simultaneously, a multiscale representation of  $\tilde{\sigma}_{\infty}$, namely
\begin{equation}\label{multiscalerepresolution}
\tilde{\sigma}_{\infty}=\lim_k\tilde{\sigma}_{n_k}=\lim_k\sum_{i=0}^{n_k}\sigma_i,
\end{equation}
where the limit has to be understood in the sense of $H$-convergence. For a definition of $H$-convergence and its basic properties we refer to \cite{All,Mur e Tar1,Mur e Tar2}. Here we just remark that for 
symmetric conductivity tensors $H$-convergence reduces to the more usual $G$-convergence and that 
$\mathcal{M}(a,b)$ and $\mathcal{M}_{sym}(a,b)$ are compact with respect to $H$-convergence.
We also recall that $G$- or $H$-convergence has already been shown to be a useful tool in the context of the inverse conductivity problem, see for instance \cite{Koh e Vog87,Ale e Cab,Far e Kur e Rui,Ron15,Ron16}.

On the other hand, the result in Corollary~\ref{corollary1} has some drawbacks. The first one is that the convergence is in an extremely weak sense and that we exploit in a crucial way the compactness of $E$ with respect to this kind of convergence. This is a particular feature of the problem we are considering but it might not occur in a more general case, like the one used in our abstract setting. The second one is that it does not hold for scalar conductivities. In fact, if we restrict ourselves to scalar conductivities, that is, we choose $E=\mathcal{M}_{scal}(a,b)$, several difficulties arise.
First of all, existence of a solution to \eqref{minpbmcond0} may fail, see for instance Example~3.4 in \cite{Ron15}, and also Example~2.5 in \cite{Ron16}.
Secondly, and more importantly, by compactness of $H$-convergence, it is still true that $\tilde{\sigma}_n$, up to a subsequence, $H$-converges to $\tilde{\sigma}_{\infty}$, but we can not assure that the limit $\tilde{\sigma}_{\infty}$ is a scalar conductivity. 

If we wish to have a stronger convergence than $H$-convergence, and to have a convergence result for scalar conductivities as well, we need to use the tighter multiscale construction from  Subsection~\ref{abstractsecondsub}. Thus, keeping the setting above, we now assume in addition that there exists $\hat{\sigma}\in E$
solving the following minimization problem
\begin{equation}\label{minimalcond0}
\min\{ |\sigma|:\ \sigma\in E\text{ and }\|\hat{\mathcal{N}}-\mathcal{N}(\sigma)\|_Y=\delta_0\}<+\infty.
\end{equation}
We call $\hat{E}$ the set of solutions of \eqref{minimalcond0}. We note that $\hat{E}$ is compact in $X$ and that corresponds to the set of (numerical) solutions of our inverse problem which have minimal value of $|\cdot|$, that is, that have minimal total variation among all possible solutions, if $|\cdot|=|\cdot|_{BV}$. We now construct $\sigma_n$, and using \eqref{tildetight} $\tilde{\sigma}_n$ as well, for $n\geq 1$ inductively by solving the minimization problems \eqref{regularizedpbmbis} and \eqref{tightconstr}. The convergence result then reads as follows.

\begin{teo}\label{mainthmcond20}
Assume that
\eqref{crucialcondition0} holds 
and that there exists a solution $\hat{\sigma}$ of \eqref{minimalcond0}.

 Consider the sequence $\{\tilde{\sigma}_n\}_{n\in\mathbb{N}}$ defined by \eqref{tildetight} from the sequence $\{\sigma_n\}_{n\in\mathbb{N}}$ obtained from \eqref{regularizedpbmbis} and  \eqref{tightconstr}.

Then a subsequence $\tilde{\sigma}_{n_k}$ converges to $\tilde{\sigma}_{\infty}$ strongly in $X$, where $\tilde{\sigma}_{\infty}$ is a \textnormal{(}possibly different from $\hat{\sigma}$\textnormal{)} solution to \eqref{minimalcond0}, that is, $d(\hat{\mathcal{N}},\mathcal{N}(\tilde{\sigma}_{\infty}))=\delta_0$ and $|\tilde{\sigma}_{\infty}|=|\hat{\sigma}|$. We thus have a multiscale decomposition of $\tilde{\sigma}_{\infty}$\textnormal{:} 
\begin{equation}\label{multiscalerepresolution2}
\tilde{\sigma}_{\infty}=\lim_k\sum_{i=0}^{n_k}\sigma_i,
\end{equation}
which is convergent in the norm of  $X$. Moreover, we have that 
$$\lim_n|\tilde{\sigma}_n|= |\hat{\sigma}|$$
and
\begin{equation}\label{distancecond}
\lim_n \mathrm{dist}(\tilde{\sigma}_n,\hat{E})=0.
\end{equation}
\end{teo}

Theorem~\ref{mainthmcond20} is a special case of Theorem~\ref{mainthmcond2}, which will be stated and proved in Section~\ref{inversesec}.
 Here we make a few remarks. If $E=\mathcal{M}(a,b)$ or
$E=\mathcal{M}_{sym}(a,b)$, then we can guarantee existence of a solution of \eqref{minpbmcond0}, see Proposition~\ref{Hconv}. Uniqueness, however, is not guaranteed. (For example, if the noise level is zero, that is 
$\hat{\mathcal{N}}=\mathcal{N}(\sigma)$ for some $\sigma\in E$, the nonuniquenes of the inverse conductivity problem for anisotropic conductivities implies that uniqueness indeed fails).

 If the measured $\hat{\mathcal{N}}$ is admissible, i.e. $\hat{\mathcal{N}}=\mathcal{N}(\sigma)$ for some $\sigma\in E$ with $|\sigma|<+\infty$ then we also have existence for \eqref{minimalcond0}. In the case of non-zero noise level, existence of a solution of \eqref{minimalcond0} is not easy to prove. 

Further details and complete proofs for our results on the Calder\'on problem are in Section~\ref{inversesec}.

\medskip

 We next address a very different problem, namely that of image registration. We seek to extend our multiscale framework to obtain hierarchical decompositions of diffeomorphisms arising in image registration problems.  

\subsection{Multiscale algorithm for diffeomorphic image registration}\label{LDDMMintrosub}

We review the Large Deformation Diffeomorphic Metric Mapping (LDDMM) approach to image registration, mainly following 
Chapter~8 of \cite{You} and Section~3 of \cite{B-H}.  
See also \cite {B-M-T-Y} and earlier references therein. To begin with, we define $G_{\mathcal{H}}$, the group of diffeomorphisms we will be working with, first introduced by Trouv\'e \cite{Trouve95,Trouve98}, along with a distance function on $G_{\mathcal{H}}$. This will make it possible to quantify the size of a deformation by its distance to the identity map.  Let $\Omega$ be an open subset of $\mathbb{R}^N$, $N\geq 1$.
We say that
 $\mathcal{H}$, a Hilbert space of vector fields on $\Omega$, is \emph{admissible} if it is contained and continuously embedded in $C^1_0(\Omega,\mathbb{R}^N)$, the space of $C^1$ vector fields $u$ on $\Omega$ such that $u$ and $Du$ vanish on $\partial\Omega$ and at infinity. An example of an admissible Hilbert space $\mathcal{H}$ is the Sobolev space $H^s_0(\Omega,\mathbb{R}^N)=W^{s,2}_0(\Omega,\mathbb{R}^N)$ for any $s>N/2+1.$ Having chosen an admissible $\mathcal{H}$, we consider the Hilbert space $L^2([0,1],\mathcal{H})$ of time-dependent vector fields with the usual scalar product
$$\langle u,v\rangle_{L^2([0,1],\mathcal{H})}=\int_0^1\langle u(t),v(t)\rangle_{\mathcal{H}} \rmd t\quad\text{for any }u,v\in
L^2([0,1],\mathcal{H}).$$

We let
\begin{equation}
G_{\mathcal{H}}=\{g=\varphi^u(1):\ 
u\in L^2([0,1],\mathcal{H})\}
\end{equation}
where $\varphi^u$ is the solution of $\partial_t\varphi=u\circ\varphi$ with initial condition $\varphi(0)=e$,  $e$ denoting the identity map.

For any $t\in[0,1]$, in particular for $t=1$, the map $\varphi^u(t)$ is a $C^1$ diffeomorphism of $\Omega$ onto itself. Actually, $\varphi^u(t)$ (and its inverse as well) may be extended to all of $\mathbb{R}^N$ by letting it be equal to the identity outside $\Omega$ and this extension is a $C^1$ diffeomorphism of $\mathbb{R}^N$ onto itself.

 The set $G_{\mathcal{H}}$ thus defined is a group with respect to the composition of maps and a complete metric space endowed with the distance
$$d_{\mathcal{H}}(g_0,g_1)=\min_{u\in L^2([0,1],\mathcal{H})}\left\{\|u\|:\ g_1=g_0\circ\varphi^u(1)
\right\}\quad\text{for any }g_0,g_1\in
G_{\mathcal{H}}.$$

Here and in the sequel $\|u\|=\|u\|_{L^2([0,1],\mathcal{H})}$. The distance satisfies the following left invariance property
$$d_{\mathcal{H}}(g_0,g_1)=d_{\mathcal{H}}(g\circ g_0,g\circ g_1)\quad\text{for any }g\in G_{\mathcal{H}}.$$

 In particular, for any $g_1\in G_{\mathcal{H}}$ 
$$d_{\mathcal{H}}(e,g_1)=\min_{u\in L^2([0,1],\mathcal{H})}\left\{\|u\|:\ g_1=\varphi^u(1)
\right\}$$
and, by left invariance, we have $d_{\mathcal{H}}(e,g_1)=d_{\mathcal{H}}(e,g_1^{-1})$.

	\begin{oss}\label{rightinvariantoss}
		It is often helpful to think of $G_{\mathcal{H}}$ as a manifold (to make this precise, one has to work in the category of Banach manifolds, see~\cite{B-H}).
		The inner product on $\mathcal{H}$ induces a right-invariant Riemannian metric on $G_{\mathcal{H}}$ by
		\begin{equation}
			T_g G_{\mathcal{H}}\times T_g G_{\mathcal{H}} \ni (U,V) \mapsto \langle U\circ g^{-1},V\circ g^{-1}\rangle_{\mathcal{H}}.
		\end{equation}
		Let $d_R$ denote the right-invariant Riemannian distance associated with this metric.
		Then the relation between $d_R$ and $d_{\mathcal{H}}$ is
		\begin{equation}
			d_{\mathcal{H}}(g_1,g_2) = d_R(g_1^{-1},g_2^{-1}).
		\end{equation}
		Consequently, due to the invariance properties, the two distances coincide when $g_2 = e$.
	\end{oss}

The LDDMM approach to image registration consists of the following.
 We are given two images $I_0$ and $I_1$ belonging to $L^2(\Omega)$.
For any $g\in G_{\mathcal{H}}$, write $\psi=g^{-1}$ and define 
\begin{equation}\label{Udefin}U_{I_0,I_1}(g)=\|I_0\circ g^{-1} -I_1\|_{L^2(\Omega)}\quad\text{and}\quad 
\tilde{U}_{I_0,I_1}(\psi)=\|I_0\circ \psi -I_1\|_{L^2(\Omega)}.
\end{equation}
Then, for some parameter $\lambda>0$, and positive constants $\alpha$ and $\beta$, we seek a diffeomorphism $$g\in G_{\mathcal{H}}$$ which is a solution to the following minimization problem
\begin{equation}\label{000}
\min\left\{\left(\lambda\|I_0\circ g^{-1} -I_1\|_{L^2(\Omega)}^{\alpha}+d_{\mathcal{H}}(g,e)^{\beta}\right):\ g\in G_{\mathcal{H}}\right\}.
\end{equation}
We note that problem \eqref{000} is equivalent to solving\begin{equation}\label{001}
\min\left\{\left(\lambda\|I_0\circ g^{-1} -I_1\|_{L^2(\Omega)}^{\alpha}+d_{\mathcal{H}}(g^{-1},e)^{\beta}\right):\ g\in G_{\mathcal{H}}\right\}
\end{equation}
that is, with the above notation $\psi=g^{-1}$, 
\begin{equation}\label{002}
\min\left\{\left(\lambda\|I_0\circ \psi -I_1\|_{L^2(\Omega)}^{\alpha}+d_{\mathcal{H}}(\psi,e)^{\beta}\right):\ \psi\in G_{\mathcal{H}}\right\}.
\end{equation}

The minimization problem \eqref{000} admits a solution. This follows from 
the proof of Theorem~21 in 
\cite{B-H} and will be outlined in Section~\ref{registrationsec}, see Theorem~\ref{minsolutionthm0}. It is essentially based on compactness and semicontinuity properties with respect to the following kind of convergence.

\begin{defin}\label{weakconv}
Given a sequence $\{g_n\}_{n\in\mathbb{N}}\subset G_{\mathcal{H}}$, we say that $g_n$ weakly converges in $G_{\mathcal{H}}$ to $g$, as $n\to \infty$, if and only if there exists a constant $C>0$ such that 
$\|Dg_n\|_{L^{\infty}}\leq C$ and $\|Dg^{-1}_n\|_{L^{\infty}}\leq C$ for any $n\in\mathbb{N}$, and
$g_n\to g$ and $(g_n)^{-1}\to g^{-1}$ uniformly on compact subsets of $\overline{\Omega}$.
\end{defin} 

It is easy to see that a weak limit to $\{g_n\}_{n\in\mathbb{N}}\subset G_{\mathcal{H}}$, if it exists, is unique and that $g_n$ weakly converges to $g$ if and only if
$g^{-1}_n$ weakly converges to $g^{-1}$.

We are now ready to describe our multiscale construction for the registration problem. Let us fix, as before, positive constants $\alpha$, $\beta$ and $\gamma$ and let $\lambda_n>0$ and $a_n\geq 0$ for any $n\in \mathbb{N}$. Again we assume that $a_n\leq a_{n-1}$ for any $n\geq 1$.

We let $g_0$ be a  minimizer of
\begin{equation}\label{0001}
\min\left\{\left(\lambda_0[\|I_0\circ g^{-1} -I_1\|_{L^2(\Omega)}^{\alpha}+a_0d_{\mathcal{H}}(g,e)^{\gamma}]+d_{\mathcal{H}}(g,e)^{\beta}\right):\ g\in G_{\mathcal{H}}\right\}.
\end{equation}
or, equivalently, let $\psi_0=g_0^{-1}$ be a minimizer of 
\begin{equation}\label{0001bis}
\min\left\{\left(\lambda_0[\|I_0\circ \psi -I_1\|_{L^2(\Omega)}^{\alpha}+a_0d_{\mathcal{H}}(\psi,e)^{\gamma}]+d_{\mathcal{H}}(\psi,e)^{\beta}\right):\ \psi\in G_{\mathcal{H}}\right\}.
\end{equation}
For the proofs, and to relate to the general framework in Section~\ref{topgroupsubs}, it
will be helpful to work with both minimization problems throughout.

Existence of minimizers will be proved in Section~\ref{registrationsec}. By induction, there exists a minimizer $g_n$, $n\geq 1$, of \begin{equation}\label{111}
\min_{g\in G_{\mathcal{H}}}\left(\lambda_n[\|I_0\circ \tilde{g}^{-1}_{n-1}\circ g^{-1} -I_1\|_{L^2(\Omega)}^{\alpha}+a_nd_{\mathcal{H}}(g\circ\tilde{g}_{n-1},e)^{\gamma}]+d_{\mathcal{H}}(g,e)^{\beta}\right)
\end{equation}
where $\tilde{g}_0=g_0$ and
for any $n\geq 1$ we denote by $\tilde{g}_n$ the composition 
$$\tilde{g}_n=g_n\circ\cdots\circ g_0.$$
 We also let $\psi_n=g_n^{-1}$ be solution of \begin{equation}\label{111bis}
\min_{\psi\in G_{\mathcal{H}}}\left(\lambda_n[\|I_0\circ \tilde{\psi}_{n-1}\circ \psi -I_1\|_{L^2(\Omega)}^{\alpha}+a_nd_{\mathcal{H}}(\tilde{\psi}_{n-1}\circ \psi,e)^{\gamma}]+d_{\mathcal{H}}(\psi,e)^{\beta}\right)
\end{equation}
where $\tilde{\psi}_0=\psi_0$ and
for any $n\geq 1$ we denote
$$\tilde{\psi}_n=\psi_0\circ\cdots\circ \psi_n=\tilde{g}_n^{-1}.$$

We have that the sequence $\{g_n\}_{n\in\mathbb{N}}$ exists, however we can not guarantee that it is uniquely determined.
By taking $g=e$, we have
\begin{multline}
\|I_0\circ \tilde{g}^{-1}_n -I_1\|_{L^2(\Omega)}^{\alpha}\leq
\|I_0\circ \tilde{g}^{-1}_n -I_1\|_{L^2(\Omega)}^{\alpha}+a_nd_{\mathcal{H}}(\tilde{g}_n,e)^{\gamma}\leq\\
\|I_0\circ \tilde{g}^{-1}_{n-1}-I_1\|_{L^2(\Omega)}^{\alpha}+a_{n-1}d_{\mathcal{H}}(\tilde{g}_{n-1},e)^{\gamma}\quad\text{for any }n\geq 1,
\end{multline}
so we can denote
$$\varepsilon_0=\lim_n\left(\|I_0\circ \tilde{g}^{-1}_n -I_1\|_{L^2(\Omega)}^{\alpha}+a_nd_{\mathcal{H}}(\tilde{g}_n,e)^{\gamma}\right)^{1/\alpha}
$$
and
$$\delta_0=
\inf\{\|I_0\circ g^{-1} -I_1\|_{L^2(\Omega)}
:\ g\in G_{\mathcal{H}}\}
.$$
Clearly we have that $\varepsilon_0\geq \delta_0$.

The first result we have is a convergence of the corresponding images.
\begin{teo}\label{mainthmter}
Assuming \eqref{secondcondcoeff} holds as before, 
then for the sequence of diffeomorphims $\tilde{g}_n$ constructed above we have $\varepsilon_0= \delta_0$ and we also have
$$\lim_n \|I_0\circ \tilde{g}^{-1}_n -I_1\|_{L^2(\Omega)}=\inf\{\|I_0\circ g^{-1} -I_1\|_{L^2(\Omega)}
:\ g\in G_{\mathcal{H}}\}=\delta_0
.$$
\end{teo}

This result will be an immediate consequence of Theorem~\ref{mainthmter1} in Section~\ref{topgroupsubs}, where we will develop an extension of our generalized abstract formulation to a topological group setting. 

We next adress the question of convergence of the sequence $\{\tilde{g}_n\}_{n\in\mathbb{N}}$. Let
$\{\tilde{u}_n\}_{n\in\mathbb{N}}$ be a sequence in $L^2([0,1],\mathcal{H})$ such that
$\tilde{g}_n=\varphi^{\tilde{u}_n}(1)$ and
$d_{\mathcal{H}}(\tilde{g}_n,e)=\|\tilde{u}_n\|$.
We will show in Section~\ref{registrationsec} that if the sequence $\{\tilde{g}_n\}_{n\in\mathbb{N}}$, or one of its subsequences, converges weakly to some $\tilde{g}_{\infty}$ in $G_{\mathcal{H}}$, then
$\tilde{g}_{\infty}$ is a solution to 
the following minimization problem
\begin{equation}\label{minpbmquater1}
\min\{\|I_0\circ g^{-1} -I_1\|_{L^2(\Omega)}:\ g\in G_{\mathcal{H}}\}=\min\{U_{I_0,I_1}(g):\ g\in G_{\mathcal{H}}\}.
\end{equation}

In fact, the following crucial lemma holds true.

\begin{lem}\label{limitproprem1}
If $\{\tilde{g}_n\}_{n\in\mathbb{N}}$ has a bounded subsequence, then
\eqref{minpbmquater1} admits a solution. In particular,
if $\{\tilde{g}_n\}_{n\in\mathbb{N}}$ has a bounded subsequence, then there exists a further subsequence $\{\tilde{g}_{n_k}\}_{k\in\mathbb{N}}$
such that,
as $k\to\infty$, $\tilde{g}_{n_k}$ converges to $\tilde{g}_{\infty}$ weakly, with $\tilde{g}_{\infty}$ a solution to 
\eqref{minpbmquater1}.
\end{lem}

Lemma~\ref{limitproprem1} will be proved as part of Lemma~\ref{distanceimpliespointwise}. Here we point out
that a solution to \eqref{minpbmquater1} is a registration map between the two images $I_0$ and
$I_1$. Existence and uniqueness of such a solution will be further discussed in Section~\ref{registrationsec}. Existence is equivalent to saying that there exists $\hat{g}\in G_{\mathcal{H}}$ such that
$$\|I_0\circ \hat{g}^{-1} -I_1\|_{L^2(\Omega)}
=\delta_0=\min\{\|I_0\circ g^{-1} -I_1\|_{L^2(\Omega)}:\ g\in G_{\mathcal{H}}\}.
$$
or, equivalently, that there exists $\hat{g}\in G_{\mathcal{H}}$
solving the following minimization problem
\begin{equation}\label{minimalnew0}
\min\{d_{\mathcal{H}}(g,e):\ g\in G_{\mathcal{H}}\text{ and }
\|I_0\circ g^{-1} -I_1\|_{L^2(\Omega)}=\delta_0\}<+\infty.
\end{equation}
We can think of $\hat{g}$ as an optimal registration, since it is a registration whose distance from the identity is minimal.

We call $\hat{G}$ the set of solutions to \eqref{minimalnew0}, which is the set of optimal registrations. We have
that $\hat{G}$
is closed and bounded with respect to the topology induced by the distance in $G_{\mathcal{H}}$
and it is
sequentially compact with respect to the weak convergence in $G_{\mathcal{H}}$.
The same topological properties are shared by $\hat{G}^{-1}=\{g^{-1}:\ g\in\hat{G}\}$.

Indeed, a stronger statement holds.
If $\{\tilde{g}_n\}_{n\in\mathbb{N}}$, or one of its subsequences, converges weakly to some $\tilde{g}_{\infty}$ in $G_{\mathcal{H}}$, then $\tilde{g}_{\infty}\in \hat{G}$, thus it is an optimal registration between the two images $I_0$ and
$I_1$ and we also have a multiscale decomposition of the diffeomorphism $\tilde{g}_{\infty}$: $$\tilde{g}_{\infty}=\lim_k\tilde{g}_{n_k}=\lim_kg_{n_k}\circ\cdots\circ g_0,$$
where the limit is in the sense of Definition~\ref{weakconv}.

In view of Lemma~\ref{limitproprem1},
we need to find conditions under which we have boundedness of $\{\tilde{g}_n\}_{n\in\mathbb{N}}$ or of one of its subsequences.
Clearly a necessary condition is that a solution to 
\eqref{minpbmquater1} does exist. Our main and surprising result is that, in the setting above, this is also a sufficient condition, as shown in the theorem below.
\begin{teo}\label{diffeoconv2}
Assume that
\eqref{crucialcondition0}
holds.

Then the sequence of diffeomorphisms 
$\{\tilde{g}_n\}_{n\in\mathbb{N}}$ constructed above is bounded if and only if a solution to \eqref{minpbmquater1} exists.

In this case, there exists a subsequence
$\{\tilde{g}_{n_k}\}_{k\in\mathbb{N}}$ and $\tilde{g}_{\infty}\in \hat{G}$
such that,
as $k\to\infty$, $\tilde{g}_{n_k}$ converges to $\tilde{g}_{\infty}$ weakly, that is, in particular, $\tilde{g}_{n_k}\to \tilde{g}_{\infty}$ and $(\tilde{g}_{n_k})^{-1}\to (\tilde{g}_{\infty})^{-1}$ uniformly on compact 
subsets of $\overline{\Omega}$.
\end{teo}

A more complete version of this result is stated in Theorem~\ref{diffeoconv}, which will be proved in Section~\ref{registrationsec}.

\section{The general abstract results in a Banach setting}\label{abstractsec}

In this section we prove the general results on the multiscale procedure in the Banach setting, in particular we prove
Theorems~\ref{mainthm0}, \ref{mainthmbis0} and \ref{minimizercor0}

We then consider the application of the abstract procedure to the $(BV,L^2)$ decomposition, in particular we prove Theorem~\ref{ourhierarchythm}.

\subsection{The multiscale approach: general abstract results}\label{abstractresultproofsec}

 We begin by proving
Theorem~\ref{mainthm0}.

\begin{proof}[Proof of Theorem~\textnormal{\ref{mainthm0}}]
Clearly we have that $\varepsilon_0\geq \delta_0$. We need to show that
$\varepsilon_0\leq \delta_0$.
By contradiction, let us assume that $\delta_0<\varepsilon_0$.
There exists $0<C_1<1$ 
such that
\begin{equation}\label{C1def}
\delta_0<C_1^{1/\alpha}\varepsilon_0<\varepsilon_0,
\end{equation}
thus there exists $\overline{\sigma}\in E$ such that
\begin{equation}\label{sigmabardef}
d(\hat{\mathcal{N}},\mathcal{N}(\overline{\sigma}))\leq C_1^{1/\alpha}\varepsilon_0\quad\text{and}\quad
|\overline{\sigma}|<+\infty.
\end{equation}
This is due to the density assumption \ref{BVass3}) on page~\pageref{BVass3} and to the continuity of $\mathcal{N}$.

Then, recalling that $\sigma_0=\tilde{\sigma}_0$,
$$\left(\lambda_0[d(\hat{\mathcal{N}},\mathcal{N}(\tilde{\sigma}_0))^{\alpha}]+|\sigma_0|^{\beta}\right)\leq \left(\lambda_0[d(\hat{\mathcal{N}},\mathcal{N}(\overline{\sigma}))^{\alpha}]+|\overline{\sigma}|^{\beta}\right).$$
Analogously, for any $n\geq 1$, choosing $\sigma=\overline{\sigma}-\tilde{\sigma}_{n-1}$, we have
\begin{equation}
\left(\lambda_n[d(\hat{\mathcal{N}},\mathcal{N}(\tilde{\sigma}_n))^{\alpha}]+|\sigma_n|^{\beta}\right)\leq
\left(\lambda_n[d(\hat{\mathcal{N}},\mathcal{N}(\overline{\sigma}))^{\alpha}]+|\overline{\sigma}-\tilde{\sigma}_{n-1}|^{\beta}\right).
\end{equation}

Hence, for any $n\geq 1$, 
\begin{equation}\label{firstcrucialestimate}
\lambda_n(1-C_1)[d(\hat{\mathcal{N}},\mathcal{N}(\tilde{\sigma}_n))^{\alpha}]+|\sigma_n|^{\beta}\leq\\
|\overline{\sigma}-\tilde{\sigma}_{n-1}|^{\beta}.
\end{equation}
From \eqref{firstcrucialestimate}, we obtain that for any $n\geq 1$,
$$|\overline{\sigma}-\tilde{\sigma}_n|\leq 
|\overline{\sigma}-\tilde{\sigma}_{n-1}|+|\sigma_n|\leq 2|\overline{\sigma}-\tilde{\sigma}_{n-1}|.$$

Let us call $C_0=|\overline{\sigma}-\tilde{\sigma}_0|$, then
$$|\overline{\sigma}-\tilde{\sigma}_n|\leq C_02^{n}\quad \text{for any }n\geq 0.$$
Therefore, again using \eqref{firstcrucialestimate},
$$\lambda_n(1-C_1)d(\hat{\mathcal{N}},\mathcal{N}(\tilde{\sigma}_n))^{\alpha}\leq 
\left(C_02^{n-1}\right)^{\beta}\quad \text{for any }n\geq 1,$$
that is
\begin{equation}
\lambda_n(1-C_1)d(\hat{\mathcal{N}},\mathcal{N}(\tilde{\sigma}_n))^{\alpha}\leq 
\tilde{C}_02^{\beta n}\quad \text{for any }n\geq 1,
\end{equation}
where
$\tilde{C}_0=\left(C_0/2\right)^{\beta}$.
Hence, for any $n\geq 1$,
$$(1-C_1)\varepsilon_0^{\alpha}\leq
(1-C_1)d(\hat{\mathcal{N}},\mathcal{N}(\tilde{\sigma}_n))^{\alpha}\leq
\tilde{C}_0\frac{2^{\beta n}}{\lambda _n}$$
which leads to a contradiction if $\limsup_n\frac{2^{\beta n}}{\lambda _n}=0$.

The case in which $\limsup_n\frac{2^{\beta n}}{\lambda _n}=\tilde{C}$, $0<\tilde{C}<+\infty$, requires the following argument. For some $\overline{n}_1\geq 1$ and for any $n\geq \overline{n}_1$ we have
$2^{\beta n}/\lambda _n\leq \tilde{C}+1$.

Then we note that, for any $n\geq \overline{n}_1$,
$$\lambda_n(1-C_1)\varepsilon_0^{\alpha}\leq |\overline{\sigma}-\tilde{\sigma}_{n-1}|^{\beta}
-|\sigma_n|^{\beta}.$$

Let $a$, $0<a<1$, to be fixed later. If, for some $n\geq \overline{n}_1$,
$|\sigma_n|^{\beta}\geq a |\overline{\sigma}-\tilde{\sigma}_{n-1}|^{\beta}$, then
$$\lambda_n(1-C_1)\varepsilon_0^{\alpha}\leq (1-a)|\overline{\sigma}-\tilde{\sigma}_{n-1}|^{\beta}\leq
(1-a)\tilde{C}_02^{\beta n},$$
therefore
$$(1-C_1)\varepsilon_0^{\alpha}\leq (1-a)\tilde{C}_0(\tilde{C}+1).$$
Hence we can find $a_0$, $0<a_0<1$, such that if $a\geq a_0$ then
$(1-a)\tilde{C}_0(\tilde{C}+1)<(1-C_1)\varepsilon_0^{\alpha}$.
This implies that for any $n\geq \overline{n}_1$ we have that
$$|\sigma_n|\leq a_0^{1/\beta} |\overline{\sigma}-\tilde{\sigma}_{n-1}|$$
or in other words there exists $b_0$, $1<b_0<2$, such that
$$|\overline{\sigma}-\tilde{\sigma}_n|\leq b_0|\overline{\sigma}-\tilde{\sigma}_{n-1}|\quad\text{for any }n\geq\overline{n}_1.$$

We conclude that
$$\lim_n \frac{|\overline{\sigma}-\tilde{\sigma}_{n-1}|^{\beta}}{\lambda_n}=
\lim_n \frac{2^{\beta n}}{\lambda_n}\frac{|\overline{\sigma}-\tilde{\sigma}_{n-1}|^{\beta}}{2^{\beta n}}=0$$
and this leads to a contradiction.
\end{proof}

In Subsection~\ref{abstractsecondsub} we slightly modified our multiscale approach and studied conditions that guarantee
convergence of $\{\tilde{\sigma}_n\}_{n\in\mathbb{N}}$, or of one of its subsequences. Before dealing with the proof of the 
second abstract formulation, we consider an easy although interesting consequence of Theorem~\ref{mainthm0}. In general, it is difficult to guarantee boundedness, in $X$, of the sequence $\{\tilde{\sigma}_n\}_{n\in\mathbb{N}}$. A related counterexample is given 
in the second version of Example~\ref{L2minexample}.
However this may be obtained through a mild coercivity condition as it is shown in the next result.

\begin{prop}\label{coerciveprop} 
Let the assumptions of Theorem~\textnormal{\ref{mainthm0}} be satisfied.
Let $d_X$ be any metric on $X$, possibly different from the one induced by the norm on $X$. We also fix $\overline{\sigma}\in E$. We assume that $\mathcal{N}$ satisfies the following mild coercivity condition. Assume that for any $\varepsilon>0$ there exists $R>0$ such that 
\begin{equation}\label{coercivity}
d(\hat{\mathcal{N}},\mathcal{N}(\sigma))\geq \delta_0+\varepsilon\quad\text{for any }\sigma\in E\text{ such that }d_X(\sigma,\overline{\sigma})\geq R.
\end{equation}

Then $\{\tilde{\sigma}_n\}_{n\in\mathbb{N}}$ is bounded with respect to the metric $d_X$, that is, there exists a constant $C$ such that $d_X(\tilde{\sigma}_n,\overline{\sigma})\leq C$ for any $n\in\mathbb{N}$.

If furthermore, for any $C\in \mathbb{R}$,
$\{\sigma\in E:\ d_X(\sigma,\overline{\sigma})\leq C\}$ is relatively sequentially compact in $X$, with respect to the metric induced by the norm on $X$, we conclude that a solution to 
\eqref{minpbm0} does exist and that $\{\tilde{\sigma}_n\}_{n\in\mathbb{N}}$ converges, up to a subsequence, to some $\tilde{\sigma}_{\infty}$, $\tilde{\sigma}_{\infty}$ solution to \eqref{minpbm0}.
\end{prop}

\begin{proof} Immediate by Theorem~\ref{mainthm0}, since
$\lim_nd(\hat{\mathcal{N}},\mathcal{N}(\tilde{\sigma}_n))=\delta_0$.
\end{proof}

Let us now proceed with the proofs of the results stated in Subsection~\ref{abstractsecondsub}.

\begin{proof}[Proof of Theorem~\textnormal{\ref{mainthmbis0}}]
We need to show that
$\varepsilon_0\leq \delta_0$.
By contradiction, let us assume that $\delta_0<\varepsilon_0$. Hence there exist $0<C_1<C_2<1$ 
such that
$$\delta_0<C_1^{1/\alpha}\varepsilon_0<C_2^{1/\alpha}\varepsilon_0<\varepsilon_0$$
and $\overline{\sigma}\in E$ satisfying \eqref{sigmabardef}. 

Then, recalling that $\sigma_0=\tilde{\sigma}_0$,
$$\left(\lambda_0[d(\hat{\mathcal{N}},\mathcal{N}(\tilde{\sigma}_0))^{\alpha}+a_0|\tilde{\sigma}_0|^{\gamma}]+|\sigma_0|^{\beta}\right)\leq \left(\lambda_0[d(\hat{\mathcal{N}},\mathcal{N}(\overline{\sigma}))^{\alpha}+a_0|\overline{\sigma}|^{\gamma}]+|\overline{\sigma}|^{\beta}\right).$$
Analogously, for any $n\geq 1$, choosing $\sigma=\overline{\sigma}-\tilde{\sigma}_{n-1}$, we have
\begin{multline}\label{cruciale}
\left(\lambda_n[d(\hat{\mathcal{N}},\mathcal{N}(\tilde{\sigma}_n))^{\alpha}+a_n|\tilde{\sigma}_n|^{\gamma}]+|\sigma_n|^{\beta}\right)\leq\\
\left(\lambda_n[d(\hat{\mathcal{N}},\mathcal{N}(\overline{\sigma}))^{\alpha}+a_n|\overline{\sigma}|^{\gamma}]+|\overline{\sigma}-\tilde{\sigma}_{n-1}|^{\beta}\right).
\end{multline}

For some $\overline{n}\geq 1$ and for any $n\geq \overline{n}$, we have that
$$a_n|\overline{\sigma}|^{\gamma}\leq (C_2-C_1)\varepsilon_0^{\alpha}
$$
since $a_n$ goes to zero as $n\to \infty$.

Hence, for any $n\geq \overline{n}$, 
\begin{equation}\label{firstcrucialestimatebis}
\lambda_n(1-C_2)[d(\hat{\mathcal{N}},\mathcal{N}(\tilde{\sigma}_n))^{\alpha}+a_n|\tilde{\sigma}_n|^{\gamma}]+|\sigma_n|^{\beta}\leq
|\overline{\sigma}-\tilde{\sigma}_{n-1}|^{\beta}.
\end{equation}
From \eqref{firstcrucialestimatebis}, we obtain that for any $n\geq\overline{n}$,
$$|\overline{\sigma}-\tilde{\sigma}_n|\leq 
|\overline{\sigma}-\tilde{\sigma}_{n-1}|+|\sigma_n|\leq 2|\overline{\sigma}-\tilde{\sigma}_{n-1}|.$$

Then the proof concludes analogously to the proof of Theorem~\ref{mainthm0}.
\end{proof}

We conclude this subsection with the proof of Theorem~\ref{minimizercor0}.
\begin{proof}[Proof of Theorem~\textnormal{\ref{minimizercor0}}]
First of all we note that $|\sigma_0|=|\tilde{\sigma}_0|\leq |\hat{\sigma}|$.
Then we use \eqref{cruciale} with $\overline{\sigma}$ replaced by $\hat{\sigma}$ for any $n\geq 1$ and we have the following two cases. In the first one we have that $|\tilde{\sigma}_n|\leq|\hat{\sigma}|$, in the second we have that
$|\tilde{\sigma}_n|>|\hat{\sigma}|$, hence $|\sigma_n|\leq |\hat{\sigma}-\tilde{\sigma}_{n-1}|$.

Let us now fix $n\geq 1$. If we have that $|\tilde{\sigma}_n|\leq|\hat{\sigma}|$, then we also have
$|\hat{\sigma}-\tilde{\sigma}_n|\leq 2|\hat{\sigma}|$. Otherwise, let us consider $m$, $0\leq m<n$, the last integer for which the first case happens. We immediately deduce, as before, that 
$$|\hat{\sigma}-\tilde{\sigma}_n|\leq 2^{n-m}(2|\hat{\sigma}|)\leq 2^{n+1}|\hat{\sigma}|.
$$
In any case, we conclude that
for any $n\geq 0$
\begin{equation}\label{growth}
|\hat{\sigma}-\tilde{\sigma}_n|\leq 2^{n+1}|\hat{\sigma}|\quad
\text{and}\quad
|\tilde{\sigma}_n|\leq (2^{n+1}+1)|\hat{\sigma}|.
\end{equation}

Let us assume that there exists $h>0$ such that $\limsup_n|\tilde{\sigma}_n|^\gamma\geq (|\hat{\sigma}|^\gamma+h)$. Then we can find a subsequence $n_k$ such that for any $k\geq 1$ we have
$$|\tilde{\sigma}_{n_k}|^{\gamma}\geq (|\hat{\sigma}|^{\gamma}+h/2).$$
Therefore,
$$\lambda_{n_k}a_{n_k} (h/2)+|\sigma_{n_k}|^{\beta}\leq |\hat{\sigma}-\tilde{\sigma}_{n_k-1}|^{\beta}
\leq 2^{\beta n_k}|\hat{\sigma}|^{\beta}$$
which leads to a contradiction if \eqref{crucialcondition0} holds.

Therefore we may conclude that
\begin{equation}\label{limsupresult}
\limsup_n|\tilde{\sigma}_n|\leq  |\hat{\sigma}|.
\end{equation}
The rest of the proof easily follows.
\end{proof}

\subsection{The $(BV,L^2)$ case: proof of Theorem~\ref{ourhierarchythm} and related results}\label{L2BVsec}

The homogeneous $BV$ space on  $\mathbb{R}^N$, $N\geq 1$, is defined as follows.
We say that $u\in BV(\mathbb{R}^N)$ if $Du$, its gradient in the distributional sense, is a bounded vector valued Radon measure on $\mathbb{R}^N$ and $u$ satisfies a suitable condition at infinity. Namely,
if $N=1$, then $BV(\mathbb{R})\subset L^{\infty}(\mathbb{R})$, with continuous immersion, and the condition at infinity here
is just that (a good representative of) $u\in BV(\mathbb{R})$ satisfies $\lim_{t\to- \infty} u(t)=0$.
If $N\geq 2$, we require that $u\in BV(\mathbb{R}^N)$ vanishes at infinity in a weak sense. We note that
$BV(\mathbb{R}^N)\subset L^{N/(N-1)}(\mathbb{R}^N)$, with continuous immersion, and $u$ belonging to
$L^{N/(N-1)}(\mathbb{R}^N)$ already guarantees that $u$ vanishes at infinity in a weak sense. Finally, we endow the homogeneous
$BV(\mathbb{R}^N)$ space with the norm given by the total variation of $Du$, namely
$$\|u\|_{BV(\mathbb{R}^N)}=|Du|(\mathbb{R}^N).$$
We refer to \cite[Section~1.12]{YMey} for further details.

The multiscale procedure developed in \cite{T-N-V} is the following.
We fix $f\in L^2(\mathbb{R}^N)$, $N\geq 1$, and positive numbers $\lambda_n$ for any $n\in \mathbb{N}$.
Let $u_0$ be the solution to
\begin{equation}\label{a11}
\min\left\{\left(\lambda_0\|f-v\|^2_{L^2(\mathbb{R}^N)}+\|v\|_{BV(\mathbb{R}^N)}\right):\ v\in L^2(\mathbb{R}^N)\right\}.
\end{equation}
Lemma~1 of \cite[Section~1.12]{YMey} guarantees that such a minimization problem admits a unique solution (uniqueness being ensured by strict convexity of the functional to be minimized).
Then, by induction, let $u_n$, $n\geq 1$, be the solution to
\begin{multline}\label{a1n}
\min\Big\{\left(\lambda_n\|f-u_0-u_1-\ldots-u_{n-1}-v)\|^2_{L^2(\mathbb{R}^N)}+\|v\|_{BV(\mathbb{R}^N)}\right):\\ v\in L^2(\mathbb{R}^N)\Big\}.
\end{multline}

Then we have the following hierarchical decomposition.

\begin{teo}\label{hierarchythm}
Let $f\in L^2(\mathbb{R}^N)$, $N\geq 1$, and assume that \eqref{firstcondcoeff} holds.

Then we have the following $(BV,L^2)$ hierarchical decomposition of $f$
\begin{equation}\label{hy}
f=\sum_{j=0}^{+\infty}u_j
\end{equation}
where the convergence is in the strong sense in $L^2(\mathbb{R}^N)$.

Furthermore, if $N=2$, we also have the following energy equality
\begin{equation}\label{energyeq1}
\|f\|_{L^2(\mathbb{R}^N)}^2=\sum_{j=0}^{+\infty}\left[\frac{1}{\lambda_j}\|u_j\|_{BV(\mathbb{R}^N)}+\|u_j\|_{L^2(\mathbb{R}^N)}^2\right].
\end{equation}
\end{teo}

\begin{proof}
Let $X=L^2(\mathbb{R}^N)$ with its usual norm, $E=X$, $Y=L^2(\mathbb{R}^N)$ and $\mathcal{N}=Id$. Let $\hat{\mathcal{N}}=f\in L^2(\mathbb{R}^N)$. We set $|\cdot|=\|\cdot\|_{BV(\mathbb{R}^N)}$ and
$\alpha=2$ and $\beta=1$.
Finally we note that $\delta_0=0$. Then the decomposition is an immediate consequence of Theorem~\ref{mainthm0}.

The restriction to $N=2$ for the energy equality is due to the fact that we need that
 $BV(\mathbb{R}^N)\subset L^2(\mathbb{R}^N)$, with continuous immersion. Its proof is already contained in \cite[Theorem~2.2]{T-N-V} and exploits arguments developed in \cite{YMey}, which we will recall in the proof of the next result, Theorem~\ref{hierarchythmbdd}.
\end{proof}

Let us note that we have proved, and actually extended to any space dimension $N\geq1$, Theorem~\ref{ourhierarchythm}.
We can further extend Theorem~\ref{hierarchythm} to bounded domains as follows.

\begin{teo}\label{hierarchythmbdd}
Let $\Omega\subset\mathbb{R}^N$, $N\geq 1$, be a bounded open set with Lipschitz boundary.
Let $f\in L^2(\Omega)$ and assume that \eqref{firstcondcoeff} holds. Let us construct the sequence $\{u_j\}_{j=0}^{+\infty}$
as before, using \eqref{a11} and \eqref{a1n} with $\mathbb{R}^N$ replaced by $\Omega$.

Then we have the following $(BV,L^2)$ hierarchical decomposition of $f$
\begin{equation}\label{hy2}
f=\sum_{j=0}^{+\infty}u_j
\end{equation}
where the convergence is in the strong sense in $L^2(\Omega)$.

Furthermore, if $N=2$, we also have the following energy equality
\begin{equation}\label{energyeq12}
\|f\|_{L^2(\Omega)}^2=\sum_{j=0}^{+\infty}\left[\frac{1}{\lambda_j}\|u_j\|_{BV(\Omega)}+\|u_j\|_{L^2(\Omega)}^2\right].
\end{equation}
\end{teo}

\begin{proof}
We begin by showing that, for any positive $\lambda$ and $\tilde{f}\in L^2(\Omega)$,
there exists a unique minimizer to
\begin{equation}\label{meyermin}
\min \left\{\lambda\|\tilde{f}-w\|^2_{L^2(\Omega)}+\|w\|_{BV(\Omega)}:\ w\in L^2(\Omega)\right\}.
\end{equation}
This guarantees that the sequence $\{u_j\}_{j=0}^{+\infty}$ exists and it is uniquely defined.

The existence and uniqueness of a solution to \eqref{meyermin} is standard and we sketch here the argument for the convenience of the reader.

A minimizing sequence $\{w_n\}_{n\in\mathbb{N}}$ is clearly uniformly bounded in $L^2(\Omega)$, hence in $L^1(\Omega)$, and their $BV$ norms are uniformly bounded as well. Then, up to a subsequence, $w_n$ converges to $w\in L^2(\Omega)$
weakly in $L^2(\Omega)$ and strongly in $L^1(\Omega)$. The existence of the minimum is guaranteed by the lower semicontinuity of the functional with respect to this kind of convergence.
Uniqueness of the solution of the minimum problem follows from the fact that the functional is strictly convex.

The hierarchical decomposition follows again from Theorem~\ref{mainthm0} as in the proof of Theorem~\ref{hierarchythm} by replacing $\mathbb{R}^N$ with $\Omega$.

 Regarding the energy equality, the restriction to $N=2$ is again due to the fact that we need that
$BV(\Omega)\subset L^2(\Omega)$, with continuous immersion, and the proof relies on arguments developed in \cite{YMey}. In fact, we can define, for any $v\in L^2(\Omega)$, the following norm
$$\|v\|_{\ast}=\sup\left\{\int_{\Omega}vh:\ \|h\|_{BV(\Omega)}\leq 1\right\}.$$
Using the reasoning developed in \cite{YMey} for the $\mathbb{R}^2$ case, one can show one has the following characterization of the solution $u$ to \eqref{meyermin} and of $v=\tilde{f}-u$. If $\|\tilde{f}\|_{\ast}\leq\frac{1}{2\lambda}$, then $u=0$ and $v=\tilde{f}$. If 
$\|\tilde{f}\|_{\ast}\geq\frac{1}{2\lambda}$, then
$$\|v\|_{\ast}=\frac{1}{2\lambda}\quad\text{and}\quad\int_{\Omega}vu=\frac{1}{2\lambda}\|u\|_{BV(\Omega)}.$$ 
Note that the second equality is true in both cases. Moreover, it immediately follows that
\begin{equation}\label{normsquare}
\|\tilde{f}\|_{L^2(\Omega)}^2=\|u+v\|_{L^2(\Omega)}^2=\|u\|_{L^2(\Omega)}^2+\frac{1}{\lambda}\|u\|_{BV(\Omega)}+\|v\|_{L^2(\Omega)}^2.
\end{equation}
By induction, we conclude that for any 
$j_0\in \mathbb{N}$ we have
$$\|f\|_{L^2(\Omega)}^2=\left(\sum_{j=0}^{j_0}\|u_j\|_{L^2(\Omega)}^2+\frac{1}{\lambda_j}\|u_j\|_{BV(\Omega)}\right)
+\|v_{j_0}\|_{L^2(\Omega)}^2.$$
By \eqref{hy2}, we infer that $\|v_{j_0}\|_{L^2(\Omega)}^2$ goes to zero as $j_0$ goes to infinity, thus \eqref{energyeq12} is proved.
\end{proof}

\begin{oss}\label{wholespacehierarchy}
Let $\Omega\subset\mathbb{R}^N$, $N\geq 1$, be a bounded open set with Lipschitz boundary.
Then the hierarchical decomposition \eqref{hy2} of Theorem~\ref{hierarchythmbdd} still holds if we replace
the $\|\cdot\|_{BV(\Omega)}$ norm with the seminorm $|\cdot|_{BV(\Omega)}$ in the minimization problem \eqref{meyermin}
and, correspondingly, in those leading to the construction of the sequence $\{u_j\}_{j=0}^{+\infty}$ (which is still uniquely determined).
However, in this case \eqref{energyeq12} is not guaranteed.

In fact, it is enough to use in Theorem~\ref{mainthm0} as $|\cdot|$ the seminorm $|\cdot|_{BV(\Omega)}$ instead of the norm $\|\cdot\|_{BV(\Omega)}$. We note that the
regularization by the $BV$ seminorm is the one usually employed in the Rudin-Osher-Fatemi denoising model.
\end{oss}
 
\section{The multiscale approach in a topological group setting}\label{topgroupsubs}

In this section we extend the abstract results to a different setting, namely the one of topological groups. This extension is essential in order to apply the multiscale procedure to image registration as described in Subsection~\ref{LDDMMintrosub}.

Let $G$ be a group with multiplication denoted by $\cdot$.
The identity element is denoted $e$.
We endow $G$ with a left-invariant distance $d$ such that $G$ with that distance is a complete metric space. By left-invariance we mean
\begin{equation}
d(\psi_0,\psi_1)=d(\psi\cdot\psi_0,\psi\cdot\psi_1)\quad\text{for any }\psi_0,\ \psi_1,\ \psi\in G.
\end{equation}

Let us assume that there exists a notion of convergence on $G$, which for simplicity we call \emph{weak convergence} in $G$ and we denote with $\rightharpoonup$, satisfying the following properties:
\begin{enumerate}[a)]
\item if the weak limit exists, it is unique;\label{a}

\item the weak convergence is left-invariant, that is, if $\psi_n\rightharpoonup\psi_{\infty}$ as $n\to\infty$, then for any $\psi\in G$ we have that $\psi\cdot \psi_n \rightharpoonup\psi\cdot\psi_{\infty}$;\label{c}

\item if, as $n\to\infty$, $\psi_n$ converges to $\psi_{\infty}$ in the distance $d$, then $\psi_n\rightharpoonup\psi_{\infty}$.\label{b}
\item any bounded subset of $G$, with respect to the distance $d$, is relatively sequentially compact with respect to the weak convergence in $G$;\label{d}

\item the distance $d$ is sequentially lower semicontinuous on $G$, with respect to the weak convergence in $G$, in the following sense
$$d(\psi_{\infty},e)\leq \liminf_n d(\psi_n,e)\quad\text{if }\psi_n\rightharpoonup\psi_{\infty}\text{ as }n\to\infty.$$\label{e}
\end{enumerate}

Let us note that, by assumption~\ref{c}), assumption~\ref{e}) is equivalent to the statement that, 
for any $\psi\in G$,
$$d(\psi_{\infty},\psi)\leq \liminf_n d(\psi_n,\psi)\quad\text{if }\psi_n\rightharpoonup\psi_{\infty}\text{ as }n\to\infty.$$

An interesting example of this setting, as anticipated earlier in Subsection~\ref{LDDMMintrosub} and developed in the next Section~\ref{registrationsec}, is the image registration problem. Another simpler example is the following.

\begin{exam}\label{reflexiveexam}
We may consider $G=X$, where $X$ is a reflexive Banach space, with norm $\|\cdot\|=\|\cdot\|_X$. Then $G$ is a group with respect to the sum, that is $x_1\cdot x_2=x_1+x_2$ and, obviously, $e=0$, that is, $d(x,e)=\|x\|$.

As weak convergence in $G$ we consider the weak convergence in the reflexive Banach space $X$. Then all the previously stated properties are immediately satisfied.
\end{exam}

Let us assume that $E\subset G$ is sequentially closed with respect to the weak convergence in $G$. Let $Y$ be a metric space with distance $d_Y$. Given $\mathcal{N}:E\to Y$, and
$\hat{\mathcal{N}}\in Y$, we assume that the function $E\ni \psi\mapsto d(\hat{\mathcal{N}},\mathcal{N}(\psi))$ is sequentially lower semicontinuous with respect to the weak convergence in $G$.

Let us fix positive constants $\alpha$, $\beta$ and $\gamma$ and let $\lambda_n>0$ and $a_n\geq 0$ for any $n\in \mathbb{N}$. 
We assume that 
$a_n\leq a_{n-1}$ for any $n\geq 1$.
We begin with the following proposition.

\begin{prop}\label{existenceprop}
There exists a minimizer $\psi_0$ solving
\begin{equation}\label{00}
\min\left\{\left(\lambda_0[(d_Y(\hat{\mathcal{N}},\mathcal{N}(\psi)))^{\alpha}+a_0d(\psi,e)^{\gamma}]+d(\psi,e)^{\beta}\right):\ \psi\in G\right\}.
\end{equation}

By induction, there exists a minimizer $\psi_n$, $n\geq 1$, solving
\begin{equation}\label{11}
\min_{\psi\in G}\left(\lambda_n[(d_Y(\hat{\mathcal{N}},\mathcal{N}(\tilde{\psi}_{n-1}\cdot \psi)))^{\alpha}+a_nd(\tilde{\psi}_{n-1}\cdot \psi,e)^{\gamma}]+d(\psi,e)^{\beta}\right)
\end{equation}
where $\tilde{\psi}_0=\psi_0$ and
for any $n\geq 1$ we set by induction
$$\tilde{\psi}_n=\tilde{\psi}_{n-1}\cdot\psi_n.$$
\end{prop}

\begin{oss}
We have that the sequence $\{\psi_n\}_{n\in\mathbb{N}}$ exists, however we can not guarantee that it is uniquely determined.
\end{oss}

\begin{proof}
The existence of a minimizer for \eqref{00} is a simple consequence of the direct method of Calculus of Variations.

We show that there exists $\psi_1$, solution to \eqref{11} with $n=1$. Again we use the direct method. We consider $\psi^m\in G$, $m\in\mathbb{N}$, a minimizing sequence. Clearly, $\psi^m$ is bounded in $G$, thus it admits a weakly converging subsequence, which we do not relabel. Let $\psi_1$ be its weak limit. We need to prove that $\psi_1$ is a minimizer for \eqref{11}, with $n=1$.
By assumption~\ref{c}), we have that
$$d_Y(\hat{\mathcal{N}},\mathcal{N}(\psi_0\cdot \psi_1))\leq \liminf_md_Y(\hat{\mathcal{N}},\mathcal{N}(\psi_0\cdot \psi^m)).$$
It remains to show that
$$d(\psi_0\cdot \psi_1,e)\leq \liminf_m d(\psi_0\cdot \psi^m,e)$$
which is obviously true since, by left invariance of the distance and assumption~\ref{e}),
$$d(\psi_0\cdot \psi_1, e)=d(\psi_1, \psi_0^{-1})\leq \liminf_m d(\psi^m,\psi_0^{-1})=\liminf_m d(\psi_0\cdot \psi^m,e).$$

In a similar fashion, by induction, we prove the existence of 
$\psi_n$, $n\geq 1$.
\end{proof}

By taking $\psi=e$, we infer that
\begin{multline}
(d_Y(\hat{\mathcal{N}},\mathcal{N}(\tilde{\psi}_n)))^{\alpha}
\leq
(d_Y(\hat{\mathcal{N}},\mathcal{N}(\tilde{\psi}_n)))^{\alpha}
+a_nd(\tilde{\psi}_n,e)^{\gamma}\leq\\
(d_Y(\hat{\mathcal{N}},\mathcal{N}(\tilde{\psi}_{n-1})))^{\alpha}
+a_{n-1}d(\tilde{\psi}_{n-1},e)^{\gamma}\quad\text{for any }n\geq 1.
\end{multline}
Let us denote
$$\varepsilon_0=\left(\lim_n\left((d_Y(\hat{\mathcal{N}},\mathcal{N}(\tilde{\psi}_n)))^{\alpha}
+a_nd(\tilde{\psi}_n,e)^{\gamma}\right)\right)^{1/\alpha}
$$
and
$$\delta_0=\inf\{d_Y(\hat{\mathcal{N}},\mathcal{N}(\psi))
:\ \psi\in G\}
.$$
Clearly we have that $\varepsilon_0\geq \delta_0$.

In the following theorem we prove convergence in the space of images.

\begin{teo}\label{mainthmter1}
We assume that 
\eqref{secondcondcoeff} holds.
Then $\varepsilon_0= \delta_0$ and we also have
$$\lim_n d_Y(\hat{\mathcal{N}},\mathcal{N}(\tilde{\psi}_n))
=\delta_0.$$
\end{teo}

\begin{proof} The proof follows the one of Theorem~\ref{mainthmbis0}.
The only difference is that we replace the additive structure of the Banach space with the operation of the group. We sketch the proof to show how to handle such a difference.

We need to show that
$\varepsilon_0\leq \delta_0$.
By contradiction, let us assume that $\delta_0<\varepsilon_0$. Hence there exist $0<C_1<C_2<1$ 
such that
$$\delta_0<C_1^{1/\alpha}\varepsilon_0<C_2^{1/\alpha}\varepsilon_0<\varepsilon_0.$$
Therefore, there exists $\overline{\psi}\in G$ such that
$$d_Y(\hat{\mathcal{N}},\mathcal{N}(\overline{\psi}))
\leq C_1^{1/\alpha}\varepsilon_0\quad\text{and}\quad
d(\overline{\psi},e)<+\infty.$$

Then, recalling that $\psi_0=\tilde{\psi}_0$,
\begin{multline*}
\left(\lambda_0[(d_Y(\hat{\mathcal{N}},\mathcal{N}(\tilde{\psi}_0)))^{\alpha}
+a_0d(\tilde{\psi}_0,e)^{\gamma}]+d(\psi_0,e)^{\beta}\right)\leq\\
\left(\lambda_0[(d_Y(\hat{\mathcal{N}},\mathcal{N}(\overline{\psi})))^{\alpha}
+a_0d(\overline{\psi},e)^{\gamma}]+d(\overline{\psi},e)^{\beta}\right).
\end{multline*}
Analogously, for any $n\geq 1$, choosing $\psi=\tilde{\psi}_{n-1}^{-1}\cdot\overline{\psi}$, we have
\begin{multline}\label{crucialenew1}
\left(\lambda_n[
(d_Y(\hat{\mathcal{N}},\mathcal{N}(\tilde{\psi}_n)))^{\alpha}
+a_nd(\tilde{\psi}_n,e)^{\gamma}]+
d(\psi_n,e)^{\beta}\right)\leq\\
\left(\lambda_n[(d_Y(\hat{\mathcal{N}},\mathcal{N}(\overline{\psi})))^{\alpha}
+a_nd(\overline{\psi},e)^{\gamma}]+d(\tilde{\psi}_{n-1}^{-1}\cdot\overline{\psi},e)^{\beta}\right).
\end{multline}

For some $\overline{n}\geq 1$ and for any $n\geq \overline{n}$, we have that
$$a_nd(\overline{\psi},e)^{\gamma}\leq (C_2-C_1)\varepsilon_0^{\alpha}
$$
since $a_n$ goes to zero as $n\to \infty$.
Hence, for any $n\geq \overline{n}$, 
\begin{equation}\label{firstcrucialestimateter1}
\lambda_n(1-C_2)
[(d_Y(\hat{\mathcal{N}},\mathcal{N}(\tilde{\psi}_n)))^{\alpha}
+a_nd(\tilde{\psi}_n,e)^{\gamma}]+
d(\psi_n,e)^{\beta}\leq
d(\tilde{\psi}_{n-1}^{-1}\cdot\overline{\psi},e)^{\beta}.
\end{equation}
Therefore, for any $n\geq\overline{n}$,
we have that
$$d(\tilde{\psi}_n^{-1}\cdot\overline{\psi},e)=
d(\tilde{\psi}_{n-1}^{-1}\cdot\overline{\psi},\psi_n)\leq
d(\tilde{\psi}_{n-1}^{-1}\cdot\overline{\psi},e)+d(\psi_n,e)\leq
2d(\tilde{\psi}_{n-1}^{-1}\cdot\overline{\psi},e).$$

Then the proof may be concluded by adapting the arguments of the proof of Theorem~\ref{mainthm0}.
\end{proof}

Let us now consider the sequence $\{\tilde{\psi}_n\}_{n\in\mathbb{N}}$.
If the sequence $\{\tilde{\psi}_n\}_{n\in\mathbb{N}}$, or one of its subsequences, converges weakly in $G$ to some $\tilde{\psi}_{\infty}$, then
$\tilde{\psi}_{\infty}$ is a solution to the following minimization problem
\begin{equation}\label{minpbm1}
\min\{d_Y(\hat{\mathcal{N}},\mathcal{N}(\psi)):\ \psi\in E\}.
\end{equation}
Therefore, we have the following immediate remark.

\begin{oss}\label{limitproprem}
If $\{\tilde{\psi}_n\}_{n\in\mathbb{N}}$ has a bounded subsequence, then
\eqref{minpbm1} admits a solution. In particular,
if $\{\tilde{\psi}_n\}_{n\in\mathbb{N}}$ has a bounded subsequence, then there exists a further subsequence $\{\tilde{\psi}_{n_k}\}_{k\in\mathbb{N}}$
such that,
as $k\to\infty$, $\tilde{\psi}_{n_k}$ converges weakly to $\tilde{\psi}_{\infty}$ in $G$, where 
$\tilde{\psi}_{\infty}$ solves
\eqref{minpbm1}.
\end{oss}

Let us now investigate which conditions allow boundedness of $\{\tilde{\psi}_n\}_{n\in\mathbb{N}}$ or of one of its subsequences.
Clearly a necessary condition is that a solution to 
\eqref{minpbm1} does exist. We shall show that this is also a sufficient condition,
provided
\eqref{crucialcondition0} holds.

Let us assume that \eqref{minpbm1} admits a solution. This means that 
there exists $\hat{\psi}\in G$ such that
$$d_Y(\hat{\mathcal{N}},\mathcal{N}(\hat{\psi}))
=\delta_0=\min\{d_Y(\hat{\mathcal{N}},\mathcal{N}(\psi)):\ \psi\in G\}.
$$
It is important to note that this is equivalent to say that there exists $\hat{\psi}\in G$ such that
$\hat{\psi}$ solves the following minimization problem
\begin{equation}\label{minimalnew1}
\min\{d(\psi,e):\ \psi\in G\text{ and }
d_Y(\hat{\mathcal{N}},\mathcal{N}(\psi))
=\delta_0\}<+\infty.
\end{equation}

We call $\hat{G}$ the set of solutions to \eqref{minimalnew1}.
We note that $\hat{G}$
is closed and bounded with respect to the topology induced by the distance in $G$
and it is
sequentially compact with respect to the weak convergence in $G$.

Then we can prove the following result.

\begin{teo}\label{mainthmter2}
Let us assume that
\eqref{crucialcondition0} holds 
and that there exists a solution $\hat{\psi}$ of \eqref{minpbm1} or, equivalently, of\eqref{minimalnew1}.

Then, $\{\tilde{\psi}_n\}_{n\in\mathbb{N}}$ is bounded and, up to a subsequence, $\tilde{\psi}_n$ converges weakly to $\tilde{\psi}_{\infty}$ where $\tilde{\psi}_{\infty}$ is a \textnormal{(}possibly different from $\hat{\psi}$\textnormal{)} solution to \eqref{minimalnew1}, that is,
 $$d_Y(\hat{\mathcal{N}},\mathcal{N}(\tilde{\psi}_{\infty}))=\delta_0\quad\text{and}\quad d(\tilde{\psi}_{\infty},e)=d(\hat{\psi},e).$$ Moreover, we have that 
$$\lim_nd(\tilde{\psi}_n,e)= d(\hat{\psi},e).$$
\end{teo}

\begin{proof}
Assuming that there exists $\hat{\psi}$, solution to \eqref{minimalnew1},
first of all we note that $d(\psi_0,e)=d(\tilde{\psi}_0,e)\leq d(\hat{\psi},e)$.
Then we use \eqref{crucialenew1} with $\overline{\psi}$ replaced by $\hat{\psi}$ for any $n\geq 1$ and, by using the argument developed in the proof of Theorem~\ref{minimizercor0}, we obtain that
\begin{equation}\label{limsupresultnew1}
\limsup_nd(\tilde{\psi}_n,e)\leq d(\hat{\psi},e).
\end{equation}

By \eqref{limsupresultnew1} we have that, up to a subsequence, $\tilde{\psi}_n$ weakly converges to $\tilde{\psi}_{\infty}$ in $G$.
Then by 
Theorem~\ref{mainthmter1} and the lower semicontinuity properties of $\mathcal{N}$ with respect to weak convergence in $G$,
we infer that
$\tilde{\psi}_{\infty}$ satisfies
$d_Y(\hat{\mathcal{N}},\mathcal{N}(\tilde{\psi}_{\infty}))=\delta_0$. We also have that, by definition of $\hat{\psi}$,
$$d(\hat{\psi},e)\leq d(\tilde{\psi}_{\infty},e)\leq \limsup_nd(\tilde{\psi}_n,e)\leq
d(\hat{\psi},e).$$

By a similar reasoning it is fairly easy to conclude that
$$
\lim_nd(\tilde{\psi}_n,e)=d(\tilde{\psi}_{\infty},e)=d(\hat{\psi},e).
$$
The proof is concluded.
\end{proof}

We conclude this section by developing these results in the simple setting of Example~\ref{reflexiveexam}. The much more interesting application to the image registration problem will be studied in Section~\ref{registrationsec}.

\subsection{The special case of a reflexive Banach space}
We consider the setting of Example~\ref{reflexiveexam}, that is $G=X$, where $X$ is a reflexive Banach space, with norm $\|\cdot\|=\|\cdot\|_X$. Then $G$ is a group with respect to the sum, that is $x_1\cdot x_2=x_1+x_2$ and, obviously, $e=0$, that is, $d(x,e)=\|x\|$.

As weak convergence in $G$ we consider the weak convergence in the reflexive Banach space $X$ and we assume that $E\subset X$ is sequentially closed with respect to the weak topology of $X$.

As before, we consider $Y$ to be a metric space with distance $d_Y$, and we consider $\mathcal{N}:E\to Y$ such that
$E\ni \sigma\mapsto d(\hat{\mathcal{N}},\mathcal{N}(\sigma))$ is sequentially lower semicontinuous with respect to the weak convergence in $X$.

Let $\sigma_0$ be a solution to
$$\min\left\{\left(\lambda_0[d(\hat{\mathcal{N}},\mathcal{N}(\sigma))^{\alpha}+a_0\|\sigma\|^{\gamma}]+\|\sigma\|^{\beta}\right):\ \sigma\in E\right\},$$
and, by induction, let $\sigma_n$, $n\geq 1$, be a solution to
$$
\min\left\{\left(\lambda_n[d(\hat{\mathcal{N}},\mathcal{N}(\tilde{\sigma}_{n-1}+\sigma))^{\alpha}+a_n\|\tilde{\sigma}_{n-1}+\sigma\|^{\gamma}]+\|\sigma\|^{\beta}\right):\ \tilde{\sigma}_{n-1}+\sigma\in E\right\}
$$
where $\tilde{\sigma}_0=\sigma_0$ and for any $n\geq 1$ we denote
$$\tilde{\sigma}_{n}=\sum_{j=0}^n\sigma_j.$$

Let us note that \eqref{minpbm1} has a solution if and only if the following minimization problem also admits a solution 
\begin{equation}\label{minimal0}
\min\{ \|\sigma\|:\ \sigma\in E\text{ and }d(\hat{\mathcal{N}},\mathcal{N}(\sigma))=\delta_0\}<+\infty.
\end{equation}
We call $\hat{E}$ the set of solutions to \eqref{minimal0} and we note that it 
is closed and bounded with respect to the strong topology of $X$ and it
is sequentially compact with respect to the weak convergence in $X$.

Then we have the following result.

\begin{teo}\label{mainthmvar}
We assume that \eqref{secondcondcoeff} holds.
Then we have that
$$\lim_nd(\hat{\mathcal{N}},\mathcal{N}(\tilde{\sigma}_n))=\delta_0.$$

If we further assume that \eqref{crucialcondition0} holds and that
there exists a solution $\hat{\sigma}$ of \eqref{minpbm1} or, equivalently, of \eqref{minimal0}, 
then, up to a subsequence, $\tilde{\sigma}_n$ converges weakly to $\tilde{\sigma}_{\infty}$ where $\tilde{\sigma}_{\infty}$ is a \textnormal{(}possibly different from $\hat{\sigma}$\textnormal{)} solution to \eqref{minimal0}, that is $d(\hat{\mathcal{N}},\mathcal{N}(\tilde{\sigma}_{\infty}))=\delta_0$ and $\|\tilde{\sigma}_{\infty}\|=\|\hat{\sigma}\|$. Moreover, we have that 
$$\lim_n\|\tilde{\sigma}_n\|= \|\hat{\sigma}\|.$$

Finally, if $X$ is such that weak convergence and convergence of the norm imply strong convergence, for instance if $X$ is a Hilbert space, we have a stronger result. In fact,
then $\tilde{\sigma}_n$ converges, up to a subsequence, to $\tilde{\sigma}_{\infty}$ not only weakly but also strongly
and
\begin{equation}\label{distance0}
\lim_n \mathrm{dist}(\tilde{\sigma}_n,\hat{E})=0.
\end{equation}
\end{teo}

\section{The multiscale approach for image registration}\label{registrationsec}

Throughout this section we use the notation introduced in Subsection~\ref{LDDMMintrosub}.

We begin by stating and proving the following existence result, which follows from arguments of \cite{B-H}.

\begin{teo}\label{minsolutionthm0}
The minimization problem \eqref{000}, that is,
$$
\min\left\{\left(\lambda\|I_0\circ g^{-1} -I_1\|_{L^2(\Omega)}^{\alpha}+d_{\mathcal{H}}(g,e)^{\beta}\right):\ g\in G_{\mathcal{H}}\right\},
$$
admits a solution.
\end{teo}

We outline the strategy for proving Theorem~\ref{minsolutionthm0}, which follows the proof of Theorem~21 in 
\cite{B-H} and it is based on the properties of weak convergence as defined in Definition~\ref{weakconv}.
The first one is a continuity property, see for instance \cite{B-H} for a proof.

\begin{prop}\label{continuityweak}
Let $\mathcal{H}$ be an admissible Hilbert space. Let us fix $I_0$ and $I_1$ in $L^2(\Omega)$.

Let $\{g_n\}_{n\in\mathbb{N}}$ be a sequence in $G_{\mathcal{H}}$ and
assume that, as $n\to\infty$, $g_n$ weakly converges to $g_{\infty}$. Then $U_{I_0,I_1}(g_n)\to U_{I_0,I_1}(g)$.

Analogously, let $\{\psi_n\}_{n\in\mathbb{N}}$ be a sequence in $G_{\mathcal{H}}$.
Assume that, as $n\to\infty$, $\psi_n$ weakly converges to $\psi_{\infty}$. Then $\tilde{U}_{I_0,I_1}(\psi_n)\to \tilde{U}_{I_0,I_1}(\psi)$.
\end{prop}

The required compactness is provided by the following well-known result, see again \cite{B-H} for a sketch of the proof. 
\begin{prop}\label{contprop}
Let $\mathcal{H}$ be an admissible Hilbert space.

Let $\{u_n\}_{n\in\mathbb{N}}$ be a sequence in $L^2([0,1],\mathcal{H})$ and
$u\in L^2([0,1],\mathcal{H})$.
If, as $n\to\infty$, $u_n$ weakly converges to $u$ in $L^2([0,1],\mathcal{H})$, then 
$\varphi^{u_n}(1)$ weakly converges to $\varphi^u(1)$.
\end{prop}

As a consequence we can show the following.
\begin{lem}\label{distanceimpliespointwise}
Let $\{g_n\}_{n\in\mathbb{N}}$ be a sequence in $G_{\mathcal{H}}$.

If $\{g_n\}_{n\in\mathbb{N}}$ is bounded in $G_{\mathcal{H}}$, then
there exists a constant $C>0$ such that 
$\|Dg_n\|_{L^{\infty}}\leq C$ and $\|Dg^{-1}_n\|_{L^{\infty}}\leq C$ for any $n\in\mathbb{N}$. Furthermore,
there exists a subsequence $\{g_{n_k}\}_{k\in\mathbb{N}}$ and $g\in G_{\mathcal{H}}$ such that
$g_{n_k}$ converges weakly to $g$ as $k\to\infty$.

If either 
$\lim_nd_{\mathcal{H}}(g_n,g)=0$ or $\lim_nd_{\mathcal{H}}(g^{-1}_n,g^{-1})=0$, for some
$g\in G_{\mathcal{H}}$,
then, as $n\to\infty$, $g_n$ converges  weakly to $g$.
\end{lem}

\begin{proof} 
Let $u_n\in L^2([0,1],\mathcal{H})$, $n\in\mathbb{N}$, be such that
$g_n=\varphi^{u_n}(1)$ and $d_{\mathcal{H}}(g_n,e)=\|u_n\|$. If $\{g_n\}_{n\in\mathbb{N}}$ is bounded in $G_{\mathcal{H}}$, then $\{u_n\}_{n\in\mathbb{N}}$ is bounded in $L^2([0,1],\mathcal{H})$. Therefore, the uniform boundedness of $Dg_n$, $n\in\mathbb{N}$, follows from Theorem~8.9 in \cite{You}.
We also have a subsequence $\{u_{n_k}\}_{k\in\mathbb{N}}$
weakly converging in $L^2([0,1],\mathcal{H})$ and the first part of the 
claim follows by 
Proposition~\ref{contprop}.

Regarding the second part of the claim, assume that 
$\lim_nd_{\mathcal{H}}(g_n,g)=0$ and
let $w_n\in L^2([0,1],\mathcal{H})$, $n\in\mathbb{N}$, be such that
$g^{-1}\circ g_n=\varphi^{w_n}(1)$ and $\|w_n\|\to 0$ as $n\to\infty$.
Again by the previous proposition, we have that $g^{-1}\circ g_n$ and $(g_n)^{-1}\circ g$
converges to the identity $e$ uniformly on compact
subsets of $\overline{\Omega}$.
Therefore it is not difficult to prove that the claim follows.\end{proof}

Finally, a lower semicontinuity results is needed and it is included in the following.
\begin{lem}\label{assumptclem}
Let $\{g_n\}_{n\in\mathbb{N}}$ be a sequence in $G_{\mathcal{H}}$.
Assume that, as $n\to\infty$, $g_n$ weakly converges to $g_{\infty}$. Then
$$d_{\mathcal{H}}(g_{\infty},e)\leq\liminf_n d_{\mathcal{H}}(g_n,e).$$
\end{lem}

\begin{proof} 
Without loss of generality, we pick
$v_n$ in $L^2([0,1],\mathcal{H})$ such that $g_n=\varphi^{v_n}(1)$ and
$d_{\mathcal{H}}(g_n,e)=\|v_n\|$. Up to a subsequence, we may assume that
$$\liminf_n d_{\mathcal{H}}(g_n,e)=\lim_k d_{\mathcal{H}}(g_{n_k},e)=\lim_k \|v_{n_k}\|$$
and also that the subsequence
$\{v_{n_k}\}_{k\in\mathbb{N}}$  is weakly converging to $v_{\infty}$, as $k\to\infty$.
By Proposition~\ref{contprop}, we immediately infer that $g_{\infty}=\varphi^{v_{\infty}}(1)$
and we have
$$d_{\mathcal{H}}(g_{\infty},e)\leq \|v_{\infty}\|\leq \liminf_k\|v_{n_k}\|$$
thus the result is proved.\end{proof}

\begin{proof}[Proof of Theorem~\textnormal{\ref{minsolutionthm0}}.] It is convenient to prove the existence of a solution to \eqref{002}. It is immediate to verify, by the direct method in the Calculus of Variations, that
Proposition~\ref{continuityweak} and Lemmas~\ref{distanceimpliespointwise} and \ref{assumptclem} imply that the minimization problem \eqref{002} admits a solution. Therefore also \eqref{000} admits a solution and the theorem is proved.\end{proof}

In order to apply the procedure described in Section~\ref{topgroupsubs}, for $G=G_{\mathcal{H}}$ and $d=d_{\mathcal{H}}$, with  $\mathcal{H}$ admissible, we need to show that the weak convergence defined in Definition~\ref{weakconv} satisfies assumptions \ref{a})---\ref{e}) of the abstract weak convergence used in Section~\ref{topgroupsubs}.

It is easy to see that assumptions~\ref{a}) and \ref{c}) are satisfied. Then assumptions~\ref{b}) and \ref{d}) are satisfied as a consequence of Lemma~\ref{distanceimpliespointwise}. Finally, 
assumption~\ref{e}) follows from Lemma~\ref{assumptclem}.
Thus all the assumptions on weak convergence stated in the abstract setting are satisfied.

Before passing to the multiscale procedure, we wish to discuss an important and significant choice for the admissible $\mathcal{H}$.
We follow the results and the notation used in \cite{B-V}. We fix $s>N/2+1$ and we call
$$\mathcal{D}^s(\mathbb{R}^N)=\{g\in e+H^s(\mathbb{R}^N,\mathbb{R}^N):\
g\text{ is bijective and }g^{-1}\in e+H^s(\mathbb{R}^N,\mathbb{R}^N)\},$$
where $H^s$ is the usual Sobolev space. We have that $\mathcal{D}^s(\mathbb{R}^N)-e$
is an open subset of $H^s(\mathbb{R}^N,\mathbb{R}^N)$
and that the inversion operation is continuous, although not smooth, in $\mathcal{D}^s(\mathbb{R}^N)$.
We call $\mathcal{D}^s(\mathbb{R}^N)_0$ the connected component of the identity in $\mathcal{D}^s(\mathbb{R}^N)$. We further note that if we choose $\mathcal{H}=H^s(\mathbb{R}^N,\mathbb{R}^N)=H^s$, then $\mathcal{H}$ is admissible. 
Then we have the following result.

\begin{teo}\label{Hscase}
Let $\mathcal{H}=H^s(\mathbb{R}^N,\mathbb{R}^N)=H^s$, with $s>N/2+1$. Then
we have that
\begin{equation}\label{equivalence}
G_{H^s}=\mathcal{D}^s(\mathbb{R}^N)_0.
\end{equation}

Furthermore, let us consider a sequence $\{ g_n\}_{n\in\mathbb{N}}\subset G_{H^s}$ and $ g\in G_{H^s}$, and corresponding $u_n\in L^2([0,1],H^s)$ such that
$ g_n=\varphi^{u_n}(1)$, for any $n\in\mathbb{N}$, and $u\in L^2([0,1],H^s)$  such that $\varphi^u(1)= g$. The following properties are satisfied. 

If, as $n\to\infty$, $u_n$ converges to $u$ strongly in $L^2([0,1],H^s)$ then
$ g_n$ converges to $ g$ and $ g_n^{-1}$ converges to $ g^{-1}$, in both cases in $H^s$.

We have that $d_{H^s}( g_n, g)$ converges to $0$, as $n\to\infty$, if and only if
$ g_n^{-1}$ converges to $ g^{-1}$ in $H^s$.

Finally, if, as $n\to\infty$, $u_n$ converges to $u$ strongly in $L^2([0,1],H^s)$ then
$ g_n$ converges to $ g$ and $ g_n^{-1}$ converges to $ g^{-1}$ in the distance $d_{H^s}$.
\end{teo}

\begin{proof} The equality in \eqref{equivalence} is proved in \cite[Theorem~8.3]{B-V}.
For the convenience of the reader we sketch the proof.
One inclusion follows immediately by the continuity of the flow proved in \cite[Theorem~4.4]{B-V} that implies that $G_{H^s}\subset \mathcal{D}^s(\mathbb{R}^N)_0$. The reverse inclusion is proved as follows.
Since
$\mathcal{D}^s(\mathbb{R}^N)-e$
is an open subset of $H^s(\mathbb{R}^N,\mathbb{R}^N)$, we can take $U$ a convex neighborhood around $e$ that is contained in $\mathcal{D}^s(\mathbb{R}^N)_0$.
Then, for any $ g\in U$ we consider the path in $U$ connecting $e$ to $ g$ given by $\varphi(t)=(1-t)e+t g$, $0\leq t\leq 1$. We note that $\varphi(t)=\varphi^u(t)$ where
$u(t)=( g-e)\circ\varphi(t)^{-1}$. Using the continuity of the inversion operation in $\mathcal{D}^s(\mathbb{R}^N)$, it is not difficult to show that $[0,1]\ni t\mapsto \varphi(t)^{-1}$ is a continuous map in $H^s(\mathbb{R}^N,\mathbb{R}^N)$, therefore by 
Lemma~2.2 in \cite{B-V} we infer that $u\in L^{\infty}([0,1],H^s)$, therefore $U\subset G_{H^s}$. Since $G_{H^s}$ is a group, and again by \cite[Lemma~2.2]{B-V}, we conclude that 
$\mathcal{D}^s(\mathbb{R}^N)_0\subset G_{H^s}$.

The second part of the thesis follows immediately again by \cite[Theorem~4.4]{B-V}.

It remains to prove the equivalence between convergence of diffeomorphisms in the distance and convergence of their inverses in $H^s$. One implication is the following. Again we use arguments developed in \cite{B-V}. If $ g_n$ converges to $ g$ in the distance $d_{H^s}$, then, by \cite[Lemma~6.6]{B-V},
it is easy to show that $ g_n^{-1}$ converges to $ g^{-1}$ in $H^s$. Here we need to note that we use a left-invariant metric, instead in \cite{B-V} the metric used is the right-invariant $d_R$ defined in Remark~\ref{rightinvariantoss}.
Let us prove the other implication. If $ g_n^{-1}$ converges in $H^s$ to $ g^{-1}$, then,
applying again \cite[Lemma~2.2]{B-V}, we conclude that
$ g_n^{-1}\circ g$ converges to the identity $e$ in $H^s$. Since we have
$d_{H^s}( g_n, g)=d_{H^s}( g_n^{-1}\circ g,e)$, we conclude the proof provided the following claim holds true.
Let $h_n\in H^s(\mathbb{R}^N,\mathbb{R}^N)$, $n\in\mathbb{N}$, be such that $h_n$ converges to zero, as $n\to\infty$, in $H^s$. Then $g_n=e+h_n$ converges to $e$ in the distance $d_{H^s}$. Applying to $g_n$, for any $n\in\mathbb{N}$, the construction used for the diffeomorphism $ g$ in the proof of \eqref{equivalence}, we obtain that 
$g_n(t)=(1-t)e+tg_n=\varphi^{u_n}(t)$ where
$u_n(t)=(g_n-e)\circ g_n(t)^{-1}$. Then it is not difficult to show that
$\|u_n\|$ converges to $0$ as $n\to\infty$ and the proof is concluded.\end{proof}

\begin{oss}
If $\Omega$ is any open set contained in $\mathbb{R}^N$, $N\geq 1$, all properties stated above for $\mathbb{R}^N$ remain true provided we replace everywhere
$H^s(\mathbb{R}^N,\mathbb{R}^N)$ with the following space
$$H^s_{\Omega}=\{h\in H^s(\mathbb{R}^N,\mathbb{R}^N):\  h(x)=0\text{ for any }x\not\in \Omega\}.$$
Adopting a corresponding notation, \eqref{equivalence} now may be written as
$G_{H^s_{\Omega}}=\mathcal{D}^s(\Omega)_0$.

We finally note that $H^s_0(\Omega,\mathbb{R}^N)\subset H^s_{\Omega}$. These two spaces coincide under suitable assumptions, for instance if $s$ is an integer and $\Omega$ is regular enough, see \cite[Theorem~5.29]{Ada}.
\end{oss}

An interesting consequence of Theorem~\ref{Hscase} and of the previous remark is the following compactness result.

\begin{prop}\label{Hscomp}
Let $\Omega$ be a bounded open subset of $\mathbb{R}^N$, $N\geq 1$. Let $s_0$, $s_1$ be such that $N/2+1<s_0<s_1$.

Let us consider a sequence $\{ g_n\}_{n\in\mathbb{N}}\subset G_{H^{s_1}_{\Omega}}$. If $\{ g_n\}_{n\in\mathbb{N}}$ is bounded in $G_{H^{s_1}_{\Omega}}$, that is, there exists a constant $C$ such that
$d_{H^{s_1}_{\Omega}}( g_n,e)\leq C$ for any $n\in\mathbb{N}$, then there exists a subsequence $\{ g_{n_k}\}_{k\in\mathbb{N}}$ converging in $d_{H^{s_0}_{\Omega}}$.
\end{prop}

\begin{proof} By Lemma~6.6 in \cite{B-V}, we immediately infer that, for some constant $C_1$,
$$\| g_n-e\|_{H^{s_1}}\leq C_1\text{ and }\| g_n^{-1}-e\|_{H^{s_1}}\leq C_1\text{ for any }n\in\mathbb{N}.$$

Since $\Omega$ is bounded, we have that $H^{s_1}_{\Omega}$ is compactly embedded in $H^{s_0}_{\Omega}$, therefore there exists a subsequence
$\{ g_{n_k}\}_{k\in\mathbb{N}}$ and $ g$, $\tilde{ g}$ belonging to 
$e+H^{s_0}_{\Omega}$ such that
$$\lim_k\left(\| g_{n_k}- g\|_{H^{s_0}}+\| g_{n_k}^{-1}-\tilde{ g}\|_{H^{s_0}}\right)=0.$$
It follows easily that $ g\in \mathcal{D}^{s_0}(\Omega)_0$ and $ g^{-1}=\tilde{ g}$. Therefore, by Theorem~\ref{Hscase}, it is immediate to conclude that
$\lim_kd_{H^{s_0}_{\Omega}}( g_{n_k}, g)=0$ as well.\end{proof}

We now consider the multiscale procedure. Using the notation of Subsection~\ref{LDDMMintrosub}, 
we fix $I_0$ and $I_1$ in $L^2(\Omega)$ and consider the maps $U_{I_0,I_1}$ and $\tilde{U}_{I_0,I_1}$ defined in \eqref{Udefin}.
The fact that the functions $G_{\mathcal{H}}\ni g\mapsto U_{I_0,I_1}(g)$ and
$G_{\mathcal{H}}\ni \psi\mapsto \tilde{U}_{I_0,I_1}(\psi)$
are sequentially continuous with respect to the weak convergence in $G_{\mathcal{H}}$ is proved in Proposition~\ref{continuityweak}.

We begin by observing that, using Proposition~\ref{existenceprop},
one can show existence of a solution to \eqref{0001} and \eqref{111}, therefore the sequence 
$\{g_n\}_{n\in\mathbb{N}}$ exists, even if it is not uniquely determined. Clearly, 
sequences $\{\psi_n\}_{n\in\mathbb{N}}$, $\{\tilde{g}_n\}_{n\in\mathbb{N}}$ and $\{\tilde{\psi}_n\}_{n\in\mathbb{N}}$ are
determined by the sequence $\{g_n\}_{n\in\mathbb{N}}$.

Next we show how the results of Section~\ref{topgroupsubs} may be applied to our registration problem and what is their rephrasing in terms of the diffeomorphisms space $G_{\mathcal{H}}$. Let us note here that the abstract results apply to the formulation given by \eqref{0001bis} and \eqref{111bis}, that is, using $\psi=g^{-1}$ instead of $g$.

\begin{proof}[Proof of Theorem~\textnormal{\ref{mainthmter}}]
Immediate by Theorem~\ref{mainthmter1}.
\end{proof}

Let us now consider the sequence $\{\tilde{g}_n\}_{n\in\mathbb{N}}$ and let
$\{\tilde{u}_n\}_{n\in\mathbb{N}}$ be a sequence in $L^2([0,1],\mathcal{H})$ such that
$\tilde{g}_n=\varphi^{\tilde{u}_n}(1)$ and
$d_{\mathcal{H}}(\tilde{g}_n,e)=\|\tilde{u}_n\|$.
We recall that we are interested in finding conditions that allow boundedness of $\{\tilde{g}_n\}_{n\in\mathbb{N}}$ or of one of its subsequences.
We shall show that a necessary and, provided
\eqref{crucialcondition0} holds, sufficient condition is that a solution to 
\eqref{minpbmquater1} does exist. We recall that \eqref{minpbmquater1} admits a solution if and only if there exists $\hat{g}\in G_{\mathcal{H}}$ solving the
minimization problem \eqref{minimalnew0}.
We recall that $\hat{G}$ is the set of solutions to \eqref{minimalnew0}.
We call $\hat{E}$ the set of $u\in L^2([0,1],\mathcal{H})$ such that
$\varphi^u(1)\in\hat{G}$ and $\|u\|=d_{\mathcal{H}}(\hat{g},e)$. We note that $\hat{E}$
is closed and bounded with respect to the strong topology of $L^2([0,1],\mathcal{H})$ and it is
sequentially compact with respect to the weak convergence in $L^2([0,1],\mathcal{H})$.
We recall that $\hat{G}$
is closed and bounded with respect to the topology induced by the distance in $G_{\mathcal{H}}$
and it is
sequentially compact with respect to the weak convergence in $G_{\mathcal{H}}$.
The same topological properties are shared by $\hat{G}^{-1}=\{g^{-1}:\ g\in\hat{G}\}$.

We have the following lemma, an extension of Lemma~\ref{limitproprem1}, showing necessity.
\begin{lem}\label{limitproprem2}
If $\{\tilde{g}_n\}_{n\in\mathbb{N}}$
has a bounded subsequence, then there exists a further subsequence $\{\tilde{g}_{n_k}\}_{k\in\mathbb{N}}$ and $\tilde{g}_{\infty}\in G_{\mathcal{H}}$ such that $\{\tilde{g}_{n_k}\}_{k\in\mathbb{N}}$ converges weakly to $\tilde{g}_{\infty}$.

Furthermore, if $\{\tilde{g}_n\}_{n\in\mathbb{N}}$ has a bounded subsequence converging to $\tilde{g}_{\infty}\in G_{\mathcal{H}}$ either weakly or in $d_{\mathcal{H}}$, then 
$\tilde{g}_{\infty}\in \hat{G}$.
\end{lem}

\begin{proof}
The first part follows immediately from Lemma~\ref{distanceimpliespointwise}.

If $\{\tilde{g}_{n_k}\}_{k\in\mathbb{N}}$ is a bounded subsequence converging weakly to $\tilde{g}_{\infty}$, the
fact that $\tilde{g}_{\infty}$ is a solution to \eqref{minpbmquater1} follows 
by Theorem~\ref{mainthmter} and Proposition~\ref{continuityweak}.
Again by Lemma~\ref{distanceimpliespointwise}, the same conclusion holds if the convergence is in $d_{\mathcal{H}}$.

The fact that $\tilde{g}_{\infty}\in \hat{G}$ will be proved later on in the proof of Theorem~\ref{diffeoconv}.
\end{proof}

About existence and uniqueness of a solution to \eqref{minimalnew0}, we make the following remarks.
We note that, provided \eqref{minpbmquater1} has a solution, $\hat{G}$ may consists of more than one point, that is, we do not have uniqueness. Let us consider for example the square $Q=[-1,1]\times [-1,1]$ and let $Q_1$ be $Q$ rotated of an angle of $\pi/4$ in the clockwise sense. In order to register $Q$ with $Q_1$ there are two perfectly equivalent strategies, namely we rotate $Q_1$ of 
an angle of $\pi/4$ either again in the clockwise sense or in the counterclockwise sense.
On the other hand, also existence may not be guaranteed in general, that is \eqref{minpbmquater1}, and hence \eqref{minimalnew0}, may not have any solution as
Younes already noted in \cite{You}
and the following simple example shows.

\begin{exam}\label{noexistenceexample}
Fixed $\varepsilon$, $0<\varepsilon<1/2$,
let us define a cutoff function $\chi:\mathbb{R}\to\mathbb{R}$ such that
$\chi\in C^{\infty}(\mathbb{R})$, $\chi$ is nondecreasing, it is equal to $0$ on
$(-\infty,\varepsilon]$ and it is equal to $1$ on $[1-\varepsilon,+\infty)$.
For any $\lambda$, $0\leq \lambda< 1$, let us define $g_{\lambda}:\overline{B}_2\to \overline{B}_2$,
$\overline{B}_2$ denoting the closed ball of radius $2$ and center the origin in $\mathbb{R}^2$,
in the following way. For any $r$, $0\leq r\leq 2$, and any $\theta$, $0\leq\theta\leq 2\pi$,
$$g(r(\cos(\theta),\sin(\theta)))=f_{\lambda}(r,\theta)(\cos(\theta),\sin(\theta))$$
where
$$f_{\lambda}(r,\theta)=(1-\lambda) r+\lambda\frac{1+\cos(2\theta)}{2}\chi(r/(1-\lambda))r+
\lambda\frac{1-\cos(2\theta)}{2}\chi(r-1)r.$$
It is not difficult to show that $g_{\lambda}$ is a $C^{\infty}$ diffeomorphism from $\overline{B}_2$
onto itself. We note that
$g_0$ is the identity, that $g_{\lambda}(x_1,x_2)=(x_1,x_2)$ for any $\|(x_1,x_2)\|\geq 2-\varepsilon$, and that
$g_{\lambda}(x_1,x_2)=(1-\lambda) (x_1,x_2)$ for any $\|(x_1,x_2)\|\leq (1-\lambda) \varepsilon$.

Let $E_0=\overline{B}_1$ and
let $I_0=\chi_{E_0}$ and $I_{\lambda}=I_0\circ g_{\lambda}^{-1}$, $0\leq \lambda<1$. We note that $I_{\lambda}=\chi_{E_{\lambda}}$ where the set $E_{\lambda}=g_{\lambda}(E_0)$. It is not difficult to show that, as $\lambda\to 1^{-}$, $I_{\lambda}=\chi_{E_{\lambda}}$ converges pointwise, hence in $L^2(B_2)$, to $I_1=\chi_{E_1}$ where
$$E_1=\left\{(x_1,x_2)\in B_2:\ r\leq   \frac{1+\cos(2\theta)}{2} \right\}.$$
Therefore,
$$0=\inf\{\|I_0\circ g -I_1\|_{L^2(B_2)}:\ g\text{ is a $C^1$ diffeomorphism}\}.$$
However, $I_1$ is the characteristic function of an eight-shaped figure which may never be obtained from the characteristic function of a ball through a $C^1$ diffeomorphism, hence we may conclude that the minimum is not attained.
\end{exam}

Let us also point out here that also coercivity in the mild sense of 
Proposition~\ref{coerciveprop} may fail as the following example shows.
\begin{exam}\label{nocoerciveexample}
Let $I_0\in L^2(\mathbb{R}^2)$ be any radial symmetric image which is $0$ outside $B_1$. Let $h\in C_0^{\infty}(\mathbb{R})$ be an auxiliary function such that $0\leq h\leq 1$, $h(0)=1$ and $h(r)=0$ for any $|r|\geq 1/4$.

Then, for any $n\geq 1$, we consider the following diffeomorphism $\varphi_n$ such that for any $r$, $r\geq 0$, and any $\theta$, $0\leq\theta\leq 2\pi$,
$$\varphi_n(r(\cos(\theta),\sin(\theta)))=r(\cos(\theta+h_n(r)),\sin(\theta+h_n(r)))$$
where $h_n(r)=h(n(r-1/2))$.

Clearly we have that, for any $n\in\mathbb{N}$, $\varphi_n\in \mathcal{D}^s(\mathbb{R}^N)_0$ for any $s>N/2+1$. Moreover, $\varphi_n=e$ outside $B_1$ for any $n\in\mathbb{N}$.

Note that $I_0\circ \varphi^{-1}_n=I_0$ for any $n\geq 1$. On the other hand, since $\{\|D\varphi_n\|_{L^{\infty}}\}_{n\in\mathbb{N}}$ is unbounded, we may not have that $\{\varphi_n\}_{n\in\mathbb{N}}$ is bounded in $G_{H^s}$, for any $s>N/2+1$.
\end{exam}

We are now ready to prove our main result, which illustrates the convergence in the diffeomorphisms space. We note that this is an extended version of Theorem~\ref{diffeoconv2}.

\begin{teo}\label{diffeoconv}
Let us assume that
\eqref{crucialcondition0}
holds.

Then 
$\{\tilde{g}_n\}_{n\in\mathbb{N}}$ is bounded if and only if a solution to \eqref{minpbmquater1} exists.

In this case,
there exists a subsequence
$\{\tilde{g}_{n_k}\}_{k\in\mathbb{N}}$ and $\tilde{g}_{\infty}\in \hat{G}$ \textnormal{(}that is, $\tilde{g}_{\infty}$ is a solution to \eqref{minimalnew0}, possibly different from $\hat{g}$\textnormal{)}
such that,
as $k\to\infty$, $\tilde{g}_{n_k}$ converges to $\tilde{g}_{\infty}$ weakly, that is, in particular,
$\tilde{g}_{n_k}\to \tilde{g}_{\infty}$ and $(\tilde{g}_{n_k})^{-1}\to (\tilde{g}_{\infty})^{-1}$ uniformly on compact 
subsets of $\overline{\Omega}$.

Moreover, we have that
\begin{equation}\label{normconvv}
\lim_nd_{\mathcal{H}}(\tilde{g}_n,e)=d_{\mathcal{H}}(\tilde{g}_{\infty},e)=d_{\mathcal{H}}(\hat{g},e),
\end{equation}
and for any compact $Q\subset\overline{\Omega}$ we have
\begin{equation}\label{distconvv0}
\lim_n \mathrm{dist}_Q(\tilde{g}_n,\hat{G})=0
\end{equation}
where for any $g\in G_{\mathcal{H}}$ 
$$\mathrm{dist}_Q(g,\hat{G})=\inf\{
\|g-\hat{g}\|_{L^{\infty}(Q)}+\|g^{-1}-\hat{g}^{-1} \|_{L^{\infty}(Q)}:\ \hat{g}\in \hat{G}\}.$$

If we further have $\mathcal{H}=H^s_{\Omega}$, with $s>N/2+1$, then
there exists a subsequence
$\{\tilde{g}_{n_k}\}_{k\in\mathbb{N}}$ and $\tilde{g}_{\infty}\in \hat{G}$
such that,
as $k\to\infty$, $\tilde{g}_{n_k}\to \tilde{g}_{\infty}$ and $(\tilde{g}_{n_k})^{-1}\to (\tilde{g}_{\infty})^{-1}$
in $G_{H^s_{\Omega}}$. Moreover,
we have that
\begin{equation}\label{distconvv}
\lim_n \mathrm{dist}(\tilde{g}_n,\hat{G})=0\quad\text{and}\quad\lim_n \mathrm{dist}(\tilde{g}^{-1}_n,\hat{G}^{-1})=0
\end{equation}
where for any $g\in G_{H^s_{\Omega}}$, $\mathrm{dist}(g,\hat{G})=\inf\{
d_{H^s_{\Omega}}(g,\hat{g}):\ \hat{g}\in \hat{G}\}$.
\end{teo}

\begin{proof} The first part and \eqref{normconvv} follow by Theorem~\ref{mainthmter2},
using $\psi=g^{-1}$. Moreover, the argument in Theorem~\ref{mainthmter2} allows us to complete the proof of Lemma~\ref{limitproprem2} as well.

Since any subsequence of $\{\tilde{g}_n\}_{n\in\mathbb{N}}$ admits a further subsequence weakly converging to an element of $\hat{G}$, \eqref{distconvv0} immediately follows.

Let us also point out the following remark.
With the same notation as before, we have that, up to a subsequence, $\tilde{u}_n$ converges to $\tilde{u}_{\infty}$ not only
weakly in $L^2([0,1],\mathcal{H})$ but also strongly. In fact, we recall that
$$d_{\mathcal{H}}(\tilde{g}_{\infty},e)\leq \|\tilde{u}_{\infty}\|\leq
\limsup_n \|\tilde{u}_n\|=\limsup_n d_{\mathcal{H}}(\tilde{g}_n,e)\leq d_{\mathcal{H}}(\hat{g},e)=d_{\mathcal{H}}(\tilde{g}_{\infty},e).$$
Therefore we can conclude that
$$d_{\mathcal{H}}(\tilde{g}_{\infty},e)= \|\tilde{u}_{\infty}\|=\lim_n \|\tilde{u}_n\|.$$
From this last property we may then conclude that, still up to a subsequence, actually $\tilde{u}_n$ strongly converges to $\tilde{u}_{\infty}$ in $L^2([0,1],\mathcal{H})$.
We may also observe that
\begin{equation}\label{distancenew}
\lim_n \mathrm{dist}(\tilde{u}_n,\hat{E})=0
\end{equation}
where for any $u\in L^2([0,1],\mathcal{H})$, $\mathrm{dist}(u,\hat{E})=\inf\{\|u-\hat{u}\|:\ \hat{u}\in \hat{E}\}$.

For the case in which $\mathcal{H}=H^s_{\Omega}$, with $s>N/2+1$, 
the convergence in the distance $d_{H^s_{\Omega}}$ and \eqref{distconvv} follow by the previous remark and
Theorem~\ref{Hscase}.\end{proof}

It is still possible, however, that two different subsequences of $\{\tilde{g}_n\}_{n\in\mathbb{N}}$
converge to two different limits, that is, to two different
solutions of \eqref{minimalnew0}, as suggested by the counterexample in 
Subsection~\ref{wholesequenceex}.

On the other hand, if \eqref{minimalnew0} has a unique solution $\hat{g}$, then
the whole sequence $\tilde{g}_n$ converges to $\hat{g}$ weakly in $G_{\mathcal{H}}$.
Finally, if $\mathcal{H}=H^s_{\Omega}$, with $s>N/2+1$,
 the whole sequence $\tilde{g}_n$ converges to $\hat{g}$
also in the distance $d_{H^s_{\Omega}}$.

\section{The multiscale approach applied to the inverse conductivity problem}\label{inversesec}

In this section we consider the multiscale procedure applied to the Calder\'on problem. We follow the notation introduced in Subsection~\ref{Calderonsub}.

Let $\Omega\subset \mathbb{R}^N$, $N\geq 2$, be a bounded domain with Lipschitz boundary. 
Throughout this section we shall keep fixed 
positive constants $a,b$ with $a\leq b$.
We recall the classes of conductivity tensors $\mathcal{M}(a,b)$, $\mathcal{M}_{sym}(a,b)$ and $\mathcal{M}_{scal}(a,b)$.
Let us point out that the ellipticity condition \eqref{ell2} is equivalent to the more usual one given by the following. For $\sigma\in \mathbb{M}^{N\times N}(\mathbb{R})$, with $N\geq 2$, and positive constants $a,\tilde{b}$ with $a\leq \tilde{b}$, we require
\begin{equation}\label{ell1}
\left\{\begin{array}{ll}
\sigma\xi\cdot\xi\geq a\|\xi\|^2&\text{ for any }\xi\in\mathbb{R}^N\\
\|\sigma\|\leq \tilde{b}.
\end{array}\right.
\end{equation}
We note that if $\sigma$ satisfies \eqref{ell2} with constants $a$ and $b$, then it also satisfies \eqref{ell1} with constants $a$ and $\tilde{b}=b$. On the other hand, if $\sigma$ satisfies \eqref{ell1} with constants $a$ and $\tilde{b}$, then it also satisfies \eqref{ell2} with constants $a$ and $b=\tilde{b}^2/a$.
If $\sigma$ is symmetric then \eqref{ell1} and \eqref{ell2} are equivalent and both correspond to the condition \eqref{ell2sym}, with $b=\tilde{b}$.
One can also define the class of conductivity tensors $\tilde{\mathcal{M}}(a,\tilde{b})$ as the set of $\sigma\in L^{\infty}(\Omega,\mathbb{M}^{N\times N}(\mathbb{R}))$ such that, for almost any $x\in\Omega$, $\sigma(x)$ satisfies \eqref{ell1},
with constants $a,\tilde{b}$. Obviously we have $\mathcal{M}(a,b)\subset \tilde{\mathcal{M}}(a,b)$
and $\tilde{\mathcal{M}}(a,\tilde{b})\subset \mathcal{M}(a,\tilde{b}^2/a)$.

We note that all these classes are closed with respect to the $L^p$ metric, for any $p$, $1\leq p\leq +\infty$, where for any conductivity tensor $\sigma$ in $\Omega$
$$\|\sigma\|_{L^p(\Omega)}=\|(\|\sigma\|)\|_{L^p(\Omega)}.$$
The following result is needed in order to check that the density condition \ref{BVass3}) of Subsection~\ref{abstract1subsec} is satisfied.
\begin{lem}\label{densitylemma}
For positive constants $a$ and $b$, with $a\leq b$,
let $\mathcal{M}$ be any of the classes $\mathcal{M}(a,b)$, $\tilde{\mathcal{M}}(a,b)$,
$\mathcal{M}_{sym}(a,b)$ or $\mathcal{M}_{scal}(a,b)$. Then $C^{\infty}(\overline{\Omega},\mathbb{M}^{N\times N}(\mathbb{R}))\cap\mathcal{M}$ is dense in $\mathcal{M}$ with respect to the $L^1(\Omega,\mathbb{M}^{N\times N}(\mathbb{R}))$-norm.
\end{lem}

\begin{proof}
The key point is that, in all these cases, $\mathcal{M}$ is convex.
Then the proof follows using a standard approximation of the identity.
\end{proof}

For any $p$, $1<p<+\infty$, let $p'$ be its conjugate exponent, that is $1/p+1/p'=1$. We call
$W^{1-1/p,p}(\partial \Omega)$ the space of traces of
$W^{1,p}(\Omega)$ functions on $\partial \Omega$ and let us recall that $W^{1-1/p,p}(\partial \Omega)\subset L^p(\partial\Omega)$, with continuous immersion.

In Subsection~\ref{Calderonsub}, we have already defined, for any conductivity tensor $\sigma$ in $\Omega$, its corresponding Neumann-to-Dirichlet map
$\mathcal{N}(\sigma)$. In an analogous way we define the Dirichlet-to-Neumann map. 

For a conductivity tensor $\sigma$ in $\Omega$, its corresponding Dirichlet-to-Neumann map is defined by
$$\Lambda(\sigma):W^{1/2,2}(\partial\Omega)\to W^{-1/2,2}(\partial\Omega)$$
where for each $\varphi\in W^{1/2,2}(\partial\Omega)$,
$$\Lambda(\sigma)(\varphi)[\psi]=\int_{\Omega}\sigma\nabla u\cdot\nabla \tilde{\psi} \quad\text{for any }\psi\in W^{1/2,2}(\partial\Omega)$$
with $u$ the solution to
\begin{equation}\label{Dirichletproblem1}
\left\{\begin{array}{ll}
-\mathrm{div}(\sigma\nabla u)=0&\text{in }\Omega\\
u=\varphi&\text{on }\partial\Omega
\end{array}\right. 
\end{equation}
and $\tilde{\psi}\in W^{1,2}(\Omega)$ such that
$\tilde{\psi}=\psi$ on $\partial\Omega$ in the trace sense.
Then $\Lambda(\sigma)$ is a well-defined bounded linear operator. Moreover, provided $\sigma\in\mathcal{M}(a,b)$, its norm is bounded by a constant depending on $N$, $\Omega$, $a$ and $b$ only. 
Let us note that, actually, we have $\Lambda(\sigma):W^{1/2,2}(\partial\Omega)\to W^{-1/2,2}_{\ast}(\partial\Omega)$
and that $\mathcal{N}(\sigma)$ is the inverse of $\Lambda(\sigma)|_{W^{1/2,2}_{\ast}(\partial\Omega)}$.

Our forward operators are 
$$
\Lambda:\mathcal{M}(a,b)
\to \mathcal{L}(W^{1/2,2}(\partial \Omega),W^{-1/2,2}_{\ast}(\partial \Omega))$$ or
$$\mathcal{N}:
\mathcal{M}(a,b)\to \mathcal{L}(W^{-1/2,2}_{\ast}(\partial \Omega),W^{1/2,2}_{\ast}(\partial \Omega)).$$

We recall that the inverse conductivity problem consists in determining an unknown conductivity $\sigma$ by performing (all possible) electrostatic measurements at the boundary of voltage and current type, that is, by measuring either its corresponding Dirichlet-to-Neumann map $\Lambda(\sigma)$ or its corresponding Neumann-to-Dirichlet map $\mathcal{N}(\sigma)$.

In order to apply our multiscale results to the inverse conductivity problem, we pick
$X=L^1(\Omega,\mathbb{M}^{N\times N}(\mathbb{R}))$, with its natural norm. We may take as the subset $E$ any of the following classes $\mathcal{M}(a,b)$,
$\mathcal{M}_{sym}(a,b)$ or
$\mathcal{M}_{scal}(a,b)$.

We need to check the conditions that allow us to use our abstract results. First of all we investigate the continuity properties of our forward operators. Let $B_1$ and $B_2$ be two Banach spaces such that $B_1\subset W^{1/2,2}(\partial\Omega)$
and $W^{-1/2,2}_{\ast}(\partial\Omega)\subset B_2$, with continuous immersions.
Moreover, let $\tilde{B}_1$ and $\tilde{B}_2$ be two Banach spaces such that $\tilde{B}_1\subset W^{-1/2,2}_{\ast}(\partial\Omega)$
and $W^{1/2,2}_{\ast}(\partial\Omega)\subset \tilde{B}_2$, with continuous immersions.

In the Dirichlet-to-Neumann case we let $Y=\mathcal{L}(B_1,B_2)$, with the distance $d$ induced by its norm, and let $\Lambda:E\to Y$. Furthermore, $\hat{\Lambda}\in Y$ is the measured Dirichlet-to-Neumann map.

In the Neumann-to-Dirichlet case we let $Y=\mathcal{L}(\tilde{B}_1,\tilde{B}_2)$, with the distance $d$ induced by its norm, and let $\mathcal{N}:E\to Y$. Furthermore, $\hat{\mathcal{N}}\in Y$ is the measured 
Neumann-to-Dirichlet map.

In \cite{Ron15} the following  lower semicontinuity result is proved.

\begin{prop}\label{Hconv}
Under the previous assumptions, let us consider a sequence of conductivity tensors $\{\sigma_n\}_{n\in\mathbb{N}}\subset E$ and a conductivity tensor $\sigma$ in the same set.

If, as $n\to\infty$, $\sigma_n$ converges to $\sigma$ strongly in $X$ or in the $H$-convergence sense, then
$$\|\hat{\mathcal{N}}-\mathcal{N}(\sigma)\|_Y\leq\liminf_n \|\hat{\mathcal{N}}-\mathcal{N}(\sigma_n)\|_Y.$$

Furthermore, if $E$ is $\mathcal{M}(a,b)$ or
$\mathcal{M}_{sym}(a,b)$, then
the following minimum problem admits a solution
\begin{equation}\label{minpbmcond}
\min\{\|\hat{\mathcal{N}}-\mathcal{N}(\sigma)\|_Y:\ \sigma\in E\}.
\end{equation}

The same result holds for $\Lambda$ and $\hat{\Lambda}$.
\end{prop}

The existence of a solution to \eqref{minpbmcond} is due to the fact that
$\mathcal{M}(a,b)$ and
$\mathcal{M}_{sym}(a,b)$ are (sequentially) compact with respect to $H$-convergence.
In order to have continuity, we need to consider suitable choices of the spaces $B_1$, $B_2$ and $\tilde{B}_1$, $\tilde{B}_2$. Namely we have the following result, see \cite{Ron15}.

\begin{prop}\label{Hconvcont}
Under the previous assumptions and notation, let $E=\mathcal{M}(a,b)$ and consider the distance $d$ on $Y$, induced by its norm.

There exists $Q_1>2$, depending on $N$, $\Omega$, $a$ and $b$ only,
such that
the following holds for any 
$2<p<Q_1$.

In the Dirichlet-to-Neumann case, we assume that
$B_1\subset W^{1-1/p,p}(\partial \Omega)$,
with continuous immersion. Then $\Lambda$ is continuous with respect to the strong convergence in $X$ and the distance $d$ on $Y$.

In the Neumann-to-Dirichlet case, we assume that
$\tilde{B}_1$ is contained, with continuous immersion, in the subspace of $g$ belonging to the
dual of $W^{1-1/p',p'}(\partial \Omega)$ such that $\langle g,1\rangle =1$.
Then $\mathcal{N}$ is continuous with respect to the strong convergence in $X$ and the distance $d$ on $Y$.
\end{prop}

A particularly interesting case for Neumann-to-Dirichlet maps is provided in the following.

\begin{oss}\label{L2L2rem}
We can choose $\tilde{B}_1=\tilde{B}_2=L^2_{\ast}(\partial\Omega)$ since 
$L^2(\partial\Omega)$ is contained in the dual of $W^{1-1/p',p'}(\partial \Omega)$ for some $p$, $2<p<Q_1$, with $p$ close enough to $2$,
and $W^{1/2,2}_{\ast}(\partial\Omega)\subset L^2_{\ast}(\partial\Omega)$, with continuous immersions.
\end{oss}
We illustrate the applicability of our abstract results for the inverse conductivity problem.
First of all, we need to consider as $Y$ and $d=d_Y$ those satisfying the assumptions of Proposition~\ref{Hconvcont}.
As noted in Subsection~\ref{Calderonsub},
there are several possible choices for $|\cdot|$. Here we choose as $|\cdot|$ either $|\cdot|_{BV(\Omega)}$ or $\|\cdot\|_{BV(\Omega)}$.

Then the results of Subsections~\ref{abstract1subsec} and \ref{abstractsecondsub} yield the following, with exactly the same notation.

\begin{teo}\label{mainthmcond}
Assume that $Y$ satisfies the hypotheses of Proposition~\textnormal{\ref{Hconvcont}} 
and that $|\cdot|$ is either $|\cdot|_{BV(\Omega)}$ or $\|\cdot\|_{BV(\Omega)}$.

If \eqref{secondcondcoeff} is satisfied, then for the multiscale sequence
$\{\tilde{\sigma}_n\}_{n\in\mathbb{N}}$ defined by \eqref{tildetight} from the sequence $\{\sigma_n\}_{n\in\mathbb{N}}$ obtained from \eqref{regularizedpbmbis} and \eqref{tightconstr} we have that
$$\lim_n\|\hat{\mathcal{N}}-\mathcal{N}(\tilde{\sigma}_n)\|_Y=\delta_0=\inf\{\|\hat{\mathcal{N}}-\mathcal{N}(\sigma)\|_Y:\ \sigma\in E\}.$$

Furthermore, if $E=\mathcal{M}(a,b)$ or 
$E=\mathcal{M}_{sym}(a,b)$, then, up to a subsequence, $\tilde{\sigma}_n$ $H$-converges to $\tilde{\sigma}_{\infty}\in E$, where $\tilde{\sigma}_{\infty}$ solves \eqref{minpbmcond}.

The same result holds for $\Lambda$ and $\hat{\Lambda}$.
\end{teo}

Let us observe that $\delta_0\geq 0$ corresponds to the noise level of our measurements and that Theorem~\ref{mainthmcond} contains as special cases, namely taking $a_n=0$ for any $n\in\mathbb{N}$, Theorem~\ref{mainthmcond0} and Corollary~\ref{corollary1}.

If we wish to have a stronger convergence than $H$-convergence, we need further assumptions, and apply the tighter multiscale construction. We assume that there exists $\hat{\sigma}\in E$ such that
\begin{equation}\label{anothermincond}
\|\hat{\mathcal{N}}-\mathcal{N}(\hat{\sigma})\|_Y
=\delta_0=\min\{\|\hat{\mathcal{N}}-\mathcal{N}(\sigma)\|_Y:\ \sigma\in E\}\quad\text{and}\quad |\hat{\sigma}|<+\infty.
\end{equation} 
As before, we may assume that $\hat{\sigma}$ solves the following minimization problem
\begin{equation}\label{minimalcond}
\min\{ |\sigma|:\ \sigma\in E\text{ and }\|\hat{\mathcal{N}}-\mathcal{N}(\sigma)\|_Y=\delta_0\}<+\infty.
\end{equation}
We call $\hat{E}$ the set of solutions of \eqref{minimalcond} and we note that $\hat{E}$ is compact in $X$.

\begin{teo}\label{mainthmcond2}
Assume that $Y$ satisfies the hypotheses of Proposition~\textnormal{\ref{Hconvcont}} 
and that $|\cdot|$ is either $|\cdot|_{BV(\Omega)}$ or $\|\cdot\|_{BV(\Omega)}$. We further assume that \eqref{crucialcondition0} holds and 
that there exists a solution $\hat{\sigma}$ of\eqref{minimalcond}.

 Consider the sequence $\{\tilde{\sigma}_n\}_{n\in\mathbb{N}}$ defined by \eqref{tildetight} from the sequence $\{\sigma_n\}_{n\in\mathbb{N}}$ obtained from \eqref{regularizedpbmbis} and \eqref{tightconstr}.

Then, up to a subsequence, $\tilde{\sigma}_n$ converges to $\tilde{\sigma}_{\infty}$ strongly in $X$, where $\tilde{\sigma}_{\infty}$ is a \textnormal{(}possibly different from $\hat{\sigma}$\textnormal{)} solution to \eqref{minimalcond}, that is $\|\hat{\mathcal{N}}-\mathcal{N}(\tilde{\sigma}_{\infty})\|_Y=\delta_0$ and $|\tilde{\sigma}_{\infty}|=|\hat{\sigma}|$. Moreover, we have that 
$$\lim_n|\tilde{\sigma}_n|= |\hat{\sigma}|$$
and
\begin{equation}\label{distancecond1}
\lim_n \mathrm{dist}(\tilde{\sigma}_n,\hat{E})=0.
\end{equation}

The same result holds for $\Lambda$ and $\hat{\Lambda}$, with the obvious changes, for example $\hat{E}$ in this case is the set of solutions of
$$\min\{ |\sigma|:\ \sigma\in E\text{ and }\|\hat{\Lambda}-\Lambda(\sigma)\|_Y=\delta_0\}<+\infty.$$
\end{teo}

Theorem~\ref{mainthmcond2} contains as a special case Theorem~\ref{mainthmcond20}.

\appendix

\section{Optimality of the abstract results}\label{counterexsub}

In this appendix we shall present arguments and examples that show the optimality of our abstract results.
Throughout we use the assumptions and the notation of Subsection~\ref{abstractsecondsub}.
 
\subsection{The single-step regularization}
\label{abstractonestep}
 
For any $\lambda>0$, let $\sigma_{\lambda}$ be a solution to
$$\min\left\{\left(\lambda[d(\hat{\mathcal{N}},\mathcal{N}(\sigma))^{\alpha}]+|\sigma|^{\beta}\right):\ \sigma\in E\right\}.$$
Our assumptions guarantee that at least one minimizer $\sigma_{\lambda}$ does exist.
Then we have the following result.

\begin{prop}\label{onestepcase}
We have that
$$\lim_{\lambda\to+\infty}d(\hat{\mathcal{N}},\mathcal{N}(\sigma_{\lambda}))=\delta_0.$$
\end{prop}

\begin{proof}
By contradiction, assume that there exists a sequence $\{\lambda_n\}_{n\in\mathbb{N}}$ of positive numbers such that $\lim_n\lambda_n=+\infty$ and such that
$\lim_{n}d(\hat{\mathcal{N}},\mathcal{N}(\sigma_{\lambda_n}))=\varepsilon_0>\delta_0$.

Hence there exist $0<C_1<1$ and $\overline{\sigma}\in E$ such that \eqref{C1def} and \eqref{sigmabardef} hold.
We have that for any $n\in\mathbb{N}$
$$d(\hat{\mathcal{N}},\mathcal{N}(\sigma_{\lambda_n}))^{\alpha}\leq
d(\hat{\mathcal{N}},\mathcal{N}(\sigma_{\lambda_n}))^{\alpha}+\frac{1}{\lambda_n}|\sigma_{\lambda_n}|^{\beta}\leq
d(\hat{\mathcal{N}},\mathcal{N}(\overline{\sigma}))^{\alpha}+\frac{1}{\lambda_n}|\overline{\sigma}|^{\beta}.$$
Since, as $n\to\infty$, the left hand side converges to $\varepsilon_0^{\alpha}$ and the right hand side converges to 
$d(\hat{\mathcal{N}},\mathcal{N}(\overline{\sigma}))^{\alpha}\leq C_1\varepsilon_0^{\alpha}<\varepsilon_0^{\alpha}$, we have a contradiction.\end{proof}

It is clear that for any sequence $\{\lambda_n\}_{n\in\mathbb{N}}$ of positive numbers such that $\lim_n\lambda_n=+\infty$ and such that
$\sigma_{\lambda_n}$ converges to some $\sigma_{\infty}$, then
$\sigma_{\infty}$ is a solution to \eqref{minpbm0}.

Let us now investigate which conditions allow convergence of $\{\sigma_{\lambda}\}_{\lambda>0}$, as $\lambda\to+\infty$, at least up to subsequences. Namely, we consider a sequence $\{\lambda_n\}_{n\in\mathbb{N}}$ of positive numbers such that
$\lim_n\lambda_n=+\infty$ and we ask whether $\{\sigma_{\lambda_n}\}_{n\in\mathbb{N}}$, or one of its subsequences, converges.
 Clearly a necessary condition is that a solution to \eqref{minpbm0} does exist.

Let us make the following stronger assumption. We assume that there exists $\hat{\sigma}\in E$ solving 
\eqref{minimal00}.
Let us recall that $\hat{E}$ is the set of solutions of \eqref{minimal00} and that
$\hat{E}$ is sequentially compact in $X$.

First of all we note that $|\sigma_{\lambda}|\leq |\hat{\sigma}|$ for any $\lambda>0$.
Hence, by the lower semicontinuity properties of $|\cdot|$, the following result is immediate.

\begin{prop}\label{singlestepprop}
Let us assume that there exists a solution $\hat{\sigma}$ of \eqref{minimal00}.

Let us consider a sequence $\{\lambda_n\}_{n\in\mathbb{N}}$ of positive numbers such that
$\lim_n\lambda_n=+\infty$.

Then, up to a subsequence, $\{\sigma_{\lambda_n}\}_{n\in\mathbb{N}}$ converges to $\sigma_{\infty}$ where $\sigma_{\infty}$ is a \textnormal{(}possibly different from $\hat{\sigma}$\textnormal{)} solution to \eqref{minimal00}, that is, $d(\hat{\mathcal{N}},\mathcal{N}(\sigma_{\infty}))=\delta_0$ and $|\sigma_{\infty}|=|\hat{\sigma}|$. Furthermore, we have that
$$\lim_{\lambda\to+\infty}|\sigma_{\lambda}|= |\hat{\sigma}|$$
and 
\begin{equation}\label{distance00}
\lim_{\lambda\to+\infty} \mathrm{dist}(\sigma_{\lambda},\hat{E})=0.
\end{equation}
\end{prop}

It can be immediately noted that, using our multiscale procedure, we do not lose any of the convergence properties that hold for the single-step regularization when we let the regularization parameter go to $+\infty$.

In the next example we wish to prove the following remarks. Let us assume that a solution to \eqref{minpbm0} does exist but 
\eqref{anothermin00} does not hold, that is, \eqref{minimal00} does not have a solution.
Then we may have that $\sigma_{\lambda}$ does not converge, not even up to subsequences, to a solution to \eqref{minpbm0}. This suggest that even for the multiscale scale, the validity of \eqref{anothermin00}, although a strong requirement, is a necessary assumption to guarantee convergence of the sequence $\{\tilde{\sigma}_n\}_{n\in\mathbb{N}}$ or of one of its subsequences.

\begin{exam}\label{L2minexample}
We shall show two different cases. In the first one we obtain that $\sigma_{\lambda}$
is bounded but does not converge, not even up to subsequences. In the second one we have that $\sigma_{\lambda}$ is such that $\lim_{\lambda\to+\infty}\|\sigma_{\lambda}\|=+\infty$.

Let us consider the following common framework for these two examples. Let
$$l_2=\left\{a=\{a_n\}_{n\geq 1}:\ a_n\in\mathbb{R}\text{ for any }n\geq 1\text{ and }\sum_{n=1}^{+\infty}a_n^2<+\infty\right\}$$
which is a Hilbert space with the scalar product
$$\langle a,b\rangle=\sum_{n=1}^{+\infty}a_nb_n\quad\text{for any }a,b\in l_2.$$
Therefore we define, for any $a\in l_2$,
$$\|a\|=\left(\sum_{n=1}^{+\infty}a_n^2\right)^{1/2}\quad\text{and}\quad |a|=\left(\sum_{n=1}^{+\infty}(na_n)^2\right)^{1/2}.$$
It is easy to show that $|\cdot|$ satisfies the assumptions stated above.

Let us fix $b\in l_2$ such that $|b|=+\infty$.
Then we shall define two different versions of a continuous function $\mathcal{N}:l_2\to \mathbb{R}$
such that $\mathcal{N}(b)=0$ and $\mathcal{N}(a)>0$ for any $a\in l_2$, $a\neq b$. 
We fix the data $\alpha=2$, $\beta=2$, $X=E=l_2$ and $\hat{\mathcal{N}}=0=\mathcal{N}(b)$.

In both cases we need the following construction.
For any $r>0$ let us consider the following minimization problem
$$f(r)=\min\{|a|:\ \|a-b\|=r\}.$$
It is easy to show that such a minimization problem has a solution.
It is a straightforward, even if long, computation to show that  
$f:(0,\|b\|]\to\mathbb{R}$ is a nonnegative, continuous, strictly decreasing function such that $f(\|b\|)=0$ and
$\lim_{r\to 0^+}f(r)=+\infty$.

Then we call $A_1=\left\{a\in l_2\text{ such that }\|a-b\|\leq \|b\|/2\right\}$ and we set
$$\mathcal{N}(a)=1/f(r)\quad\text{for any }a\in A_1\text{ such that }\|a-b\|=r,$$ 
thus meaning also that $\mathcal{N}(b)=0$.
We can show that, fixed $\lambda>0$,
$$\min\left\{\left(|\mathcal{N}(a)|^2+|a|^2/\lambda\right):\ a\in A_1\right\}$$
is equivalent to solve
$$\min\left\{\left(\frac{1}{x} +\frac{x}{\lambda}\right) :\ x\geq C\right\}=
\left\{\begin{array}{ll}
2/\sqrt {\lambda} & \text{if }\sqrt{\lambda}\geq C\\
1/C+C/\lambda & \text{if }\sqrt{\lambda}\leq C\\
\end{array}\right.
$$
where $C=(f(\|b\|/2))^2$.

The definition of $\mathcal{N}$ outside $A_1$ is different for the two examples. Let us begin with the first one.

\smallskip

\noindent
\textbf{First version.}
Let us define an auxiliary function $\tilde{\mathcal{N}}:l_2\to \mathbb{R}$ as follows
$$\tilde{\mathcal{N}}(a)=\sum_{n=1}^{+\infty}\frac{|a_n|}{n^2}\quad\text{for any }a=\{a_n\}_{n\geq 1}\in l_2.$$
Clearly $\tilde{\mathcal{N}}$ is continuous and it is positive except for $a=0$. Fixed a positive constant $r$, let us consider 
$$c(\lambda,r)=\min\left\{\left(|\tilde{\mathcal{N}}(a)|^2+|a|^2/\lambda\right):\ a\in l_2\text{ such that }\|a\|=r\right\}.$$
Let us assume that for some $n\geq 3$, we have $n^4\leq \lambda <(n+1)^4$, that is
$\sqrt{\lambda}=c_1^2n^2+c_2^2(n+1)^2$ with $c_1$ and $c_2$ nonnegative numbers such that $c_1^2+c_2^2=1$.
Take $a\in l^2$ such that $a_i=0$ for any $i$ different from $n$ and $n+1$, $a_n=rc_1$, and $a_{n+1}=rc_2$.
We obtain that 
$$
c(\lambda,r)\leq r^2\left(\left(\frac{c_1}{n^2}+\frac{c_2}{(n+1)^2}\right)^2
+ \frac{1}{\sqrt{\lambda}}\right)\leq
r^2\left(\frac{2}{n^4}+ \frac{1}{\sqrt{\lambda}}\right)\leq r^2\left(\frac{2}{\sqrt{\lambda}}\right).
$$

Therefore there exists $\overline{\lambda}\geq 3^4>0$ such that for any $\lambda\geq \overline{\lambda}$ we have $\sqrt{\lambda}\geq 4C$.
Then we have that
$$\frac{2}{\sqrt{\lambda}}\leq \frac{1}{C}=\mathcal{N}(a)^2\quad\text{for any }a\in\partial A_1.$$

We fix $r_0$ such that $0<r_0<\max\{1/2,\|b\|/4\}$.
Then we call
$$A_2=\left\{a\in l_2\text{ such that }r_0/2\leq \|a\|\leq 3r_0/2\right\}$$
 and we define
$$\mathcal{N}(a)=\tilde{\mathcal{N}}\left(r_0\frac{a}{\|a\|}\right)
+ C_1\left|\|a\|-r_0\right|\quad\text{for any }a\in A_2$$
where $C_1$ is a positive constant such that
$|\mathcal{N}(a)|^2>1/C$ for any $a\in \partial A_2$.

Then we can extend $\mathcal{N}$ in a continuous way outside $A_1\cup A_2$ in such a way that $\mathcal{N}(a)>0$ and $|\mathcal{N}(a)|^2\geq 1/C$ for any $a\in l_2\backslash (A_1\cup A_2)$. We obtain that $\mathcal{N}$ is continuous, nonnegative, and it is $0$ only at $b$. Nevertheless, for any $\lambda\geq \overline{\lambda}$ we have that
$\sigma_{\lambda}\in A_2$, actually we easily deduce that $
r_0/2\leq \|\sigma_{\lambda}\|\leq r_0$.

Therefore for any sequence $\{\lambda_n\}_{n\in\mathbb{N}}$ of positive numbers such that
$\lim_n\lambda_n=+\infty$ we have that
$\{\sigma_{\lambda_n}\}_{n\in\mathbb{N}}$ can not converge. In fact, if we had convergence of $\sigma_{\lambda_n}$ to $\sigma_{\infty}$ we would obtain that 
$r_0/2\leq \|\sigma_{\infty}\|\leq r_0\leq \|b\|/4$ and $\mathcal{N}(\sigma_{\infty})=0$ and this is a contradiction.

\smallskip

\noindent
\textbf{Second version.} We consider the following auxiliary function $g:[0,+\infty)\to\mathbb{R}$ such that $g$ is continuous, strictly positive, nonincreasing and satisfies the following assumptions. First, 
$g(r)=1/f(\|b\|/2)$ for any $r$ such that $0\leq r\leq 2\|b\|$. Then, for any $r$ such that $r>2\|b\|$, we set
 $g(r)=C_1\rme^{-r^2}$ with $C_1$ a constant such that $C_1\rme^{-(2\|b\|)^2}=1/f(\|b\|/2).$ It is easy to show that there exists $\overline{\lambda}>0$  such that
 for any $\lambda\geq\overline{\lambda}$ we have
$$\min_{r\geq 0}\{(g(r))^2+r^2/\lambda\}=\frac{1}{2\lambda}\left(1+\log(2C_1^2\lambda)\right).$$
Such a minimum is reached, uniquely, at $r=r_{\lambda}=\sqrt{(1/2)\log(2C_1^2\lambda)}$ for $\lambda\geq \overline{\lambda}$.

Then we define $\mathcal{N}(a)=g(\|a\|)$ for any $a\in l_2\backslash A_1$. Clearly we have that $\mathcal{N}$ is again continuous, nonnegative, and it is $0$ only at $b$.
Moreover, for any $\lambda\geq\overline{\lambda}$ we have
$$\min\{|\mathcal{N}(a)|^2+|a|^2/\lambda:\ a\in l_2\backslash A_1\}=\frac{1}{2\lambda}(1+\log(2C_1^2\lambda))$$
and it is reached, uniquely, in
$a=a^{\lambda}$ where $a^{\lambda}_1=r_{\lambda}$ and $a^{\lambda}_n=0$ for any $n\geq 2$.

It is immediate to note that for some $\overline{\lambda}_1\geq \max\{\overline{\lambda},C^2\}$ we have that for any $\lambda\geq \overline{\lambda}_1$
it holds
$$\frac{1}{2\lambda}(1+\log(2C_1^2\lambda))<\frac{2}{\sqrt{\lambda}}$$
therefore there exists a unique solution to the minimization problem
$$\min\{|\mathcal{N}(a)|^2+|a|^2/\lambda:\ a\in l_2\}$$
given by $\sigma_{\lambda}=a^{\lambda}$. We conclude that
$$\lim_{\lambda\to+\infty}\|\sigma_{\lambda}\|=+\infty.$$
\end{exam}

\subsection{A counterexample}\label{wholesequenceex}

In this example we consider the abstract tighter multiscale construction presented in 
Subsection~\ref{abstractsecondsub}.
We wish to show that, in general, the sequence $\{\tilde{\sigma}_n\}$ may not converge and that it  may have different subsequences converging to different limits.

Le us consider the space $X$ to be $\mathbb{R}^2$ with the usual Euclidean norm and $Y$ to be $\mathbb{R}$ with the usual distance. We also set $E=\mathbb{R}^2$.

Let us consider two continuous functions $\mathcal{N},\ \Xi:\mathbb{R}^2\to\mathbb{R}$ such that
for any $x\in\mathbb{R}^2$ we have
$$\mathcal{N}(x),\ \Xi(x)\text{ is greater than, equal or lesser than }1\text{ if and only if }\|x\|\text{ is}.$$
Therefore, in both cases, if we fix $\hat{\mathcal{N}}=1$ we have that $\hat{E}=\partial B_1$.

We assume that, for any $n\geq 0$ we have
$$\lambda_n=b^n\quad \text{and}\quad a_n=1/c^n$$
where
\begin{equation}
c\geq 9\quad\text{and}\quad b/c>2.
\end{equation}
note that $a_n\leq a_{n-1}$ for any $n\geq 1$ and that $\lim_n a_n=0$.
Finally, we set
$$\alpha=\beta=\gamma=1.$$
We conclude that \eqref{crucialcondition0} holds.

Let us consider the sequences $\{\tilde{\sigma}_n\}_{n\in\mathbb{N}}$ corresponding to $\mathcal{N}$ and $\{\tilde{\tau}_n\}_{n\in\mathbb{N}}$ corresponding to $\Xi$. It is not difficult to show that Theorem~\ref{minimizercor0} and \eqref{distance}  hold for both sequences.

We shall assume that $\Xi$ is radial and that $\mathcal{N}$ is lesser than or equal to $\Xi$, namely there exists a continuous nondecreasing function $\tilde{\Xi}:[0,+\infty)\to \mathbb{R}$ such that
for any $x\in\mathbb{R}^2$
$$\mathcal{N}(x)\leq \Xi(x)=\tilde{\Xi}(\|x\|).$$

We consider the following two sequences
$$r_n=\sum_{j=0}^{n}\left(\frac{1}{2}\right)^{j+1}\text{ and }s_n=r_n+\frac{r_{n+1}-r_n}{2}\quad\text{for any }n\geq 0.$$
Clearly, $0<r_n<s_n<r_{n+1}<1$ for any $n\geq 0$ and $\lim_nr_n=\lim_n s_n=1$. To simplify the notation sometimes we may use $r_{-1}=0$.

We also need to define this further sequence
$$h_n=1-\frac{9}{8}\left(\frac{1}{c^{n+1}}+\frac{1}{b^{n+1}}\right)\frac{1}{2^{n+2}}
\quad\text{for any }n\geq 0.$$
By our assumptions on $c$ we have $0<3/4\leq h_n<h_{n+1}<1$ for any $n\geq 0$ and $\lim_nh_n=1$.

We define $\tilde{\Xi}$ as follows, for a given $\varepsilon>0$,
\begin{equation}\label{Xidefinition}
\tilde{\Xi}(r)=\left\{\begin{array}{ll}
h_0+2(r-r_0)-\varepsilon(r-r_0)^2 &\text{if } 0\leq r\leq r_0\\
h_n &\text{if } r_n\leq r\leq s_n,\ n\geq 0\\
\displaystyle{h_n+\frac{h_{n+1}-h_n}{r_{n+1}-s_n}(r-s_n)}
&\text{if }s_n\leq r\leq r_{n+1},\ n\geq 0\\
1+2(r-1) &\text{if }1\leq r
\end{array}
\right.
\end{equation}
whereas, concerning $\mathcal{N}$, we assume
that
\begin{equation}\label{Lambdaprop}
\mathcal{N}(x)=\left\{\begin{array}{ll}
\Xi(x)
&\text{if } 0\leq \|x\|\leq r_0\\
\Xi(x)=h_{2n} &\text{if } r_{2n}\leq \|x\|\leq s_{2n},\ n\geq 0\\
\Xi(x)
&\text{if }s_{2n+1}\leq \|x\|\leq r_{2n+2},\ n\geq 0\\
\Xi(x) &\text{if }1\leq \|x\|.
\end{array}
\right.
\end{equation}
and that $h_{2n}\leq \mathcal{N}(x)\leq \Xi(x)$ for any $x$ such that $s_{2n}\leq \|x\|\leq s_{2n+1}$, $n\geq 0$.

We note that $\tilde{\Xi}$ is a Lipschitz function over $[0,+\infty)$, therefore also $\Xi $ is a Lipschitz function over $\mathbb{R}^2$.

We call
$$f_0(x)=|1-\mathcal{N}(x)|+2\|x\|\quad\text{and}\quad
g_0(x)=|1-\Xi(x)|+2\|x\|.$$
We begin by stating that
$$\partial B_{r_0}=\argmin_{x\in \mathbb{R}^2} f_0(x)=\argmin_{x\in \mathbb{R}^2} g_0(x).$$
Therefore, without loss of generality we may assume that
$\tilde{\sigma}_0=\tilde{\tau}_0=(r_0,0).$

In fact, first of all we note that
\begin{multline*}
g_0(r_0,0)=(1-h_0)+2r_0<g_0(x)=(1-(h_0+2(\|x\|-r_0)-\varepsilon(\|x\|-r_0)^2)+2\|x\|=\\
(1-h_0)+2r_0+\varepsilon(\|x\|-r_0)^2\quad \text{for any }0\leq \|x\|< r_0.
\end{multline*}
Then we have that, since $c\geq 9$,
$$g_0(r_0,0)=
\frac{9}{8}\left(\frac{1}{c}+\frac{1}{b}\right)\frac{1}{2^{2}}+1
<2r_1=3/2\leq g_0(x)\quad \text{for any }\|x\|\geq r_1.$$
It remains to consider the case $r_0<\|x\|<r_1$ where we have that
$g_0(r_0,0)<g_0(x)$ is equivalent to
$$\frac{9}{8}\left(\frac{1}{c}+\frac{1}{b}\right)\frac{1}{2^{2}}+1<
(1-\Xi(x))+2\|x\|\quad \text{for any }r_0< \|x\|< r_1.$$
This is obviously true if $r_0<\|x\|\leq s_0$ since $\Xi(x)=h_0$ there.
If $s_0<\|x\|\leq r_1$, then we need to show that
$$1<-\frac{h_1-h_0}{r_1-s_0}(r-s_0)
+2r\quad \text{for any }s_0<r<r_1$$
that follows from easy computations.

The result for $f_0$ follows by noticing that $f_0(r_0,0)=g_0(r_0,0)$ and
$f_0(x)\geq g_0(x)$ for any $x\in\mathbb{R}^2$.

The first  important result is the following.

\begin{prop}\label{Xisequence}
Under the previous notation and assumptions, let us assume that
$\tilde{\tau}_0=(r_0,0)$. Then for any $n\geq 0$ we have
$$\tilde{\tau}_n=(r_n,0).$$
\end{prop}

The proof of this proposition follows essentially by the next lemma. First let us recall the following notation.
For a fixed $n\geq 0$ let us suppose that we have computed $\tilde{\tau}_n$ and
$\tilde{\sigma}_n$.
Then $\tilde{\tau}_{n+1}$ is the (unique) minimizer over $\mathbb{R}^2$ of the functional
$g_{n+1}(x)=b^{n+1}|1-\Xi(x)|+(b/c)^{n+1}\|x\|+\|x-\tilde{\tau}_n\|$, whereas 
$\tilde{\sigma}_{n+1}$ is a minimizer over $\mathbb{R}^2$ of the functional
$f_{n+1}(x)=b^{n+1}|1-\mathcal{N}(x)|+(b/c)^{n+1}\|x\|+\|x-\tilde{\sigma}_n\|$.

\begin{lem}\label{lemma1}
For a fixed $n\geq 0$ let us assume that $\tilde{\tau}_n=(r_n,0)$.
Then
$$r_n\leq \|\tilde{\tau}_{n+1}\|< s_{n+1}.$$
\end{lem}

\begin{proof} It is enough to show that
\begin{multline}\label{outside}
g_{n+1}(x)\geq (b/c)^{n+1}s_{n+1}+s_{n+1}-r_n> \\
b^{n+1}(1-h_n)+(b/c)^{n+1}r_n
=g_{n+1}(\tilde{\tau}_{n+1})\quad\text{for any }x\text{ s.t. }\|x\|\geq s_{n+1}
\end{multline}
and that
\begin{multline}\label{inside}
g_{n+1}(x)\geq b^{n+1}(1-\Xi(x))+(b/c)^{n+1}\|x\|
+r_n-\|x\|> \\
b^{n+1}(1-h_n)+(b/c)^{n+1}r_n
=g_{n+1}(\tilde{\tau}_{n+1})
\quad\text{for any }x\text{ s.t. }\|x\|< r_n.
\end{multline}

We begin by noticing that \eqref{outside} is equivalent to
$$\left(\frac{1}{c^{n+1}}+\frac{1}{b^{n+1}}\right)(s_{n+1}-r_n)>
1-h_n=\frac{9}{8}\left(\frac{1}{c^{n+1}}+\frac{1}{b^{n+1}}\right)\frac{1}{2^{n+2}}
$$
which is true since $c\geq 9$ and
$$s_{n+1}-r_n=(s_{n+1}-r_{n+1})+(r_{n+1}-r_n)=(1/2)\frac{1}{2^{n+3}}+
\frac{1}{2^{n+2}}=(5/4)\frac{1}{2^{n+2}}.$$

For what concerns \eqref{inside}, we argue in the following way.
We note that \eqref{inside} is equivalent to
$$\tilde{\Xi}(r)< h_n+\left(\frac{1}{c^{n+1}}-\frac{1}{b^{n+1}}\right)(r-r_n),\quad0\leq r< r_n$$
which is implied by
\begin{equation}\label{inside2}
\tilde{\Xi}(r)<G_n(r)= h_n+\left(\frac{1}{c^{n+1}}\right)(r-r_n),\quad0\leq r< r_n.
\end{equation}
Therefore it remains to prove \eqref{inside2}.
The case $n=0$ is trivial. Fixed $n\geq 1$, we begin by proving \eqref{inside2} for $r$ satisfying $r_{n-1}\leq r <r_n$. First of all we observe that $h_{n-1}<G_n(r_{n-1})$ therefore
\eqref{inside2} holds for any $r$, $r_{n-1}\leq r\leq s_{n-1}$. In particular,
$\tilde{\Xi}(s_{n-1})<G_n(s_{n-1})$. We have that for any $r$, $s_{n-1}\leq r\leq r_n$,
$\tilde{\Xi}(r)=\tilde{\Xi}(s_{n-1})+l_n(r-s_{n-1})$ whereas
$G_n(r)=G_n(s_{n-1})+(1/c^{n+1})(r-s_{n-1})$. Since
$\tilde{\Xi}(s_{n-1})<G_n(s_{n-1})$ and $\tilde{\Xi}(r_n)=G_n(r_n)=h_n$ it is easy to conclude that
\eqref{inside2} holds for any $r$, $r_{n-1}\leq r<r_n$ and any $n\geq 0$, where $r_{-1}=0$.
In order to conclude the proof we note that, for any $n\geq 1$,
$G_{n-1}(r_{n-1})=h_{n-1}<G_n(r_{n-1})$ and that $G_n$ is affine with the coefficient of the linear part given by $(1/c^{n+1})$ which is decreasing with respect to $n$. We obtain that
$G_{n-1}(r)<G_n(r)$ for any $r$, $0\leq r\leq r_{n-1}$. By induction we conclude that for any $0\leq m<n$ we have
$$\tilde{\Xi}(r)\leq G_m(r)<G_n(r)\quad\text{for any }r_{m-1}\leq r\leq r_m.$$
Therefore \eqref{inside2} holds true and the proof is concluded.\end{proof}

By using the fact the $\Xi$ is radial, thus reducing the problem to a minimization on the interval
$[r_n, s_{n+1}]$, it is easy to show that if, for a fixed $n\geq 0$, we have $\tilde{\tau}_n=(r_n,0)$,
then $\tilde{\tau}_{n+1}=(r_{n+1},0)$. Therefore, Proposition~\ref{Xisequence} follows by an elementary induction argument.

An interesting corollary of Lemma~\ref{lemma1} is the following. Let us assume that, for a fixed $n\geq 0$,
$\|\tilde{\sigma}_n\|=\|\tilde{\tau}_n\|=r_n$ and $\mathcal{N}(\tilde{\sigma}_n)=\Xi(\tilde{\sigma}_n)=
\Xi(\tilde{\tau}_n)=h_n$. Then
$$r_n\leq \|\tilde{\sigma}_{n+1}\|< s_{n+1}.$$
Furthermore, for any $n\geq 0$, if we have 
$\tilde{\sigma}_{2n+1}=r_{2n+1}(\cos(\theta), \sin(\theta))$, for some $\theta\in [0,2\pi)$,  and $\mathcal{N}(\tilde{\sigma}_{2n+1})=h_{2n+1}$, then there exists a unique minimizer
$$\tilde{\sigma}_{2n+2}=r_{2n+2}(\cos(\theta), \sin(\theta)).$$

The counterexample is based on the following lemma.

\begin{lem}\label{constructionlemma}
There exists an integer $\overline{n}\geq 1$ such that for any $n\geq \overline{n}$ we have the following.
We may construct a Lipschitz continuous function $\mathcal{N}$, a modification of the function $\Xi$ on the annulus $\overline{B}_{s_{2n+1}}\backslash B_{s_{2n}}$, satisfying the previous assumptions and the following properties.
For any $x=(x_1,x_2)\in \overline{B}_{s_{2n+1}}\backslash B_{s_{2n}}$ we have, first, that
$h_{2n}\leq \mathcal{N}(x)\leq \Xi(x)$.
Second, 
$\mathcal{N}(x_1,x_2)=\mathcal{N}(x_1,-x_2)$.
Third,  $\mathcal{N}((0,r_{2n+1}))=h_{2n+1}$ and, assuming that
$\tilde{\sigma}_{2n}=(r_{2n},0)$,
$$\{(0,r_{2n+1}),(0,-r_{2n+1})\}=\argmin_{x\in\mathbb{R}^2}f_{2n+1}(x).$$
Therefore
$\tilde{\sigma}_{2n+1}=(0,r_{2n+1})$ satisfies
$\mathcal{N}(\tilde{\sigma}_{2n+1})=h_{2n+1}$ and is a minimizer to $f_{2n+1}$.
\end{lem}

\begin{proof} We already know that
$$\argmin_{x\in\mathbb{R}^2}f_{2n+1}(x)=\argmin_{x\in\mathbb{R}^2:\ r_{2n}\leq\|x\|\leq s_{2n+1}}f_{2n+1}(x).$$

We modify $\Xi$ in $\overline{B}_{r_{2n+2}}\backslash B_{r_{2n}}$. Let $x=(x_1,x_2)\in\overline{B}_{r_{2n+2}}\backslash B_{r_{2n}}$.
We set
$\mathcal{N}(x)=\Xi(x)$ if $r_{2n}\leq \|x\|\leq s_{2n}$ and if $s_{2n+1}\leq \|x\|\leq r_{2n}$.
We set $\mathcal{N}(x)=\Xi(x)$ also if $x_1\leq 0$. Furthermore, we assume that
$\mathcal{N}(x_1,x_2)=\mathcal{N}(x_1,-x_2)$. Therefore, we limit ourselves to consider points
$x\in\overline{B}_{r_{2n+2}}\backslash B_{r_{2n}}$ such that $x_1,\ x_2\geq 0$.

In $\overline{B}_{s_{2n+1}}\backslash B_{s_{2n}}$ we perform the following geometric construction, which is summarized in Figure~\ref{figure1}.

\begin{figure}[htb]
\centering
\includegraphics[width = .6\textwidth]{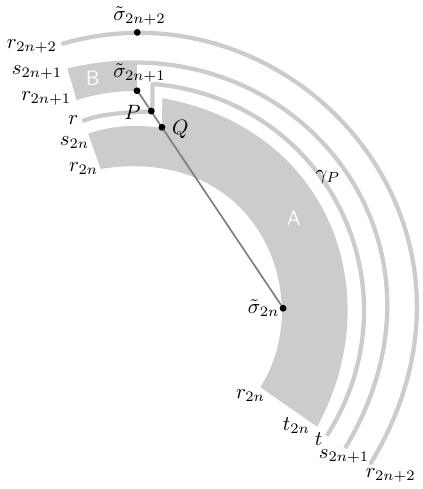}
\caption{Illustration of the geometric construction underlying the counterexample in Theorem~\ref{counterexampleteo}.}
\label{figure1}
\end{figure}

We consider the segment connecting $\tilde{\sigma}_{2n}=(r_{2n},0)$ and our candidate to be a minimizer $\tilde{\sigma}_{2n+1}=(0,r_{2n+1})$. We call $Q$ the point which is the intersection of 
this segment with $\partial B_{s_{2n}}$. Finally we consider a point $P$ belonging to the segment connecting $Q$ to $\tilde{\sigma}_{2n+1}$. We call $r=r(P)=\|P\|$ and we denote 
$t=t(P)=r(P)+(s_{2n+1}-r_{2n+1})$. When $P=Q$ we denote $t_{2n}=t(Q)=r(Q)+(s_{2n+1}-r_{2n+1})
=s_{2n}+(s_{2n+1}-r_{2n+1})<r_{2n+1}$ and we observe that, when $P=\tilde{\sigma}_{2n+1}$,
$t(\tilde{\sigma}_{2n+1})=s_{2n+1}$.

Let us assume that $P=(x_1^0,x_2^0)$, we call $P'=(x_1^0,-x_2^0)$.
Then we consider the points $\tilde{P}=(x_1^0,x_2^1)$, with $x_2^1>0$, and $\tilde{P}'=(x_1^0,-x_2^1)$ that are
the intersections of the vertical line $\{x_1=x_1^0\}$ with $\partial B_{t(P)}$.
We call $\gamma_P$ the following curve that is formed by four parts. The first part, $\gamma_P^1$ consists of the points $x$ belonging  to $\partial B_{r(P)}$ such that $x_1\leq x_1^0$. Two other parts, $\gamma_P^2$ and $\gamma_P^3$, are the vertical segments connecting $P$ to $\tilde{P}$ and $P'$ to $\tilde{P}$, respectively. Finally, the fourth part $\gamma_P^4$
consists of the points $x$ belonging  to $\partial B_{t(P)}$ such that $x_1\geq x_1^0$.

Then we define $\mathcal{N}$ in the following way. We set $\mathcal{N}(x)=h_{2n}$ if $x$ belongs to $\partial B_{r_{2n}}$, if $x$ belongs to $\gamma_Q$ and if $x$ lies between these two curves (region A in Figure~\ref{figure1}). We set $\mathcal{N}(x)=h_{2n+1}$
 if $x$ belongs to $\partial B_{s_{2n+1}}$, if $x$ belongs to $\gamma_{\tilde{\sigma}_{2n+1}}$ and if $x$ lies between these two curves (region B in Figure~\ref{figure1}).
We note, in particular, that $\mathcal{N}(\tilde{\sigma}_{2n+1})=h_{2n+1}$.
Finally for any $x\in \gamma_P$, $P$ belonging to the segment connecting $Q$ to $\tilde{\sigma}_{2n+1}$, we set $\mathcal{N}(x)=\tilde{\Xi}(r(P))$.

We observe that for any $P$ belonging to the segment connecting $Q$ to $\tilde{\sigma}_{2n+1}$,
\begin{equation}\label{argmingammaP}
\argmin_{x\in \gamma_P}f_{2n+1}(x)=\{P,P'\}.
\end{equation}
In fact, for any $x\in \gamma_P$ we have $\|x\|\geq \|P\|$ and
$\mathcal{N}(x)=\mathcal{N}(P)$. For any $x$ belonging to $\gamma_P^1$, $\gamma_P^2$ or $\gamma_P^3$ we have that $\|x-\tilde{\sigma}_{2n}\|\geq \|P-\tilde{\sigma}_{2n}\|$ with equality holding only if $x=P$ or $x=P'$, therefore 
$$\argmin_{x\in (\gamma_P^1\cup \gamma_P^2\cup \gamma_P^3)}f_{2n+1}(x)=\{P,P'\}.$$
Finally, it is easy to remark that
$$\argmin_{x\in \gamma_P^4}f_{2n+1}(x)=\{(t(P),0)\}.$$
Therefore, the candidates to be a minimizer of $f_{2n+1}$ on $\gamma(P)$ are either the couple $P$ and $P'$ or the point $(t(P),0)$.
 We have that $f_{2n+1}(P)<f_{2n+1}(t(P),0)$ if and only if
 $$\left(\frac{b}{c}\right)^{2n+1}r(P)+ \|P-\tilde{\sigma}_{2n}\|<\left(\frac{b}{c}\right)^{2n+1}t(P)+(t(P)-r_{2n}).$$
On the other hand, $t(P)-r(P)=s_{2n+1}-r_{2n+1}=(1/2)(1/2^{2n+3})=1/2^{2n+4}$, therefore since $b>2c$, there exists $\overline{n}_1\geq 1$ such that for any $n\geq \overline{n}_1$ we have
$$16<\left(\frac{b}{2c}\right)^{2n+1},$$
hence
$$\|P-\tilde{\sigma}_{2n}\|\leq 2< \left(\frac{b}{c}\right)^{2n+1}[t(P)-r(P)]<
\left(\frac{b}{c}\right)^{2n+1}[t(P)-r(P)]+(t(P)-r_{2n})$$
and \eqref{argmingammaP} is proved.

Next we show that
\begin{equation}\label{minimosegment}
f_{2n+1}(\tilde{\sigma}_{2n+1})<f_{2n+1}(\tilde{\sigma}_{2n})=
b^{2n+1}(1-h_{2n})+
\left(\frac{b}{c}\right)^{2n+1}r_{2n}
\end{equation}
that is
\begin{multline*}
\left(\frac{1}{c}\right)^{2n+1}(r_{2n+1}-r_{2n})+\left(\frac{1}{b}\right)^{2n+1}\|\tilde{\sigma}_{2n+1}-\tilde{\sigma}_{2n}\|<h_{2n+1}-h_{2n}=\\
\frac{9}{8}\left[\left(\frac{1}{c^{2n+1}}+\frac{1}{b^{2n+1}}\right)\frac{1}{2^{2n+2}}
-
\left(\frac{1}{c^{2n+2}}+\frac{1}{b^{2n+2}}\right)\frac{1}{2^{2n+3}}\right].
\end{multline*}
Since $r_{2n+1}-r_{2n}=1/2^{2n+2}$ and $\|\tilde{\sigma}_{2n+1}-\tilde{\sigma}_{2n}\|\leq 2$, it is enough to prove that
$$2\left(\frac{1}{b}\right)^{2n+1}<
\frac{1}{8}\left(\frac{1}{c^{2n+1}}+\frac{1}{b^{2n+1}}\right)\frac{1}{2^{2n+2}}
-\frac{9}{8}\left(\frac{1}{c^{2n+2}}+\frac{1}{b^{2n+2}}\right)\frac{1}{2^{2n+3}}.$$
We note that, since $c\geq 9$, we have that
$$\frac{9}{2}\left(\frac{1}{c^{2n+2}}+\frac{1}{b^{2n+2}}\right)\leq 
\frac{1}{2}\left(\frac{1}{c^{2n+1}}+\frac{1}{b^{2n+1}}\right)$$
hence it is enough to prove that
$$2\left(\frac{1}{b}\right)^{2n+1}<
\frac{1}{16}\left(\frac{1}{c^{2n+1}}+\frac{1}{b^{2n+1}}\right)\frac{1}{2^{2n+2}}
$$
or even
$$2\left(\frac{1}{b}\right)^{2n+1}<
\frac{1}{32}\left(\frac{1}{2c}\right)^{2n+1}
$$
which is true if
$$64<\left(\frac{b}{2c}\right)^{2n+1}.$$
Therefore there exists $\overline{n}\geq\overline{n}_1$ such that for any $n\geq\overline{n}$
\eqref{minimosegment} holds true.

It remains to show that
$$f_{2n+1}(\tilde{\sigma}_{2n+1})<f_{2n+1}(P)$$
for any $P$ belonging to the segment connecting $Q$ to $\tilde{\sigma}_{2n+1}$, $P$ clearly different from $\tilde{\sigma}_{2n+1}$. We begin by noticing that
$f_{2n+1}(\tilde{\sigma}_{2n+1})<f_{2n+1}(Q)$ since $f_{2n+1}(Q)>f_{2n+1}(\tilde{\sigma}_{2n})$.
Then the problem reduces to a minimization over a real interval and the conclusion follows by elementary computations.

We finally remark that, by construction, $\mathcal{N}$ is Lipschitz on $\overline{B}_{s_{2n+1}}\backslash B_{s_{2n}}$.\end{proof}

We have the following theorem, which contains our counterexample.

\begin{teo}\label{counterexampleteo}
Under the previous notation and assumptions, there exists a Lipschitz function $\mathcal{N}$ and an integer $\overline{n}\geq 1$ 
such that for any $n\geq 0$ we have $\tilde{\sigma}_n=r_n(cos(\theta_n),\sin(\theta_n))$,
with $\theta_n\in [0,+\infty)$ satisfying the following property
$$\theta_n=0\text{ for any }n,\ 0\leq n\leq 2\overline{n}$$
and, for any $m\geq 1$,
$$\theta_n=m\pi/2 \text{ for any }n=2\overline{n}+2m-1,\ 2\overline{n}+2m.$$
\end{teo}

\begin{proof} We assume that $\mathcal{N}$ and $\Xi$ coincide for any $x$ such that $\|x\|\leq r_{2\overline{n}}$. Therefore, assuming $\tilde{\sigma}_0=(r_0,0)$, it is easy to conclude that
$\tilde{\sigma}_n=(r_n,0)$ for any $n$, $0\leq n\leq 2\overline{n}$.

Then, on $\overline{B}_{s_{2\overline{n}+1}}\backslash B_{s_{2\overline{n}}}$, we use the function $\mathcal{N}$ defined in Lemma~\ref{constructionlemma}. Therefore, we infer that
$\tilde{\sigma}_{2\overline{n}+1}=(0,r_{2\overline{n}+1})$ and, by the remark made after Lemma~\ref{lemma1}, we also have that $\tilde{\sigma}_{2\overline{n}+2}=(0,r_{2\overline{n}+2})$. Then by repeating the same construction, up to a rotation, we easily conclude the proof by an induction argument.\end{proof}\bigskip

\noindent\textbf{Acknowledgements.} We wish to thank the Mittag-Leffler Institute (Program on Inverse Problems and Applications), the  Erwin Schr\"odinger International Institute (Program on Infinite-dimensional Riemannian Geometry with Applications to Image Matching and Shape Analysis) and the Poincar\'e Institute (Program on Inverse Problems) for their hospitality and for the excellent working conditions which have made our collaboration possible. We also wish to thank 
Martin Bauer for encouraging discussions during the early stages of this project. KM is supported by the EU Horizon 2020 MSC grant No 661482, by the Swedish Foundation for Strategic Research grant ICA12-0052, and by the STINT grant PT2014-5823.
AN is partially supported by the NSERC Discovery Grant RGPIN-06329. 
LR is partially supported by GNAMPA--INdAM through 2017 projects and by the Universit\`a di Trieste through FRA~2016.


\bigskip



\begin{thebibliography}{99}

\bibitem{Ada}
R.~A.~Adams and J.~J.~F.~Fournier,
\emph{Sobolev Spaces}, Second Edition, Academic Press, Amsterdam, 2003.

\bibitem{Ale e Cab}
G.~Alessandrini and E.~Cabib,
\emph{EIT and the average conductivity},
J. Inverse Ill-Posed Probl. \textbf{16} (2008) 727--736.

\bibitem{All}
G.~Allaire,
\emph{Shape Optimization by the Homogenization Method},
Springer-Verlag, New York, 2002.

\bibitem{ALP2}
K.~Astala, M.~Lassas and L.~P\"aiv\"arinta,
\emph{The borderlines of invisibility and visibility in Calder\'on's inverse problem},
Anal. PDE \textbf{9} (2016) 43--98.

\bibitem{Ast e Pai}
K.~Astala and L.~P\"aiv\"arinta,
\emph{Calder\'on's inverse conductivity problem in the plane},
Ann. of Math. (2), \textbf{163} (2006) 265--299.

\bibitem{Ast e Pai e Las}
K.~Astala, L.~P\"aiv\"arinta and M.~Lassas,
\emph{Calder\'on's inverse problem for anisotropic conductivity in the plane},
Comm. Partial Differential Equations \textbf{30} (2005) 207--224.

\bibitem{AXRNG}
P.~Athavale, R.~Xu, P. ~Radau, A.~Nachman and G.~A.~Wright, \emph{Multiscale properties of weighted total variation flow with applications to denoising and registration}, Medical Image Analysis \textbf{23} (2015) 28--42.

\bibitem{B-M-T-Y} M.~F.~Beg, M.~I.~Miller, A.~Trouv\'e and L.~Younes, 
\emph {Computing Large Deformation Metric Mappings via Geodesic Flows of Diffeomorphisms}, International Journal of Computer Vision \textbf{61} (2005) 139--157.

\bibitem{B-H}
M.~Bruveris and D.~D.~Holm,
\emph{Geometry of Image Registration}: \emph{The Diffeomorphism Group and Momentum Maps},
in D.~E.~Chang, D.~D.~Holm, G.~Patrick and T.~Ratiu
eds., \emph{Geometry, Mechanics, and Dynamics}, Springer-Verlag, New York, 2015, pp.~19--56.

\bibitem{B-V}
M.~Bruveris and F.-X.~Vialard,
\emph{On completeness of groups of diffeomorphisms},
J. Eur. Math. Soc. (JEMS) \textbf{19} (2017) 1507--1544.

\bibitem{Car-Rog}
P.~Caro and K.~M.~Rogers,
\emph{Global uniqueness for the Calder\'on's problem with Lipschitz conductivities}, Forum Math. Pi \textbf{4} (2016) e2 (28 pp).

\bibitem{Chan e Tai}
T.~F.~Chan and X.-C.~Tai,
\emph{Level set and total variation regularization for elliptic inverse problems with discontinuous coefficients},
J. Comput. Phys. \textbf{193} (2003) 40--66.

\bibitem{Chen}
Y.~Chen,
\emph{Inverse scattering via Heisenberg's uncertainty principle},
Inverse Problems \textbf{13} (1997) 253--282.

\bibitem{Chu e Chan e Tai}
E.~T.~Chung, T.~F.~Chan and X.-C.~Tai,
\emph{Electrical impedance tomography using level set representation and total variational regularization},
J. Comput. Phys. \textbf{205} (2005) 357--372.

\bibitem{Dob e San94}
D.~C.~Dobson and F.~Santosa,
\emph{An image-enhancement technique for electrical impedance tomography},
Inverse Problems \textbf{10} (1994) 317--334.

\bibitem{Far e Kur e Rui}
D.~Faraco, Y.~Kurylev and A.~Ruiz,
$G$-\emph{convergence}, \emph{Dirichlet to Neumann maps and invisibility},
J. Funct. Anal. \textbf{267} (2014) 2478--2506.

\bibitem{Gris-et-al}
B.~Gris, S.~Durrleman and A.~Trouv\'e.
\emph{A sub-Riemannian modular framework for diffeomorphism based analysis of shape ensembles},
SIAM J. Imaging Sci. \textbf{11} (2018) 802--833.

\bibitem{Hab}
B.~Haberman,
\emph{Uniqueness in Calder\'on's problem for conductivities with unbounded gradient}, Comm. Math. Phys. \textbf{340}  (2015) 639--659.

\bibitem{Hab-Tat}
B.~Haberman and D.~Tataru,
\emph{Uniqueness in Calder\'on's problem with Lipschitz conductivities},
Duke Math. J. \textbf{162} (2013) 496--516.

\bibitem{Jaw-Mil}
B.~Jawerth and M.~Milman, \emph{Lectures on Optimization, Image Processing, and
Interpolation Theory} in \emph{Function Spaces, Inequalities and Interpolation}, 
Spring School on Analysis, Paseky 2007, Matfyzpress, Prague, 2007.

\bibitem{Koh e Vog84:1}
R.~Kohn and M.~Vogelius,
\emph{Determining conductivity
by boundary measurements},
Comm. Pure Appl. Math. \textbf{37} (1984) 289--298.

\bibitem{Koh e Vog85}
R.~V.~Kohn and M.~Vogelius,
\emph{Determining conductivity
by boundary measurements} II. \emph{Interior results},
Comm. Pure Appl. Math. \textbf{38} (1985) 643--667.

\bibitem{Koh e Vog87}
R.~V.~Kohn and M.~Vogelius,
\emph{Relaxation of a variational method for impedance computed tomography}, Comm. Pure Appl. Math. \textbf{40} (1987)
745--777.

\bibitem{YMey}
Y.~Meyer,
\emph{Oscillating Patterns in Image Processing and Nonlinear Evolution Equations},
American Mathematical Society, Providence, 2001.

\bibitem{Mur e Tar1}
F.~Murat and L.~Tartar,
\emph{$H$-convergence}
in A.~Cherkaev and R.~Kohn eds., \emph{Topics in the Mathematical Modelling of Composite Materials},
Birkh\"auser, Boston, 1997, pp.~21--43.

\bibitem{Mur e Tar2}
F.~Murat and L.~Tartar,
\emph{Calculus of Variations and Homogenization}
in A.~Cherkaev and R.~Kohn eds., \emph{Topics in the Mathematical Modelling of Composite Materials},
Birkh\"auser, Boston, 1997, pp.~139--173.

\bibitem{Nac}
A.~I.~Nachman,
\emph{Global uniqueness for a two-dimensional inverse boundary value
problem},
Ann. of Math. (2) \textbf{143} (1996) 71--96.

\bibitem{NRT}
A.~Nachman, I.~Regev, D.~Tataru, \emph{A Nonlinear Plancherel Theorem with Applications to Global Well-Posedness for the Defocusing Davey-Stewartson Equation and to the Inverse Boundary Value Problem of Calder\'on}, 
Invent. Math. (2019), online first, https://doi.org/10.1007/s00222-019-00930-0.

\bibitem{Paquin-Levy-Schreibmann-Xing}
D.~Paquin, D.~Levy, E.~Schreibmann and L.~Xing,
\emph{Multiscale image registration},
Math. Biosci. Eng. \textbf{3} (2006) 389--418. 

\bibitem{Paquin-Levy-Xing1}
D.~Paquin, D.~Levy and L.~Xing,
\emph{Hybrid multiscale landmark and deformable image registration},
Math. Biosci. Eng. \textbf{4} (2007) 711--737.

\bibitem{Paquin-Levy-Xing2}
D.~Paquin, D.~Levy and L.~Xing,
\emph{Multiscale deformable registration of noisy medical images},
Math. Biosci. Eng. \textbf{5} (2008) 125--144. 

\bibitem{Risser-et-al0}
L.~Risser, F.-X.~Vialard, R.~Wolz, D.~D.~Holm and D.~Rueckert,
\emph{Simultaneous fine and coarse diffeomorphic registration}: \emph{application to atrophy measurement in Alzheimer's disease}, in T.~Jiang, N.~Navab, J.~P.~W.~Pluim and M.~A.~Viergever eds.,
\emph{Medical Image Computing and Computer-Assisted Intervention} --- \emph{MICCAI 2010}, Proceedings Part II, Springer-Verlag, Berlin, Heidelberg, 2010, pp.~610--617.

\bibitem{Risser-et-al}
L.~Risser, F.-X.~Vialard, R.~Wolz, M.~Murgasova, D.~D.~Holm and D.~Rueckert,
\emph{Simultaneous Multi-scale Registration Using Large
Deformation Diffeomorphic Metric Mapping},
IEEE Transaction on Medical Imaging \textbf{30} (2011) 1746--1759.

\bibitem{Ron08}
L.~Rondi,
\emph{On the regularization of the inverse conductivity problem with discontinuous conductivities}, Inverse Probl. Imaging \textbf{2} (2008) 397--409.

\bibitem{Ron15}
L.~Rondi, \emph{Continuity properties of Neumann-to-Dirichlet maps with respect to the $H$-convergence of the coefficient matrices}, Inverse Problems \textbf{31} (2015) 045002 (24 pp). 

\bibitem{Ron16}
L.~Rondi, \emph{Discrete approximation and regularisation for the inverse conductivity problem}, Rend. Istit. Mat. Univ. Trieste \textbf{48} (2016) 315--352.

\bibitem{Ron e San}
L.~Rondi and F.~Santosa,
\emph{Enhanced electrical impedance
tomography \emph{via} the Mumford-Shah functional},
ESAIM Control Optim. Calc. Var. \textbf{6} (2001) 517--538.

\bibitem{Sch-et-al}
O.~Scherzer, M.~Grasmair, H.~Grossauer, M.~Haltmeier, F.~Lenzen,
\emph{Variational Methods in Imaging},
Springer-Verlag, New York, 2009.

\bibitem{Sommer1}
S.~Sommer, F.~Lauze, M.~Nielsen and X.~Pennec, 
\emph{Kernel Bundle EPDiff: Evolution Equations for Multi-scale Diffeomorphic Image Registration}, in
A.~M.~Bruckstein, B.~M.~ter Haar Romeny, A.~M.~Bronstein and  M.~M.~Bronstein eds., \emph{Scale Space and Variational Methods in Computer Vision. SSVM 2011}, Springer-Verlag, Berlin, Heidelberg, 2012, pp.~677--688. 

\bibitem{Sommer2}
S.~Sommer, F.~Lauze, M.~Nielsen and X.~Pennec, 
\emph{Sparse multi-scale diffeomorphic registration: the kernel bundle framework}, J. Math. Imaging Vision \textbf{46} (2013) 292--308.

\bibitem{Syl}
J.~Sylvester,
\emph{An anisotropic inverse boundary value problem},
Comm. Pure Appl. Math. \textbf{43} (1990) 201--232. 

\bibitem{Syl e Uhl87}
J.~Sylvester and G.~Uhlmann,
\emph{A global uniqueness theorem for an inverse boundary value problem},
Ann. of Math. (2) \textbf{125} (1987) 153--169.

\bibitem{T-N-V}
E.~Tadmor, S.~Nezzar and L.~Vese,
\emph{A multiscale image representation using hierarchical} $(BV,L^2)$ \emph{decompositions},
Multiscale Model. Simul. \textbf{2} (2004) 554--579.

\bibitem{Trouve95}
A.~Trouv\'e,
\emph{Action de groupe de dimension infinie et reconnaissance de formes},
C. R. Acad. Sci. Paris S\'er. I Math. \textbf{321} (1995) 1031--1034.
 
\bibitem{Trouve98}
A.~Trouv\'e,
\emph{Diffeomorphisms groups and pattern matching in image analysis},
International Journal of Computer Vision \textbf{28} (1998) 213--221.

\bibitem{You}
L.~Younes,
\emph{Shapes and Diffeomorphisms},
Springer-Verlag, Berlin, 2010.

\end{thebibliography}
\end{document}